\documentclass[a4paper,11pt]{amsart}
\usepackage[latin1]{inputenc}
\usepackage{amsmath}
\usepackage{amsfonts}
\usepackage{amssymb}
\usepackage{amsthm}
\usepackage{mathrsfs}

\newcommand{\R}{{\mathbb{R}}}

\newcommand{\Z}{{\mathbb{Z}}}
\newcommand{\C}{\mathbb{C}}

\newcommand{\F}{\mathcal{F}}

\renewcommand{\H}{\mathbb{H}}
\newcommand{\BB}{\widehat{B}}

\def\vecx{{\text{\boldmath$x$}}}
\def\vecy{{\text{\boldmath$y$}}}
\def\vecz{{\text{\boldmath$z$}}}
\def\vecu{{\text{\boldmath$u$}}}

\def\vecm{{\text{\boldmath$m$}}}

\def\vece{{\text{\boldmath$e$}}}
\def\vecM{{\text{\boldmath$M$}}}
\def\vecX{{\text{\boldmath$X$}}}
\def\vecal{{\text{\boldmath$\alpha$}}}
\def\vecgam{{\text{\boldmath$\gamma$}}}
\def\vecdel{{\text{\boldmath$\delta$}}}
\def\vecom{{\text{\boldmath$\omega$}}}
\def\vecmu{{\text{\boldmath$\mu$}}}
\def\vecnu{{\text{\boldmath$\nu$}}}
\def\vec0{{\text{\boldmath$0$}}}

\def\Re{\operatorname{Re}}

\newcommand{\B}{\mathfrak{B}}
\newcommand{\ve}{\varepsilon}

\newcommand{\sfrac}[2]{{\textstyle \frac {#1}{#2}}}
\newcommand{\qspace}{\Gamma \setminus \H^{n+1}}
\newcommand{\matr}[4]{\left( \begin{smallmatrix} #1 & #2 \\ #3 & #4
\end{smallmatrix} \right) }
\newcommand{\matri}[4]{\left( \begin{matrix} #1 & #2 \\ #3 & #4
\end{matrix} \right) }
\newcommand{\PSL}{\mathrm{PSL}}
\newcommand{\PSO}{\mathrm{PSO}}
\newcommand{\SL}{\mathrm{SL}}
\newcommand{\SO}{\mathrm{SO}}

\newcommand{\GL}{\mathrm{GL}}

\newcommand{\Gammainfty}{\Gamma_{\infty}}

\newtheorem{thm}{Theorem}[section]
\newtheorem{lem}[thm]{Lemma}
\newtheorem{prop}[thm]{Proposition}
\newtheorem{cor}[thm]{Corollary}
\newtheorem{defi}[thm]{Definition}

\theoremstyle{remark}
\newtheorem{remark}[thm]{Remark}

\numberwithin{equation}{section}

\begin{document}
%opening
\title[On the uniform equidistribution of closed horospheres]{On the uniform equidistribution of closed horospheres in hyperbolic manifolds}
\author{Anders S\"odergren}
\address{Department of Mathematics, Uppsala University, Box 480, 751 06 Uppsala, Sweden\newline
\rule[0ex]{0ex}{0ex}\hspace{8pt} {\tt sodergren@math.uu.se}\newline
\rule[0ex]{0ex}{0ex} \hspace{8pt}\textit{Present address:} 
School of Mathematics, Institute for Advanced Study, Einstein\newline 
\rule[0ex]{0ex}{0ex} \hspace{8pt}Drive, Princeton, NJ 08540, USA\newline
\rule[0ex]{0ex}{0ex} \hspace{8pt}{\tt sodergren@math.ias.edu}} 
\date{\today}

\maketitle

\begin{abstract}
We prove asymptotic equidistribution results for pieces of large closed
horospheres in cofinite hyperbolic manifolds of arbitrary dimension. This extends earlier results
by Hejhal \cite{dennis2} and Str\"ombergsson \cite{andreas} in dimension 2. Our proofs use spectral methods, and lead to precise estimates on the rate of convergence to equidistribution.
\end{abstract}

\tableofcontents

\section{Introduction}

Let $M$ be a non-compact hyperbolic manifold of finite volume. To present our problem we first consider the case when $M$ is two dimensional. Then the surface $M$ has a finite number of cusps, and to each cusp corresponds a one-parameter family of closed horocycle curves in $M$. In each family there exists a unique closed horocycle of any given length $\ell>0$, and it is known that as $\ell\to\infty$, the closed horocycle becomes asymptotically equidistributed on $M$ with respect to the hyperbolic area measure. Investigations related to this fact have been carried out by a number of people over the years, including Selberg (unpublished), Zagier \cite{zag}, Sarnak \cite{sarnak}, Hejhal \cite{dennis2}, Flaminio and Forni \cite{flam}, and
Strömbergsson \cite{andreas}. 

In \cite{dennis2}, Hejhal asked to what exact degree of \textit{uniformity} does this equidistribution hold? Specifically, given a \textit{subsegment} of length $\ell_1<\ell$ of a closed horocycle of length $\ell$, under what conditions on $\ell_1$ can we ensure that this subsegment becomes asymptotically equidistributed on $M$? Hejhal proved that this holds as long as we keep $\ell_1\geq\ell^{c+\ve}$ ($\ve>0$) as $\ell\to\infty$, where $c\geq\frac{2}{3}$ is a constant which only depends on the surface $M$. This was later improved by Str\"ombergsson \cite{andreas} to allowing $c=\frac{1}{2}$, independently of $M$. This constant is optimal. The equidistribution results both in \cite{dennis2} and \cite{andreas} were obtained with explicit rates. 

Our purpose in the present paper is to generalize these equidistribution results to the case when $M$ is a hyperbolic manifold of arbitrary dimension $n+1$. We realize $M$ as the quotient $M=\Gamma\setminus \H^{n+1}$, where $\H^{n+1}$ is the hyperbolic upper half space,
\begin{align*}
\H^{n+1}=\big\{P=(\vecx,y)\mid \vecx\in\R^{n},y\in\R_{>0}\big\}
\end{align*}
with Riemannian metric $ds^2=y^{-2}(dx_1^2+\ldots+dx_n^2+dy^2)$, and $\Gamma$ is a cofinite (but not cocompact) discrete subgroup of the group $G$ of orientation preserving isometries of $\H^{n+1}$. Without loss of generality we can assume that one of the cusps is placed at infinity. Then the fixator group $\Gamma_{\infty}\subset\Gamma$ contains a subgroup of finite index consisting of translations,
\begin{align*}
 \Gamma_{\infty}'=\Big\{(\vecx,y)\mapsto (\vecx+\vecom,y)\:\big|\:\vecom\in\Lambda\Big\}
\end{align*}
where $\Lambda$ is some lattice in $\R^n$. Let $\Omega=\{\vecom_1,\vecom_2,\ldots,\vecom_n\}$ be a basis of $\Lambda$. Now, for each $y>0$ and any $\alpha_i, \beta_i\in\R$ with $\alpha_i<\beta_i$, we set
\begin{align*}
 \B=\Big\{(u_1\vecom_1+\cdots+u_n\vecom_n,y)\mid u_i\in[\alpha_i,\beta_i]\:\textrm{for}\:i=1,\ldots,n\Big\}.
\end{align*}
This is a box-shaped subset of a closed horosphere in $M$. (Note, though, that if $\Gamma_{\infty}\neq\Gamma_{\infty}'$ then the map $\B\to M$ need not be injective, even if $\beta_i<\alpha_i+1$ for all $i$.) Our first main theorem says that as $y\to0$, the box $\B$ becomes asymptotically equidistributed in $M$ with respect to the hyperbolic volume measure $d\nu=y^{-n-1}dx_1\ldots dx_ndy$, so long as we keep all $\beta_i-\alpha_i\geq y^{1/2-\ve}$.

\begin{thm}\label{thmA}
Let $\Gamma$ be a cofinite discrete subgoup of $G$ such that $\Gamma\setminus \H^{n+1}$ has a cusp at infinity. Let $\ve>0$, and let $f$ be a fixed continuous function of compact support on $\Gamma\setminus \H^{n+1}$. Then
\begin{align}\label{thmAeq}
 \frac{1}{(\beta_1-\alpha_1)\cdots(\beta_n-\alpha_n)}\int_{\alpha_1}^{\beta_1}\cdots\int_{\alpha_n}^{\beta_n}f(u_1\vecom_1+\cdots+u_n\vecom_n,y)\,du_1\ldots
du_n\\
\to\frac{1}{\nu(\Gamma\setminus\H^{n+1})}\int_{\Gamma\setminus\H^{n+1}}f(P)\,d\nu(P),\nonumber
\end{align}
uniformly as $y\to0$ so long as $\beta_1-\alpha_1,\ldots,\beta_n-\alpha_n\geq y^{1/2-\ve}$.
\end{thm}

Note that when $n=1$ this specializes to the equidistribution result from \cite{andreas} described above, since $\ell\sim y^{-1}$ in this case. The exponent $\frac{1}{2}$ in Theorem \ref{thmA} is in fact the best possible in arbitrary dimension, if we restrict to considering boxes which have all side lengths comparable,
as $y\to0$. Indeed, if we keep \mbox{$\beta_1-\alpha_1=\ldots=\beta_n-\alpha_n=cy^{1/2}$,} say,
with a sufficiently small fixed constant $c>0$, then the box $\B$
can be placed in such a way that all of $\B$ stays far out in some
cuspidal region of $M$ as $y\to0$. Cf.\ Remark \ref{finalremark}.

We obtain all our equidistribution results with \textit{precise rates}, provided that $f$ is sufficiently smooth. In this direction we obtain the best results if instead of the box $\B$ we use a smooth cutoff function in the closed horosphere. Thus let us fix $\chi:\R^n\to \R$ to be a smooth function of compact support. Given $\delta_1,\ldots,\delta_n\in(0,1]$ and $\vecgam:=(\gamma_1,\ldots,\gamma_n)\in\R^n$ we define
\begin{align*}
\chi_{\vecdel,\vecgam}(\vecu):=\chi\big(\sfrac{u_1-\gamma_1}{\delta_1},\ldots,\sfrac{u_n-\gamma_n}{\delta_n}\big),
\end{align*}
and consider the following horosphere integral
\begin{align*}
 \frac{1}{\delta_1\cdots\delta_n}\int_{\R^n}\chi_{\vecdel,\vecgam}(\vecu)f(u_1\vecom_1+\cdots+u_n\vecom_n,y)\,d\vecu,
\end{align*}
where $d\vecu=du_1\ldots du_n$. Note that if we relax the condition that $\chi$ be smooth, and take $\chi$ to be the characteristic function of the unit cube $[-\frac{1}{2},\frac{1}{2}]^n$, then we get back the left hand side of \eqref{thmAeq} with $\alpha_i=\gamma_i-\delta_i/2$ and $\beta_i=\gamma_i+\delta_i/2$.

The precise rate of equidistribution which we obtain depends on the spectrum of the Laplace-Beltrami operator $\Delta$ on $M$. If there are small eigenvalues $0<\lambda<n^2/4$ of $-\Delta$ on $M$, then we take $\sigma_1\in(\frac{n}{2},n)$ so that $\lambda=\sigma_1(n-\sigma_1)$ is the smallest non-zero eigenvalue; otherwise, if there are no small eigenvalues, we set $\sigma_1=\frac{n}{2}$.

\begin{thm}\label{thmB}
Let $\Gamma$ be a cofinite discrete subgroup of $G$ such that
$\Gamma\setminus \H^{n+1}$ has a cusp at infinity. Let $\ve>0$,
let $f$ be a fixed $C^{\infty}$-function on $\H^{n+1}$ which is $\Gamma$-invariant and of compact support on
$\Gamma\setminus \H^{n+1}$, and let $\chi:\R^n\to \R$ be a smooth
function of compact support. Then, uniformly over all
$\vecgam=(\gamma_1,\ldots,\gamma_n)\in\R^n$ and all
$\delta_1,\ldots,\delta_n$ satisfying
$\sqrt{y}\leq\delta_1,\ldots,\delta_n\leq 1$,
 \begin{align}\label{errorterm}
  &\frac{1}{\delta_1\cdots\delta_n}\int_{\R^n}\chi_{\vecdel,\vecgam}(\vecu)f(u_1\vecom_1+\cdots+u_n\vecom_n,y)\,d\vecu\\
  &=\frac{\langle\chi\rangle}{\nu(\Gamma\setminus\H^{n+1})}\int_{\Gamma\setminus\H^{n+1}}f(P)\,d\nu(P)+O\Big(y^{-\ve}\big(y/\delta_{\min}^2\big)^{n-\sigma_1}\Big),\nonumber
 \end{align}
where $\sigma_1\in[\frac{n}{2},n)$ is as above, $\langle\chi\rangle=\int_{\R^n}\chi(\vecx)\,d\vecx$, $\delta_{\min}=\underset{i\in\{1,\ldots,n\}}{\min}\delta_i$, and the implied constant depends on $\Gamma$, $f$, $\chi$, $\ve$ and $\Omega$.
\end{thm}

We will actually prove a stronger version of Theorem \ref{thmB}, where the dependence on $f$ and $\chi$ is explicit, see Theorem \ref{chithm}. Moreover, Theorem \ref{thmA} follows from \mbox{Theorem~\ref{chithm}} by an approximation argument; in fact we obtain Theorem \ref{thmA} with the rate $O\Big(y^{-\ve}\big(y/\delta_{\min}^{2}\big)^{1-\sigma_1/n}\Big)$ and with an explicit dependence on $f$, cf.\ Theorem \ref{chi0thm}.

Our method of proof of these results is to use the spectral decomposition of the Laplace-Beltrami operator on $M$, involving the theory of Eisenstein series. To obtain our results we need to develop several bounds on sums over the Fourier coefficients of the individual eigenfunctions of $M$, which are also of independent interest. In particular we prove a precise Rankin-Selberg type bound (cf.\ Proposition \ref{FinalRS}) and bounds on coefficient sums twisted with an additive character (cf.\ Proposition \ref{Lemma1} and Theorem \ref{Lemma1'}).

Bounds of the latter type allow us to prove that in the
special case when the test function $f$ in Theorem \ref{thmB} is a
\textit{cusp form} (i.e.\ an eigenfunction of the Laplace-Beltrami operator
which decays exponentially in each cusp) with eigenvalue
$\lambda=s(n-s)$, $s\in(\frac{n}{2},n)$ or $s\in\frac{n}{2}+i\R_{\geq0}$, then the $\delta_j$'s can be allowed to shrink much more rapidly than before, and we obtain \eqref{errorterm} with an error term $O\Big(y^{-\ve}\big(y/\delta_{\min}\big)^{n-\Re
s}\Big)$, uniformly over $0<\delta_1,\ldots,\delta_n\leq1$. Cf.\ Proposition \ref{chicuspprop}.

\vspace{5pt}

It is interesting to compare the results in Theorems \ref{thmA} and \ref{thmB},
and in particular the explicit bounds obtained in Theorem \ref{thmB}, with
known facts related to the mixing property of the geodesic flow on
the unit tangent bundle $T_1M$ of $M$.
It is well known that in the special case of \textit{fixed}
$\delta_1,\ldots,\delta_n$ (say $\delta_1=\ldots=\delta_n=1$),
the equidistribution result in Theorem \ref{thmA} (and Theorem \ref{thmB} except for the precise rate)
can be obtained as a consequence of the fact that the geodesic flow
on $T_1M$ is mixing. Cf.\ Kleinbock and Margulis \cite[Prop.\ 2.4.8]{klein}
and also Eskin and McMullen \cite[Thm.\ 7.1]{em}.
Using the fact that the geodesic flow is known to be mixing with
an exponential rate (Moore, \cite{moore}), we can even obtain this
equidistribution result with a rate of type $O(y^{a_1})$ for some $a_1>0$.

To give a precise statement, first recall that the homogeneous space
$\Gamma\backslash G$ can be identified with the (oriented orthonormal) \textit{frame bundle}
over $M$, consisting of oriented
orthonormal bases $\vece_1,\ldots,\vece_{n+1}$ in the tangent space $T_pM$
above each point $p\in M$. There is a distinguished one-parameter
subgroup
$\{\varphi_t \: | \: t\in\R\}\subset G$ such that the flow
$\Gamma g \mapsto \Gamma g\varphi_t$ corresponds to unit speed parallel
transport
along the geodesic in the direction given by $\vece_1$.\footnote{Explicitly,
if we identify $\Gamma\backslash G$ with the frame bundle of $M$
in such a way that $\Gamma e$ corresponds to $(\vecx,y)=(\mathbf{0},1)\in
\H^{n+1}$ and a frame with $\vece_1$ pointing straight upwards,
and if we use $G=\PSL(2,C_n)$ as introduced in Section \ref{notation} below,
then $\varphi_t=\matr{e^{t/2}}00{e^{-t/2}}$;
if we use instead the notation in \cite[Ch.\ 4.5-6]{rat} then
$\varphi_t\in \PSO(n+1,1)$ is the matrix which has bottom right block equal to $\matri{\cosh t}{\sinh t}{\sinh t}{\cosh t}$, and all other entries as in the identity matrix.}
Clearly this flow descends to the geodesic flow on $T_1M$ under the
projection
$\Gamma\backslash G\ni(p,\vece_1,\ldots,\vece_{n+1})\mapsto(p,\vece_1)\in T_1M$.
Now using representation theoretic techniques
(cf.\ Hirai \cite{hir},
Knapp \cite[Ch.\ VIII.8]{knapp},
Moore \cite[Thm.\ 4.1]{moore}, Shalom \cite[Thm.\ 2.1, \S 8.1]{shalom}), one has the following
precise version of (optimal rate) exponential mixing for the flow $\{\varphi_t\}$:
For any fixed, sufficiently smooth and decaying functions $v,w\in L^2(\Gamma\backslash G)$, with either $v$ or $w$ being a lift of a function on $M$, we have
\begin{align} \label{EXPMIXING}
&\int_{\Gamma\backslash G} v(g\varphi_t)w(g)\,dg
=\int_{\Gamma\backslash G} v(g)\,dg \int_{\Gamma\backslash G} w(g)\,dg
+O_{v,w,\ve}\big(e^{(\sigma_1-n+\ve)t}\big)\hspace{7pt}\text{as }\: t\to\infty,
\end{align}
where $\sigma_1$ is as before, and where $dg$ is the Haar measure on $G$
normalised so that $\int_{\Gamma\backslash G} dg=1$.
Cf.\ also Ratner \cite{ratner1} in the 2-dimensional case and
Pollicott \cite{polli} in the 3-dimensional case.

See \cite[Prop 2.4.8]{klein} for how to use a mixing result as in \eqref{EXPMIXING} to
obtain
an equidistribution result in the situation of Theorem \ref{thmB} with
$\delta_1=\ldots=\delta_n=1$, with a rate of type $O(y^{a_1})$
($a_1>0$). Note however that this argument results in a
sacrifice
in the exponent; $a_1$ comes out significantly smaller than $n-\sigma_1-\ve$. By
contrast,
a nice feature of Theorem \ref{thmB} above is that the exponent is just as good as
what one would \textit{formally} obtain from \eqref{EXPMIXING} by taking
$v$ to be the appropriate \textit{measure} on $\Gamma\backslash G$
with compact support inside the closed horosphere at $y=1$, and taking $w$
to be the lift of $f$ to $\Gamma\backslash G$.

\vspace{5pt}

Let us now return to Theorem \ref{thmA}, i.e.\ let us again consider the box $\B$
in place of the smooth cutoff function $\chi$. It is an interesting
question to ask what explicit rate of convergence we can ideally hope to
obtain in this equidistribution result. As we have already mentioned, using
an approximation argument we are able to prove Theorem \ref{thmA} with an explicit
rate $O\Big(y^{-\ve}\big(y/\delta_{\min}^2\big)^{1-\sigma_1/n}\Big)$.
When $n=1$ this is basically the same rate as was obtained by Str\"ombergsson in
\cite[Thm.\ 3]{andreas}. However, for $n\geq 2$ we expect that it should be possible to obtain a
better error term by working more directly with the left hand side of \eqref{thmAeq}, applying the spectral expansion of $f$.
Indeed for $n=2$ we have found a proof along these lines of an explicit
error bound
$O\Big(y^{-\ve}\big(y/\delta_{\min}^{2}\big)^{2-\sigma_1}\Big)$
in Theorem \ref{thmA}, which is just as good as the bound in Theorem \ref{thmB}.
This proof, which will not be given in the present paper, is \textit{far}
from a direct adaptation of our proof of Theorem \ref{thmB}; in particular it
involves using bounds on sums of the Eisenstein series coefficients
twisted with an additive character (viz., an extension of Theorem \ref{Lemma1'} to the case of Eisenstein series) whereas in the proof
of Theorem \ref{thmB} we only need a Rankin-Selberg type bound on the Eisenstein series coefficients.
We have so far not been able to carry over this treatment to the
case $n\geq 3$ in a satisfactory way,
although certain partial results lead us to hope that it might
be possible to eventually obtain the bound
$O\Bigl(y^{-\ve} \big(y/\delta_{\min}^{2}\big)^{\min(1,n-\sigma_1)}\Bigr)$
in Theorem \ref{thmA}.

\vspace{5pt}

We note that it is also possible to prove Theorem \ref{thmA} using the classification of
all ergodic invariant measures for the horosphere action on
$\Gamma\backslash G$ given by Ratner in \cite{ratner2}. This approach
leads to a more general formulation of our result in that we obtain
asymptotic equidistribution on the frame bundle of
$M$. We furthermore note that it is possible to relax the conditions
on the function $f$ in Theorem \ref{thmA} to a condition on the growth of $f$ in each cusp. For
the details of the argument in the case $n=1$ see Str\"ombergsson \cite{andreas}.

We remark that another possible approach to the precise study of error
terms in our equidistribution results would be to try to extend the
representation theoretic method of Flaminio and Forni \cite{flam}, involving the classification of invariant distributions  of the horocycle flow, to the present case. If this could be carried out, it would have the benefit of
also giving results on the frame bundle of $M$.

\vspace{5pt}

Finally we mention that recently Kontorovich and Oh \cite{KoOh1}, \cite{KoOh2} have proved effective results on the equidistribution of expanding closed horospheres also in two and three dimensional hyperbolic manifolds of \emph{infinite} volume. In addition to being interesting in their own right, these results have beautiful applications in several number theoretic counting problems.

\section{Preliminaries}

\subsection{Basic set-up}\label{notation}

The upper half-space model $\H^{n+1}$ of the $(n+1)$-dimensional hyperbolic space consists of the set
\begin{align*}
 \H^{n+1}:=\big\{P=(x_1,x_2,\ldots,x_{n},y)\mid (x_1,x_2,\ldots,x_{n})\in\R^{n},y\in\R_{>0}\big\}
\end{align*}
equipped with the Riemannian metric
\begin{align}\label{metric}
 ds^2=\frac{dx_1^2+dx_2^2+\cdots+dx_{n}^2+dy^2}{y²}.
\end{align}
We will frequently use the notation $P=(\vecx,y)$ for points in $\H^{n+1}$, where $\vecx=(x_1,\ldots,x_{n})\in\R^{n}$. The hyperbolic volume element is given by
\begin{align*}
 d\nu=\frac{dx_1\ldots dx_{n}dy}{y^{n+1}}
\end{align*}
and the Laplace-Beltrami operator associated with \eqref{metric} is given by
\begin{align*}
 \Delta=y^2\Big(\frac{\partial^2}{\partial x_1^2}+\cdots+\frac{\partial^2}{\partial x_{n}^2}+\frac{\partial^2}{\partial y^2}\Big)-(n-1)y\frac{\partial}{\partial y}.
\end{align*}
It is well known that the group of orientation preserving isometries of $\H^{n+1}$, $\mathrm{Isom}^+\,\H^{n+1}$, is isomorphic to the classical Lie group $\PSO(n+1,1)$ (see e.g.\ \cite[Sec.\ 3.2 (Cor.\ 3)]{rat}). It is also possible to realize $\mathrm{Isom}^+\,\H^{n+1}$ as the group $\mathcal{M}(n)$ of orientation preserving M\"obius transformations of $\hat{\R}^{n}$, acting on $\H^{n+1}$ by their Poincaré extensions, with the topology of uniform convergence in the chordal metric on $\hat{\R}^{n}$ (cf.\ \cite[Cor.\ 1.10]{his} and \cite[Sec.\ 3.7]{bear}). These M\"obius transformations can be expressed in terms of $2\times2$ matrices whose entries are (certain) Clifford numbers. The details are as follows.

Consider the Clifford algebra $C_{n}$. This is the real associative algebra generated by $i_1,\ldots,i_{n-1}$ subject (only) to the relations $i_k^2=-1$ for $k=1,\ldots,n-1$ and $i_ki_h=-i_hi_k$ for $h\neq k$. The elements of $C_{n}$ are called \emph{Clifford numbers}. Each element $a\in C_{n}$ can be uniquely written as $a=\sum_{I}a_II$, where $a_I\in\R$ and the summation runs over all products $I=i_{\nu_1}\cdots i_{\nu_p}$ with $1\leq\nu_1<\cdots<\nu_p\leq n-1$.\footnote{Here the empty product is interpreted as the real number $1$.} Thus $C_{n}$ is a vector space of real dimension $2^{n-1}$. We equip $C_{n}$ with the square norm, i.e.\ if $a=\sum_{I}a_II$ then $|a|^2=\sum_Ia_I^2$. Furthermore there are three commonly used involutions of $C_{n}$. The first involution, $':C_{n}\rightarrow C_{n}$, is the algebra automorphism which replaces each $i_k$ with $-i_k$. The second involution, $^*:C_{n}\rightarrow C_{n}$, reverses the order of the factors in each $i_{\nu_1}\cdots i_{\nu_p}$ and determines an anti-automorphism of $C_{n}$. The final involution, $^-:C_{n}\rightarrow C_{n}$, is a combination of the previous two, i.e. for each $a\in C_{n}$ we have $\bar{a}=(a')^*=(a^*)'$.

Elements of $C_{n}$ of the form $\vecx=x_0+x_1i_1+\cdots+x_{n-1}i_{n-1}$ are called \emph{vectors}. They form an $n$-dimensional vector space which we naturally identify with $\R^n$. For vectors we have the identity $\vecx\bar{\vecx}=\bar{\vecx}\vecx=|\vecx|^2$, which implies that every non-zero vector is invertible in $C_{n}$. Since products of invertible elements are invertible, the non-zero vectors generate a multiplicative group, $\Gamma_n$, called the \emph{Clifford group}. It is proved in \cite{ahl} that $a\bar{a}=\bar{a}a=|a|^2$ holds for all $a\in\Gamma_n$ and also that if $a,b\in\Gamma_n$ then $|ab|=|a||b|$.

\begin{defi}\label{GLSL}
We define the following groups of matrices:
\begin{enumerate}
 \item[(i)]$\GL(2,C_n):=\bigg\{\matri{a}{b}{c}{d}\,\Big|\, a,b,c,d\in\Gamma_n\cup\{0\},$
 \begin{flushright}$ab^*, cd^*, c^*a, d^*b\in\R^n, ad^*-bc^*\in\R\setminus\{0\}\bigg\}.$
 \end{flushright}
 \item[(ii)]$\SL(2,C_n):=\bigg\{\matri{a}{b}{c}{d}\in\GL(2,C_n)\,\Big|\, ad^*-bc^*=1\bigg\}.$
 \item[(iii)]$\PSL(2,C_n):=\SL(2,C_n)/\{\pm I\}.$
\end{enumerate}
\end{defi}

A matrix $g\in\GL(2,C_n)$ acts on elements in $\hat\R^n$ according to the formula
\begin{align}\label{action}
 g\vecx=(a\vecx+b)(c\vecx+d)^{-1}.
\end{align}
It is clear that $C_n\subset C_{n+1}$ and that $\Gamma_n\subset\Gamma_{n+1}$. Hence it follows that $\GL(2,C_n)$ acts, by the formula \eqref{action}, also on $\hat\R^{n+1}$. It is shown in \cite{ahl} and \cite{wat} that $\GL(2,C_n)$ is precisely the set of $2\times2$ matrices, with Clifford numbers as entries, that induce bijective mappings $\hat\R^n\rightarrow\hat\R^n$ and $\hat\R^{n+1}\rightarrow\hat\R^{n+1}$. \cite{ahl} and \cite{wat} further show that $\PSL(2,C_n)$ acts on $\hat\R^n$ as the group of orientation preserving M\"obius transformations and on $\hat\R^{n+1}$ as the Poincaré extensions of these M\"obius transformations. According to the above this means that $\PSL(2,C_n)$ is isomorphic to $\mathrm{Isom}^+\,\H^{n+1}$. In coordinates the Poincaré extension of \eqref{action} to $\H^{n+1}$ looks like
\begin{align*}
P=\vecx+yi_{n}\mapsto g(P)=\vecx_{g(P)}+y_{g(P)}i_n
\end{align*}
where $\vecx_{g(P)}$ and $y_{g(P)}$ are given by
\begin{align}\label{xcoord}
\vecx_{g(P)}=\frac{a\bar c|P|^2+b\bar d+a\vecx\bar d+b\bar \vecx\bar c}{|cP+d|^2}=\frac{(a\vecx+b)(\overline{c\vecx+d})+a\bar cy^2}{|c\vecx+d|^2+|c|^2y^2}
\end{align}
and
\begin{align}\label{ycoord}
y_{g(P)}=\frac{(ad^*-bc^*)y}{|cP+d|^2}=\frac{(ad^*-bc^*)y}{|c\vecx+d|^2+|c|^2y^2},
\end{align}
respectively. 

\begin{defi}
 For $g=\matr{a}{b}{c}{d}$, with $a,b,c,d\in C_n$, the norm of $g$ is given by
\begin{align*}
 \|g\|=\big(|a|^2+|b|^2+|c|^2+|d|^2\big)^{1/2}.
\end{align*}
\end{defi}

We give $\GL(2,C_n)$ and $\SL(2,C_n)$ the topology induced by the metric $d(g,h)=\|g-h\|$. In the same way as in \cite[pp.\ 78-81]{bear}, but keeping in mind the noncommutativity of $C_n$, one can show that the two-fold cover $\Phi:\SL(2,C_n)\rightarrow\PSL(2,C_n)\equiv\mathcal{M}(n)$ induces the topology of uniform convergence in the chordal metric on $\hat\R^n$, i.e. the natural topology, on $\mathcal{M}(n)$.

\vspace{5pt}

Let $\Gamma\subset\PSL(2,C_n)$ be a discrete subgroup of finite covolume acting on $\H^{n+1}$. We further assume that the quotient manifold $\qspace$ has at least one cusp, i.e.\ that $\Gamma$ is not cocompact. Let $\eta_{1},...,\eta_{\kappa}$ be a maximal set of $\Gamma$-inequivalent cusps of $\Gamma$. For every $k\in\{1,...,\kappa\}$ we choose $A_{k}\in\PSL(2,C_n)$ such that $A_k(\eta_k)=\infty$. In particular we get $A_k^{}\Gamma_{\eta_k}A_k^{-1}=(A_k^{}\Gamma A_k^{-1})_{\infty}$, where $\Gamma_{\eta_k}$ denotes the stabilizer of $\eta_k$ in $\Gamma$. We now let $N(C_n)$ be the abelian subgroup $\big\{\matr{1}{\vecz}{0}{1}\:|\:\vecz\in\R^n\big\}$ of $\PSL(2,C_n)$ and define
\begin{align*}
 \Gamma_{\eta_k}':=\Gamma\cap A_k^{-1}N(C_n)A_k=\Gamma_{\eta_k}\cap A_k^{-1}N(C_n)A_k,
\end{align*}
for $k\in\{1,...,\kappa\}$. We note that $A_k^{}\Gamma_{\eta_k}'A_k^{-1}=(A_k^{}\Gamma A_k^{-1})_{\infty}'$ and recall that
\begin{align*}
(A_k^{}\Gamma A_k^{-1})_{\infty}'=\bigg\{\matri{1}{\vecom}{0}{1}\:\Big|\:\vecom\in\Lambda_k\bigg\}
\end{align*}
for some lattice $\Lambda_k$ in $\R^n$. By \cite[Thm.\ 5.5.5 and Thm.\ 5.4.3]{rat} we also know that $(A_k^{}\Gamma A_k^{-1})_{\infty}'$ has finite index in $(A_k^{}\Gamma A_k^{-1})_{\infty}$. Finally, we note that
\begin{multline*}
(A_k^{}\Gamma A_k^{-1})_{\infty}\subset\bigg\{\matri{a}{a\vecom}{0}{a^{*-1}}=\matri{a}{0}{0}{a^{*-1}}\matri{1}{\vecom}{0}{1}\:\Big|\: a\in\Gamma_n, |a|=1, \vecom\in\R^n\bigg\}.
\end{multline*} 
We recall here that $\big\{\matr{a}{0}{0}{a^{*-1}}\,|\,a\in\Gamma_n, |a|=1\big\}/\{\pm I\}$ is (in the natural way) identified with $\SO(n)$, cf.\ \cite[pp.\ 89-90]{wat}.

For each $k\in\{1,...,\kappa\}$, we choose a closed convex fundamental polyhedron $\mathcal{C}_k$ for the action of $(A_k^{}\Gamma A_k^{-1})_{\infty}$ on $\R^n$. We may (and will) assume that the cusp normalizing maps $A_k$ have been chosen in such a way that $|\mathcal{C}_k|=1$ for each $k\in\{1,...,\kappa\}$, where $|\mathcal{C}_k|$ denotes the Euclidean $n$-dimensional volume of the set $\mathcal{C}_k$. We define, for each $B>0$ and each $k\in\{1,...,\kappa\}$, the sets
\begin{align}\label{semicuspreg}
 \widetilde{\mathcal{F}}_k(B)=\big\{P=\vecx+yi_n\in\H^{n+1}\mid \vecx\in\mathcal{C}_k, y\geq B\big\}.
\end{align}
We also let 
\begin{align*}
\tau_k:=\min_{\vecom\in\Lambda_k\setminus\{\vec0\}}|\vecom|.
\end{align*} 
It follows from \cite[Prop.\ 3.3]{her1} (see also \cite[p.\ 49, Thm.\ 3.4]{egm}) that there exists a constant $B_0>\max(1,\max_{k\in\{1,...,\kappa\}}\tau_k)$ such that for all $B\geq B_0$, the cuspidal regions
\begin{align}\label{cuspreg}
 \mathcal{F}_k(B):=A_k^{-1}\big(\widetilde{\mathcal{F}}_k(B)\big),\qquad k\in\{1,...,\kappa\},
\end{align}
are disjoint and their union form part of a fundamental domain for $\Gamma$. In fact, for each $B\geq B_0$ we can choose a certain bounded set $\F_B$ so that
\begin{align*}
 \F:=\F_B\cup\bigcup_{k=1}^{\kappa} \mathcal{F}_k(B)
\end{align*}
is a fundamental domain for $\Gamma$ (see e.g. \cite[p.\ 51, Prop.\ 3.9]{egm} for the \mbox{$3$-dimensional} case). We fix the choice of $\F_B$ in such a way that $\F_B\cap\bigcup_{k=1}^{\kappa} \mathcal{F}_k(B)=\emptyset$ and $\F_B\subset\F_{B_0}\cup\bigcup_{k=1}^{\kappa} \mathcal{F}_k(B_0)$ for all $B\geq B_0$.

\subsection{Two elementary lemmas}

\begin{lem} \label{shim}
Let $k,\ell\in\{1,...,\kappa\}$ and $T=\matr{a}{b}{c}{d}\in A_k^{}\Gamma A_{\ell}^{-1}$ be given. Then $|c|\geq 1/\sqrt{\tau_k\tau_{\ell}}$, unless $k=\ell$ and $T\in A_k^{}\Gamma_{\eta_k}A_k^{-1}$.
\end{lem}

\begin{proof}
By definition $T=A_k^{}W A_{\ell}^{-1}$ with some $W\in\Gamma$. We let $\gamma=\matr{1}{\vecom}{0}{1}\in A_{\ell}^{}\Gamma_{\eta_{\ell}}' A_{\ell}^{-1}$, where $\vecom\in\Lambda_{\ell}\setminus{\{\vec0\}}$ is chosen such that $|\vecom|$ is minimal (i.e.\ $|\vecom|=\tau_{\ell}$), and consider $T\gamma T^{-1}$. We find that
\begin{align*}
 T\gamma T^{-1}=\matri{a}{b}{c}{d}\matri{1}{\vecom}{0}{1}\matri{d^*}{-b^*}{-c^*}{a^*}=\matri{1-a\vecom c^*}{a\vecom a^*}{-c\vecom c^*}{1+c\vecom a^*}
\end{align*}
and $T\gamma T^{-1}\in A_k^{}\Gamma A_k^{-1}$. Here $A_k^{}\Gamma A_k^{-1}$ is a discrete subgroup of $\PSL(2,C_n)$ with a cusp at $\infty$ having stabilizer $(A_k^{}\Gamma A_k^{-1})_{\infty}=A_k^{}\Gamma_{\eta_k}A_k^{-1}$. Hence, by Shimizu's lemma (as in \cite[Thm.\ A]{her2} and \cite[p.\ 471, Rem.\ 1]{her1}), we have either $|-c\vecom c^*|\geq1/\tau_k$ or $|-c\vecom c^*|=0$. Since $|c^*|=|c|$ this shows that either $|c|\geq 1/\sqrt{\tau_k\tau_{\ell}}$ or $T\gamma T^{-1}\in(A_k^{}\Gamma A_k^{-1})_{\infty}'$. When $T\gamma T^{-1}\in(A_k^{}\Gamma A_k^{-1})_{\infty}'$ it follows that $T^{-1}(\infty)$ is fixed by $\gamma$, which in turn implies that $W(\eta_{\ell})=\eta_k$. Since $\eta_{1},...,\eta_{\kappa}$ by construction are pairwise $\Gamma$-inequivalent the only possibility is that $k=\ell$ and $W\in\Gamma_{\eta_k}$, i.e. $T=A_k^{}W A_k^{-1}\in A_k\Gamma_{\eta_k} A_k^{-1}$.
\end{proof}

\begin{lem}\label{cuspineq}
For all $P\in\H^{n+1}$ and $W\in\Gamma$ we have
\begin{align*}
y_{A_k W(P)}\leq\frac{\tau_k\tau_{\ell}}{y_{A_{\ell}(P)}},
\end{align*}
unless $k=\ell$ and $W\in\Gamma_{\eta_k}$.
\end{lem}

\begin{proof}
Let $T=\matr{a}{b}{c}{d}=A_k^{}W A_{\ell}^{-1}$. It follows from Lemma \ref{shim} that either $k=\ell$ and
$W\in\Gamma_{\eta_k}$ or else $|c|\geq 1/\sqrt{\tau_k\tau_{\ell}}$. In the second case we use \eqref{ycoord} to get
\begin{align*}
 y_{A_kW(P)}=y_{T A_{\ell}(P)}=\frac{y_{A_{\ell}(P)}}{|c\vecx_{A_{\ell}(P)}+d|^2+|c|^2(y_{A_{\ell}(P)})^2}
 \leq\frac{y_{A_{\ell}(P)}}{|c|^2(y_{A_{\ell}(P)})^2}\leq\frac{\tau_k\tau_{\ell}}{y_{A_{\ell}(P)}}.
\end{align*}
\end{proof}

\subsection{The invariant height function}

We introduce the invariant height function $\mathcal{Y}_{\Gamma}$,
defined on $\H^{n+1}$ by
\begin{align}\label{invh}
\mathcal{Y}_{\Gamma}(P):=\max_{k\in\{1,...,\kappa\}}\max_{W\in\Gamma}
\textrm{  } y_{A_kW(P)}.
\end{align}
It is straightforward to show that this function is well-defined (i.e.\ that for every $P\in\H^{n+1}$ the maximum in \eqref{invh} is attained for some $k\in\{1,\ldots,\kappa\}$ and $W\in\Gamma$), continuous and $\Gamma$-invariant. Note also that $\mathcal{Y}_{\Gamma}(P)$ only depends on $\Gamma$ and $P$, and not on our specific choices of cusp representatives $\eta_1,\ldots,\eta_{\kappa}$ and normalizing maps $A_1,\ldots,A_{\kappa}$. (This fact makes crucial use of our assumption $|\mathcal C_k|=1$.) In the next couple of lemmas we collect some more basic properties of  $\mathcal{Y}_{\Gamma}$.

\begin{lem}\label{invhineq}
For all $B\geq B_0$, $k\in\{1,...,\kappa\}$ and $P\in\F_k(B)$ we have
\begin{align*}
\mathcal{Y}_{\Gamma}(P)=y_{A_k(P)}\geq B.
\end{align*}
\end{lem}

\begin{proof}
For $B\geq B_0$ and $P\in\F_k(B)$ we have by definition
$y_{A_k(P)}\geq B>\sqrt{\tau_k\tau_{\ell}}$. If $\ell=k$ and $W\in\Gamma_{\eta_k}$ then
$y_{A_{\ell}W(P)}=y_{A_k(P)}$. For all other pairs
$(\ell,W)\in\{1,...,\kappa\}\times\Gamma$ Lemma \ref{cuspineq} implies
\begin{align*}
y_{A_{\ell}W(P)}\leq \frac{\tau_k\tau_{\ell}}{y_{A_k(P)}}<\sqrt{\tau_k\tau_{\ell}}<B,
\end{align*}
and the lemma follows.
\end{proof}

\begin{lem}\label{invhbound}
The function $\mathcal{Y}_{\Gamma}$ satisfies
$\mathcal{Y}_{\Gamma}^0:=\inf_{P\in\H^{n+1}}\mathcal{Y}_{\Gamma}(P)>0$.
\end{lem}

\begin{proof}
By the $\Gamma$-invariance of $\mathcal{Y}_{\Gamma}$ and Lemma \ref{invhineq} it suffices to prove that
\linebreak$\inf_{P\in\F_{B_0}}\mathcal{Y}_{\Gamma}(P)>0$. But this follows from the continuity of $\mathcal{Y}_{\Gamma}$, since $\F_{B_0}$ is a bounded set. 
\end{proof}

\begin{lem}\label{estim}
Let $f$ be a $\Gamma$-invariant function satisfying, for some
constants $a\in\R$, $C_1,C_2\geq 0$ and for all $k\in\{1,...,\kappa\}$,
\begin{align}\label{cond}
|f(P)|\leq C_1 (y_{A_k(P)})^a+C_2
\end{align}
when $y_{A_k(P)}\geq\mathcal{Y}_{\Gamma}^0$. Then
\begin{align*}
|f(P)|\leq C_1 \big(\mathcal{Y}_{\Gamma}(P)\big)^a+C_2
\end{align*}
for all $P\in\H^{n+1}$.
\end{lem}

\begin{proof}
Take any $P\in\H^{n+1}$ and choose $(k,W)\in\{1,...,\kappa\}\times\Gamma$ such that
$y_{A_kW(P)}=\mathcal{Y}_{\Gamma}(P)\geq\mathcal{Y}_{\Gamma}^0$. By the $\Gamma$-invariance of
$f$ and condition
\eqref{cond} we get
\begin{align*}
|f(P)|=|f(W(P))|\leq C_1
(y_{A_kW(P)})^a+C_2=C_1\big(\mathcal{Y}_{\Gamma}(P)\big)^a+C_2.
\end{align*}
\end{proof}

Finally, we estimate the integral of the $n$:th power of $\mathcal{Y}_{\Gamma}$ over the bounded set $\F_B$.

\begin{lem}\label{integral}
For all $B\geq B_0$,
\begin{align*}
\int_{\F_B}\big(\mathcal{Y}_{\Gamma}(P)\big)^n\,d\nu(P)
=O\big(\log(2B)\big),
\end{align*}
where the implied constant depends only on $\Gamma$.
\end{lem}

\begin{proof}
We write
\begin{align*}
\F_B=\F_{B_0}\cup\bigcup_{k=1}^{\kappa}\big(\F_k(B_0)\cap\F_B\big)
\end{align*}
and note that
\begin{align*}
\F_k(B_0)\cap\F_B=\F_k(B_0)\setminus\F_k(B)=A_k^{-1}\big(\mathcal{C}_k\times[B_0,B)\big).
\end{align*}
Using the continuity of $\big(\mathcal{Y}_{\Gamma}(P)\big)^n$ and Lemma \ref{invhineq} we get,
\begin{align*}
&\int_{\F_B}\big(\mathcal{Y}_{\Gamma}(P)\big)^n\,
d\nu(P)=O(1)+\sum_{k=1}^{\kappa}\int_{\F_k(B_0)\cap\F_B}\big(\mathcal{Y}_{\Gamma}(P)\big)^n\,
d\nu(P)\\
&=O(1)+\sum_{k=1}^\kappa\int_{\mathcal{C}_k\times[B_0,B)}y^n\,\frac{dx_1\ldots dx_{n}dy}{y^{n+1}}
= O(1)+\kappa\int_{B_0}^B
\frac{dy}{y}=O\big(\log(2B)\big),
\end{align*}
since $B\geq B_0>1$.
\end{proof}

\subsection{Repeated summation by parts}

We will use summation by parts extensively in later sections. Here we fix the notation and give a convenient formula for repeated summation by parts.

We are interested to perform summation by parts on sums of the form
\begin{align*}
 \sum_{m_1=\alpha_1}^{\beta_1}\cdots\sum_{m_n=\alpha_n}^{\beta_n}g(m_1,\ldots,m_n)a(m_1,\ldots,m_n),
\end{align*}
where $\alpha_j,\beta_j\in\Z$ satisfy $0\leq \alpha_j\leq\beta_j$ for $j\in\{1,\ldots,n\}$, $a:\Z_{\geq0}^n\to\C$ and $g$ is a smooth function. We define
\begin{align*}
S(X_1,\ldots,X_n):=\sum_{0\leq m_1\leq X_1}\cdots\sum_{0\leq m_n\leq X_n}a(m_1,\ldots,m_n).
\end{align*}
We let $N:=\{1,\ldots,n\}$. Given any $A=\{i_1,\ldots,i_{|A|}\}\subset N$, with $i_1<\cdots<i_{|A|}$, and any $B\subset N\setminus A$ we define the embeddings $I_{A,B}:\R^{|A|}\to\R^n$ and $\tilde{I}_{A,B}:\R^{|A|}\to\R^n$ by
\begin{align*}
 I_{A,B}(\vecx)=I_{A,B}(x_1,\ldots,x_{|A|})=(x_1',\ldots,x_n'), \hspace{5pt}\textrm{where}\hspace{5pt}
x_j'=\begin{cases}
    \alpha_j & \textrm{ if $j\notin A\cup B$}\\
    \beta_j & \textrm{ if $j\in B$}\\
    x_{\ell} & \textrm{ if $j=i_{\ell}\in A$}
    \end{cases}
\end{align*}
and
\begin{align*}
 \tilde{I}_{A,B}(\vecx)=(x_1',\ldots,x_n'), \hspace{5pt}\textrm{where}\hspace{5pt}
x_j'=\begin{cases}
    \alpha_j-1 & \textrm{ if $j\notin A\cup B$}\\
    \beta_j & \textrm{ if $j\in B$}\\
    x_{\ell} & \textrm{ if $j=i_{\ell}\in A$.}
    \end{cases}
\end{align*}
We also introduce the notation
\begin{align}\label{g}
 g_{A,B}(\vecx):=\frac{\partial^{|A|}(g\circ I_{A,B})}{\partial x_1\cdots\partial x_{|A|}}(\vecx)
\end{align}
and
\begin{align}\label{S}
 S_{A,B}(\vecx):=S\circ \tilde{I}_{A,B}(\vecx).
\end{align}

\begin{lem}\label{sumbp}
Let $\alpha_j,\beta_j\in\Z$ satisfy $0\leq \alpha_j\leq\beta_j$. Let $g:\prod_{j\in N}[\alpha_j,\beta_j]\to\C$ be a smooth function and let $a:\Z_{\geq0}^n\to\C$. Then
\begin{multline}\label{intbp}
\sum_{m_1=\alpha_1}^{\beta_1}\cdots\sum_{m_n=\alpha_n}^{\beta_n}g(m_1,\ldots,m_n)a(m_1,\ldots,m_n)\\
=(-1)^n\underset{A\subset N}{\sum}\underset{B\subset N\setminus A}{\sum}(-1)^{|B|}\underset{\prod_{j\in A}[\alpha_j,\beta_j]}{\int} g_{A,B}(\vecx)S_{A,B}(\vecx)\,d\vecx,
\end{multline}
where we interpret the integral to be equal to $g_{\emptyset,B}S_{\emptyset,B}$ whenever $A=\emptyset$.
\end{lem}

\begin{proof}
The proof is by a straightforward induction on $n$. (Note that when $n=1$ this is the well-known summation by parts formula.)
\end{proof}

\begin{cor}\label{corsumbp}
Let $\beta_j\in\Z_{\geq 0}$. Let $a:\Z_{\geq0}^n\to\C$ and let \mbox{$g:\prod_{j\in N}[0,\beta_j]\setminus\{\vec0\}\to\C$} be a smooth function. Suppose that $S(0,\ldots,0)=0$, i.e.\ that $a(0,\ldots,0)=0$, and define $g(\vec0)$ arbitrarily. Then
\begin{multline}\label{again}
\sum_{m_1=0}^{\beta_1}\cdots\sum_{m_n=0}^{\beta_n}g(m_1,\ldots,m_n)a(m_1,\ldots,m_n)\\
=\underset{A\subset N}{\sum}(-1)^{|A|}\underset{\prod_{j\in A}[0,\beta_j]}{\int} g_{A,N\setminus A}(\vecx)S_{A,N\setminus A}(\vecx)\,d\vecx.
\end{multline}
\end{cor}

\begin{proof}
We change $g$ in $\prod_{j\in N}[0,\beta_j]\cap\{\vecy\mid |\vecy|\leq 1/2\}$ in such a way that $g$ becomes smooth in $\prod_{j\in N}[0,\beta_j]$. We call the new function $\tilde g$. With $g$ replaced by $\tilde g$, \eqref{again} is just a restatement of \eqref{intbp}, for $\alpha_1=\ldots=\alpha_n=0$ implies that $S_{A,B}(\vecx)\equiv0$ whenever $A\cup B\neq N$. Finally we note that both sides of \eqref{again} are unchanged when $g$ is replaced by $\tilde g$ since $a(0,\ldots,0)=0$.
\end{proof}

\begin{remark}\label{remsumbp}
We will be interested in using formula \eqref{intbp} (and the related formula \eqref{again}) also when some of the $\beta_j$ are infinite. For formula \eqref{intbp} to hold also in this case we need to care about convergence issues. However, in each case where we are interested in applying the extended version of formula \eqref{intbp} it is immediate to verify that the summation formula still holds. In particular all terms involved will be convergent and every boundary term at infinity will be zero.
\end{remark}

We let $C\subset N$ be defined by $C:=\{j\in N\mid \beta_j=\infty\}$. In the situation described by Remark \ref{remsumbp}, formula \eqref{intbp} turns into
\begin{multline}\label{intbp*}
\sum_{m_1=\alpha_1}^{\beta_1}\cdots\sum_{m_n=\alpha_n}^{\beta_n}g(m_1,\ldots,m_n)a(m_1,\ldots,m_n)\\
=(-1)^n\underset{A\subset N}{\sum}\underset{B\subset N\setminus (A\cup C)}{\sum}(-1)^{|B|}\underset{\prod_{j\in A}[\alpha_j,\beta_j]}{\int} g_{A,B}(\vecx)S_{A,B}(\vecx)\,d\vecx.
\end{multline}

\section{Basic counting bounds in hyperbolic geometry}

In this section we give estimates of various counting functions, which we will need in later sections. We begin with two elementary lemmas concerning the geometry of $\R^n$.

\begin{lem}\label{circlelemma}
 Given any $R,r>0$ and $\vecx,\vecu_1,\vecu_2,...,\vecu_m\in\R^n$ such that $|\vecx-\vecu_j|\leq R$ and $|\vecu_i-\vecu_j|\geq r$ for all $i,j\in\{1,...,m\}$, $i\neq j$, we have
\begin{align}\label{circleproblem}
 m\leq 1+k_0\Big(\frac{R}{r}\Big)^n,
 \end{align}
where $k_0>0$ is a constant which only depends on $n$.
\end{lem}

\begin{proof}
The open balls with radii $\sfrac{r}{2}$ and centers $\vecu_1,\vecu_2,\ldots,\vecu_m$ are pairwise disjoint and contained in the ball with center $\vecx$ and radius $R+\sfrac{r}{2}$. Hence
\begin{align}\label{circle}
 mV_n\Big(\frac{r}{2}\Big)^n\leq V_n\Big(R+\frac{r}{2}\Big)^n,
\end{align}
where $V_n$ is the volume of the unit ball in $\R^n$. If $r\leq 2R$ this implies that $m\ll_n\big(\frac{R}{r}\big)^n$. In the remaining case, $r>2R$, we clearly have $m\leq1$.
\end{proof}

We let $\mathfrak{B}$ be a right-angled closed box in $\R^n$, having side lengths $b_1,\ldots,b_n$.

\begin{lem}\label{boxlemma}
 Given $r>0$ and any $\vecu_1,\vecu_2,...,\vecu_m\in\mathfrak{B}$ such that $|\vecu_i-\vecu_j|\geq r$ for all $i,j\in\{1,...,m\}$, $i\neq j$, we have
\begin{align*}
 m\leq k_1\prod_{i=1}^n\Big(1+\frac{b_i}{r}\Big),
\end{align*}
where $k_1>0$ is a constant which only depends on $n$.
\end{lem}

\begin{proof}
Arguing as in the proof of Lemma \ref{circlelemma} we find that the inequality corresponding to \eqref{circle} is
\begin{align*}
 mV_n\Big(\frac{r}{2}\Big)^n\leq \prod_{i=1}^n\big(b_i+r\big),
\end{align*}
which implies the desired inequality.
\end{proof}

Our next task is a generalization of \cite[Lemma 2.10]{iwa} (cf.\ Lemma \ref{counting} below).

\begin{lem}\label{counting2}
 Assume that $\Gamma$ has a cusp at $\infty$. Then there exists a constant $k_2>0$, depending only on $\Gamma$, such that the following holds for all $R,L>0$ and $P=\vecx+yi_n\in\H^{n+1}$ with $y\geq R$:
 \begin{align}\label{counting1}
  \#\Big\{T=\matr{a}{b}{c}{d}\in\Gammainfty\backslash\Gamma\,\,\big|\,\, y_{T(P)}>R\textrm{    and    } L\leq |c|<2L\Big\}\leq k_2L^n\Big(\frac{y}{R}\Big)^{n/2}.
 \end{align}
\end{lem}

Let $M$ denote the set in the left hand side of \eqref{counting1}. Before we turn to the proof of the lemma we note that $M$ is well-defined. If the elements $\matr{a}{b}{c}{d},\matr{\alpha}{\beta}{\gamma}{\delta}$ belong to the same $\Gammainfty$-coset in $\Gamma$ it follows that
\begin{align*}
 \matri{a}{b}{c}{d}=\matri{e}{f}{0}{e^{*-1}}\matri{\alpha}{\beta}{\gamma}{\delta}=\matri{*}{*}{e^{*-1}\gamma}{*},
\end{align*}
where $\matr{e}{f}{0}{e^{*-1}}\in\Gammainfty$ (recall that $|e|=1$). Hence $|c|=|e^{*-1}\gamma|=|\gamma|$.\\

\begin{proof}[Proof of Lemma \ref{counting2}]
If the elements $\matr{a}{b}{c}{d},\matr{\alpha}{\beta}{\gamma}{\delta}$ belong to different $\Gammainfty$-cosets in $\Gamma$ we note that
\begin{multline*}
 \matri{\alpha}{\beta}{\gamma}{\delta}\matri{a}{b}{c}{d}^{-1}=\matri{\alpha}{\beta}{\gamma}{\delta}\matri{d^*}{-b^*}{-c^*}{a^*}\\=\matri{\alpha d^*-\beta c^*}{-\alpha b^*+\beta a^*}{\gamma d^*-\delta c^*}{-\gamma b^*+\delta a^*}\in\Gamma-\Gammainfty.
\end{multline*}
It follows from Shimizu's lemma (see \cite[Thm.\ A]{her2} and \cite[p.\ 471, Rem.\ 1]{her1}) that there exists a constant $\widehat{k}>0$, depending only on $\Gamma$, such that $|\gamma d^*-\delta c^*|\geq \widehat{k}$. If furthermore $\matr{a}{b}{c}{d},\matr{\alpha}{\beta}{\gamma}{\delta}$ are representatives of elements in $M$ we get
\begin{align}\label{del1}
 \big|c^{-1}d-\gamma^{-1}\delta\big|=\big|\gamma^{-1}(\gamma d^*-\delta c^*)c^{*-1}\big|=\frac{|\gamma d^*-\delta c^*|}{|\gamma||c|}\geq \frac{\widehat{k}}{|\gamma||c|}>\frac{\widehat{k}}{4L^2},
\end{align}
where we have used the fact that $(c^{-1}d)^*=c^{-1}d\in\R^n$ since $cd^*\in\R^n$ (cf.\ the proof of \cite[Lemma 1.4]{ahl}).
 
On the other hand, for each $T\in M$
we have $y_{T(P)}>R$, which by \eqref{ycoord} can be written as
\begin{align}\label{again2}
  \frac{y}{|c\vecx+d|^2+|c|^2y^2}>R \iff |c\vecx+d|^2+|c|^2y^2<\frac{y}{R}.
\end{align}
From \eqref{again2} we get $|c\vecx+d|<\sqrt{\frac{y}{R}}$\,, which implies that
\begin{align}\label{del2}
|\vecx+c^{-1}d|<\frac{1}{|c|}\sqrt{\frac{y}{R}}\leq \frac{1}{L}\sqrt{\frac{y}{R}}\,.                            
\end{align}
Using \eqref{del1} and \eqref{del2} together with Lemma \ref{circlelemma} we obtain
\begin{align}\label{calc}
 \#M
 \leq 1+k_0\Big(\frac{4L}{\widehat{k}}\Big)^n\Big(\frac{y}{R}\Big)^{n/2}.
\end{align}
This concludes the proof, for note that if $2L<\widehat{k}$ then $M$ is empty.
\end{proof}

\begin{lem}\label{counting}
 Assume that $\Gamma$ has a cusp at $\infty$. Then there exists a constant $k_3>0$, depending only on $\Gamma$, such that the following holds for all $R>0$ and $P=\vecx+yi_n\in\H^{n+1}$:
 \begin{align}\label{countingresult}
  \#\Big\{T\in\Gammainfty\backslash\Gamma \,\,\big|\,\, y_{T(P)}>R\Big\}\leq1+\frac{k_3}{R^n}.
 \end{align}
\end{lem}

\begin{proof}
To begin with we note that the left hand side in the inequality above is invariant if the point $P$ is changed to $T(P)$ for any $T\in\Gamma$. Hence we may assume that $y_{T(\vecx+yi_n)}\leq y$ for all $T\in\Gamma$. If $y\leq R$ then the left side of \eqref{countingresult} equals zero and the inequality therefore holds trivially. Therefore we may also assume that $y>R$.

If $T=\matr{a}{b}{c}{d}\in\Gamma$ satisfies $y_{T(P)}>R$ then, using \eqref{again2}, we get $|c|^2y^2<\frac{y}{R}$ and it follows that $|c|<\frac{1}{\sqrt{yR}}$. We recall that $c=0$ if and only if $T\in\Gammainfty$. Using this fact together with Lemma \ref{counting2} we obtain
\begin{align*}
 &\#\Big\{T\in\Gammainfty\backslash\Gamma \,\,\big|\,\, y_{T(P)}>R\Big\}\\
 &\leq 1+\sum_{j=1}^{\infty}\#\bigg\{T=\matr{a}{b}{c}{d}\in\Gammainfty\backslash\Gamma\,\,\Big|\,\,y_{T(P)}>R\textrm{    and    }\frac{2^{-j}}{\sqrt{yR}}\leq |c|<\frac{2^{1-j}}{\sqrt{yR}}\bigg\}\\
 &\leq1+\sum_{j=1}^{\infty}k_2\Big(\frac{2^{-j}}{\sqrt{yR}}\Big)^n\Big(\frac{y}{R}\Big)^{n/2}=1+\frac{k_2}{(2^n-1)R^n},
\end{align*}
 where $k_2$ is as in Lemma \ref{counting2}.
 \end{proof}

We next consider the counting function
\begin{align*}
\mathfrak{C}_{\mathfrak B}^k(X):=\#\Big\{W=A_k^{-1}\matr{a}{b}{c}{d}\in\Gamma_{\eta_k}\backslash\Gamma \,\,\Big|\,\, 0<|c|\leq X, -c^{-1}d\in\mathfrak{B}\Big\},
\end{align*}
where $\mathfrak B$ is the box introduced just above Lemma \ref{boxlemma}. We write $S$ for the set in the right hand side. By an argument similar to the discussion of \eqref{counting1} we find that $S$ is well-defined.

\begin{lem}\label{counting3}
 For any $k\in\{1,...,\kappa\}$ and $X>0$ we have
\begin{align*}
\mathfrak{C}_{\mathfrak B}^k(X)\leq k_4\prod_{i=1}^n\big(1+b_iX^2\big),
\end{align*}
where $k_4>0$ is a constant which only depends on $\Gamma$.
\end{lem}

\begin{proof}
 If $W_0=A_k^{-1}\matr{\alpha}{\beta}{\gamma}{\delta}$ and $W_1=A_k^{-1}\matr{a}{b}{c}{d}$ are any two distinct elements of $S$, then $W_0W_1^{-1}\notin\Gamma_{\eta_k}$. Furthermore it is clear that $\matr{\alpha}{\beta}{\gamma}{\delta}\matr{a}{b}{c}{d}^{-1}=A_kW_0W_1^{-1}A_k^{-1}\in A_k\Gamma A_k^{-1}$. Hence Lemma \ref{shim} gives $|\gamma d^*-\delta c^*|\geq1/\tau_k$. As in \eqref{del1}, since $0<|\gamma|,|c|\leq X$, we also get that
\begin{align*}
 \big|c^{-1}d-\gamma^{-1}\delta\big|=\frac{|\gamma d^*-\delta c^*|}{|\gamma||c|}\geq\frac{1}{\tau_k |\gamma||c|}\geq\frac{1}{\tau_kX^2}.
\end{align*}
Using Lemma \ref{boxlemma} with $r=1/(\tau_kX^2)$ we get 
\begin{align*}
 \mathfrak{C}_{\mathfrak{B}}^k(X)\leq k_1\prod_{i=1}^n\big(1+b_i\tau_kX^2\big),
\end{align*}
and the lemma follows.
\end{proof}

\begin{cor}\label{corc}
Let $b_0=0$ and $b_{n+1}=\infty$ and assume that $\mathfrak B$ is such that $0<b_1\leq\ldots\leq b_n$. Let $0<C\leq D$.
If $\frac{C}{\sqrt{b_{j+1}}}\leq X \leq\frac{D}{\sqrt{b_{j}}}$ for some $j\in\{0,\ldots,n\}$ (here $D/\sqrt{b_0}$ is interpreted as infinity), then
\begin{align*}
\mathfrak{C}_{\mathfrak B}^k(X)\leq k_5X^{2(n-j)}\prod_{i=j+1}^nb_i,
\end{align*}
where $k_5>0$ depends only on $C$, $D$ and $\Gamma$.
\end{cor}

\begin{proof}
Using Lemma \ref{counting3} and the bounds on $X$ we find that
\begin{align*}
\mathfrak{C}_{\mathfrak B}^k(X)\leq k_4\prod_{i=1}^j(1+D^2)\prod_{i=j+1}^n\Big(\big(1+1/C^2\big)b_iX^2\Big)
\leq k_5X^{2(n-j)}\prod_{i=j+1}^nb_i,
\end{align*}
where $k_5=k_4\underset{j\in\{0,\ldots,n\}}{\max}\Big(\prod_{i=1}^j(1+D^2)\prod_{i=j+1}^n\big(1+1/C^2\big)\Big)$.
 \end{proof}

We now prove a lemma that will be very useful when dealing with non-cuspidal eigenfunctions in Section \ref{noncuspchap}.

\begin{lem}\label{csum}
Let $k\in\{1,\ldots,\kappa\}$. Assume that $\mathfrak B$ is such that $0=b_0<b_1\leq\ldots\leq b_n<b_{n+1}=\infty$. Let $0<C\leq D$ and let $0<A\leq B$ be such that $[A,B]\subset\Big[\frac{C}{\sqrt{b_{j+1}}},\frac{D}{\sqrt{b_{j}}}\Big]$ for some $j\in\{0,\ldots,n\}$. Let $\sum_{W}$ denote the sum over a set of representatives $W=A_k^{-1}\matr{a}{b}{c}{d}\in\Gamma_{\eta_k}\backslash\Gamma$ restricted by $A\leq|c|\leq B$ and $-c^{-1}d\in\mathfrak{B}$. Then the following holds:
\begin{align*}
  &(i)\,\,\sum_{W}\frac{1}{|c|^{\nu}}\leq k_5\frac{2(n-j)}{2(n-j)-\nu}B^{2(n-j)-\nu}\prod_{i=j+1}^nb_i\hspace{15pt}\text{ for any } 0<\nu<2(n-j),\\
  &(ii)\,\,\sum_{W}\frac{1}{|c|^{\nu}}\leq k_5\frac{\nu}{\nu-2(n-j)}A^{2(n-j)-\nu}\prod_{i=j+1}^nb_i\hspace{15pt}\text{ for any } \nu>2(n-j).
\end{align*}
Here $k_5=k_5(\Gamma,C,D)$ is the constant from Corollary \ref{corc}. The result in $(ii)$ holds also with $j=0$ and $B=\infty$.
\end{lem}

\begin{proof}
We present the proof of $(i)$. The proof of $(ii)$ is entirely similar. Using Corollary \ref{corc} we get that for any $\gamma\in(0,A)$
the following holds:
\begin{align*}
\sum_{W}\frac{1}{|c|^{\nu}}&\leq\int_{A-\gamma}^{B}\frac{1}{X^{\nu}}\,d\mathfrak{C}_{\mathfrak{B}}^k(X)=\Big[\frac{1}{X^{\nu}}\mathfrak{C}_{\mathfrak{B}}^k(X)\Big]_{A-\gamma}^B+\nu\int_{A-\gamma}^{B}\frac{\mathfrak{C}_{\mathfrak{B}}^k(X)}{X^{\nu+1}}\,dX\\
&\leq k_5B^{2(n-j)-\nu}\Big(\prod_{i=j+1}^nb_i\Big)+\nu\int_{A-\gamma}^{A}\frac{\mathfrak{C}_{\mathfrak{B}}^k(X)}{X^{\nu+1}}\,dX\\
&+k_5\frac{\nu}{2(n-j)-\nu}\Big(\prod_{i=j+1}^nb_i\Big)\Big(B^{2(n-j)-\nu}-A^{2(n-j)-\nu}\Big)\\
&\leq k_5\frac{2(n-j)}{2(n-j)-\nu}B^{2(n-j)-\nu}\Big(\prod_{i=j+1}^nb_i\Big)+\nu\int_{A-\gamma}^{A}\frac{\mathfrak{C}_{\mathfrak{B}}^k(X)}{X^{\nu+1}}\,dX.
\end{align*}
Finally we let $\gamma\to0$ which gives the desired inequality.
\end{proof}

\section{Spectral theory}

\subsection{Spectral decomposition of $L^2(\Gamma\setminus\H^{n+1})$}

Let $\Delta$ denote the Laplace-Beltrami operator on $\Gamma\setminus\H^{n+1}$ and let $\phi_0,\phi_1,\phi_2,...$ be the $L^2$-eigenfunctions of $-\Delta$. We take these to be orthonormal and ordered with increasing eigenvalues $0=\lambda_0<\lambda_1\leq\lambda_2\leq...$. In general, we do not know if the set $\{\phi_0,\phi_1,\phi_2,...\}$ is infinite or not.

Each $\phi_m$ is smooth, and by mimicking \cite[pp.\ 23-26 (Prop.\ 4.10, Prop.\ 4.12)]{dennis}, \cite[pp.\ 105-107 (Thm.\ 3.1, Thm.\ 3.2)]{egm} one finds that it has a Fourier expansion at the cusp $\eta_{\ell}$ of the form
\begin{align*}
  \phi_m\big(A_{\ell}^{-1}(\vecx+yi_n)\big)=c_{\vec0}y^{n-s_m}+\sum_{\vec0\neq\vecmu\in\Lambda_{\ell}^*}c_{\vecmu}y^{n/2}K_{s_m-n/2}(2\pi|\vecmu|y)e^{2\pi i\langle\vecmu,\vecx\rangle},
 \end{align*}
where $s_m\in[n/2,n]\cup [n/2,n/2+i\infty)$ is given by $\lambda_m=s_m(n-s_m)$ and where $\Lambda_{\ell}^*$ denotes the dual lattice of $\Lambda_{\ell}$ defined by
\begin{align*}
 \Lambda_{\ell}^*=\big\{\vecmu\mid\langle\vecmu,\vecgam\rangle\in\Z\textrm{  for all   }\vecgam\in\Lambda_{\ell}\big\}.
\end{align*}
Here the coefficients $c_{\vecmu}$ depend on $\ell$ (as well as on $m$), and if we find it necessary to specify this dependence we will write $c_{\vecmu}^{(\ell)}$. Note that $c_{\vec0}^{(\ell)}=0$ for all $\ell\in\{1,\ldots,\kappa\}$ if and only if $\phi_m$ is a cusp form, and this is always the case when $\lambda_m\geq(\frac{n}{2})^2$.

Next, for each $k\in\{1,\ldots,\kappa\}$, $P\in\H^{n+1}$ and $s\in\C$ with $\Re s>n$, the \emph{Eisenstein series} $E_k(P,s)$ is defined by the absolutely convergent sum
\begin{align*}
 E_{k}(P,s):=\sum_{M\in\Gamma_{\eta_k}\setminus\Gamma}(y_{A_kM(P)})^{s}.
\end{align*}
(One easily checks that the right hand side does not depend on the choice of admissible $A_k$.)
This function has a meromorphic continuation to all $s\in\C$, and
\begin{align*}
 &E_k(W(P),s)=E_k(P,s), \qquad \forall W\in\Gamma,\,\forall P\in\H^{n+1};\\
 &E_k(P,s) \,\text{ is $C^{\infty}$ on } \H^{n+1}\times(\C-\{\text{poles}\});\\
 &\big(\Delta+s(n-s)\big)E_{k}(P,s)=0 \,\text{ on } \H^{n+1}\times(\C-\{\text{poles}\}).
\end{align*}
Furthermore, all poles of $E_k(P,s)$ in $\Re s\geq \frac{n}{2}$ are restricted to the segment $s\in(\frac{n}{2},n]$; in particular there are no poles for $\Re s=\frac{n}{2}$. Cf.\ \cite[Ch.\ 6]{cs}, and also \cite[\S 7.27]{cs}.

For $s\neq\text{pole}$, the Eisenstein series $E_{k}(P,s)$ has a Fourier expansion at the cusp $\eta_{\ell}$ of the form
\begin{multline*}
E_{k}\big(A_{\ell}^{-1}(\vecx+yi_n),s\big)=\delta_{k\ell}y^{s}+\varphi_{k\ell}(s)y^{n-s}\\
+\sum_{\vec0\neq\vecmu\in\Lambda_{\ell}^*}a_{\vecmu}(s)y^{n/2}K_{s-n/2}(2\pi|\vecmu|y)e^{2\pi i\langle\vecmu,\vecx\rangle}.
\end{multline*}
(Cf.\ \cite[6.1.42]{cs}, with a misprint corrected.) Of course the coefficients $a_{\vecmu}$ depend on $k$ and $\ell$ but this will be suppressed in the notation.
We collect the meromorphic functions $\varphi_{k\ell}$ in a matrix $\Phi(s)=\big(\varphi_{k\ell}(s)\big)$ called the \emph{scattering matrix} for $\Gamma$ (cf.\ \cite[\S 6.1.55]{cs}). $\Phi(s)$ is symmetric and for generic $s\in\C$ we have
\begin{align}\label{firelation}
 \Phi(s)\Phi(n-s)=I,\,\,\textrm{   i.e.   } \,\sum_{k=1}^{\kappa}\varphi_{jk}(s)\varphi_{k\ell}(n-s)=\delta_{j\ell},
\end{align}
and the functional equations
\begin{align}\label{functional}
E_{k}(P,n-s)=\sum_{\ell=1}^{\kappa}\varphi_{k\ell}(n-s)E_{\ell}(P,s),\qquad k\in\{1,\ldots,\kappa\}.
\end{align}
Differentiating \eqref{firelation} we obtain
\begin{align}\label{diff}
 \sum_{k=1}^{\kappa}\varphi_{jk}'(s)\varphi_{k\ell}(n-s)=\sum_{k=1}^{\kappa}\varphi_{jk}(s)\varphi_{k\ell}'(n-s).
\end{align}
Furthermore $\Phi(s)$ satisfies the relation
\begin{align}\label{transpose}
\overline{\Phi(\bar{s})}=\Phi(s),
\end{align}
and we recall that (by \eqref{firelation}, \eqref{transpose} and the symmetry) $\Phi(s)$ is unitary on the line $\Re s=\sfrac{n}{2}$, i.e.\
\begin{align}\label{unitary}
 \sum_{k=1}^{\kappa}\varphi_{jk}(\sfrac{n}{2}+iT)\overline{\varphi_{\ell k}(\sfrac{n}{2}+iT)}=\delta_{j\ell}.
\end{align}

Now any $f\in L^2(\Gamma\setminus\H^{n+1})$ can be spectrally decomposed as follows:
\begin{align}\label{specialD'}
 f(P)=\sum_{m\geq0}c_m \phi_m(P) + \sum_{\ell=1}^{\kappa}\int_{0}^{\infty}g_{\ell}(t) E_{\ell}\big(P,\sfrac{n}{2}+it\big)\,dt
\end{align}
(convergence in the $L^2(\Gamma\setminus\H^{n+1})$-norm), where $c_m=\langle f,\phi_m \rangle$ and ``$g_{\ell}(t)=\linebreak \frac{1}{2\pi}\int_{\F}f(P)\overline{E_{\ell}\big(P,\sfrac{n}{2}+it\big)}\,d\nu(P)$'' (since $E_{\ell}\big(\cdot,\sfrac{n}{2}+it\big)\notin L^2(\Gamma\setminus{\H^{n+1}})$ this has to be properly considered as a limit in $L^2(0,\infty)$), and we have a corresponding Parseval's formula:
\begin{align}\label{norm}
  \|f\|_{L^2}^2=\sum_{m\geq0}|c_m|^2+2\pi\sum_{\ell=1}^{\kappa}\int_{0}^{\infty}|g_{\ell}(t)|^2\,dt.
 \end{align}
Cf.\ \cite[Thm.\ 7.1, Cor.\ 7.2, \S 7.29]{cs} and also use \eqref{functional} and \eqref{unitary}.

\subsection{Pointwise convergence}

We will give conditions on $f$ that guarantee that the right hand side of  \eqref{specialD'} is uniformly and absolutely convergent on compact subsets of $\H^{n+1}$. Note that whenever this holds, equation \eqref{specialD'} holds true as a pointwise relation, for almost every $P\in\H^{n+1}$. This is proved by a standard argument, taking inner products with characteristic functions of nicely shrinking sets and using \cite[Thm.\ 7.10]{rudin}.

The arguments in the present section should be compared with \cite[pp.\ 243-245, 732]{dennis}, where the case $n=1$ is considered. Note that our treatment is different from Hejhal's, in that we do not use the Green's function.

For any real $k\geq0$ we introduce the Sobolev spaces $H^k(\Gamma\setminus\H^{n+1})$,
\begin{align*}
 &H^k(\Gamma\setminus\H^{n+1})\\
&:=\bigg\{f\in L^2(\Gamma\setminus{\H^{n+1}}) \,\,\Big|\,\, \sum_{m\geq0}|c_m|^2|r_m+1|^{2k}+\sum_{\ell=1}^{\kappa}\int_{0}^{\infty}|g_{\ell}(t)|^2(t+1)^{2k}\,dt<\infty\bigg\},
\end{align*}
where $r_m\in[-i\frac{n}{2},0)\cup[0,+\infty)$ is defined by the relation $s_m=\frac{n}{2}+ir_m$, and it is understood that $f$ satisfies \eqref{specialD'}. This space is equipped with the norm
\begin{align*}
 \|f\|_{H^k}^2=\sum_{m\geq0}|c_m|^2|r_m+1|^{2k}+\sum_{\ell=1}^{\kappa}\int_{0}^{\infty}|g_{\ell}(t)|^2(t+1)^{2k}\,dt.
\end{align*}

\begin{prop}\label{pointwise}
 Let $k>\sfrac{n+1}{2}$ and $f\in H^k(\Gamma\setminus\H^{n+1})$. Then the spectral decomposition \eqref{specialD'} converges uniformly and absolutely on compact subsets of $\H^{n+1}$.
\end{prop}

\begin{proof}
Let $R>0$ be a large number, say $R\geq100$. We need to estimate the expression
\begin{align*}
D_R:=\sum_{r_m\geq R}\big|c_m \phi_m(P)\big| + \sum_{\ell=1}^{\kappa}\int_{R}^{\infty}\big|g_{\ell}(t) E_{\ell}\big(P,\sfrac{n}{2}+it\big)\big|\,dt.
\end{align*}
Using the Cauchy-Schwarz inequality we find
\begin{align}\label{CSspec}
 &D_R\leq\Big(\sum_{r_m\geq R}|c_m|^2|r_m+1|^{2k}\Big)^{1/2}\Big(\sum_{r_m\geq R}|r_m+1|^{-2k}|\phi_m(P)|^2\Big)^{1/2}\\
&+\sum_{\ell=1}^{\kappa}\Big(\int_{R}^{\infty}|g_{\ell}(t)|^2(t+1)^{2k}\,dt\Big)^{1/2}\Big(\int_{R}^{\infty}(t+1)^{-2k}\big|E_{\ell}\big(P,\sfrac{n}{2}+it\big)\big|^2\,dt\Big)^{1/2}.\nonumber
\end{align}
 Since $f\in H^k(\Gamma\setminus\H^{n+1})$, the first factor in each product in \eqref{CSspec} tends to zero as $R\to\infty$. Hence it suffices to show that the remaining factors are uniformly bounded on compact subsets of $\H^{n+1}$.

For $T\geq0$ let
\begin{align*}
 S_P(T):=\sum_{|r_m|\leq T}\big|\phi_m(P)\big|^2 + \sum_{\ell=1}^{\kappa}\int_{0}^{T}\big| E_{\ell}\big(P,\sfrac{n}{2}+it\big)\big|^2\,dt.
\end{align*}
By the Bessel inequality argument in \cite[\S 7.3, Cor.\ 7.7]{cs}, but applied with the full spectral expansion \eqref{specialD'} (cf.\ \cite[Prop.\ 7.2]{iwa} for the $2$-dimensional case), we have, for all $T\geq1$ and all $P\in\H^{n+1}$,
\begin{align}\label{sp}
S_P(T)=O\big(T^{n+1}+T\mathcal{Y}_{\Gamma}(P)^n\big),
\end{align}
where the implied constant depends only on $\Gamma$. The factors in \eqref{CSspec} that remain to be estimated all have squares bounded by the integral
\begin{align}\label{intspec}
 \int_{0}^{\infty}(t+1)^{-2k}\,dS_P(t).
\end{align}
By integration by parts we obtain
\begin{align*}
 \int_{0}^{\infty}(t+1)^{-2k}\,dS_P(t)=\Big[(t+1)^{-2k}S_P(t)\Big]_{0}^{\infty}+2k\int_{0}^{\infty}(t+1)^{-2k-1}S_P(t)\,dt.
\end{align*}
Hence, using the assumption $k>\sfrac{n+1}{2}$, the bound \eqref{sp} and the fact that $\mathcal{Y}_{\Gamma}$ is continuous, it follows that the integral in \eqref{intspec} is uniformly bounded on compact subsets of $\H^{n+1}$. This concludes the proof.
\end{proof}

Next we will give more concrete conditions on $f$ which force Proposition \ref{pointwise} to apply.

\begin{lem}\label{deltaspec}
If  $f\in C^2(\H^{n+1})\cap L^2(\Gamma\setminus\H^{n+1})$ and $\Delta f\in L^2(\Gamma\setminus\H^{n+1})$, and if $f$ has a spectral decomposition as in \eqref{specialD'}, then the spectral decomposition of $-\Delta f$ is
\begin{align*}
 -\Delta f=\sum_{m\geq0}c_m \big((\sfrac{n}{2})^2+r_m^2\big)\phi_m + \sum_{\ell=1}^{\kappa}\int_{0}^{\infty}g_{\ell}(t) \big((\sfrac{n}{2})^2+t^2\big)E_{\ell}\big(\cdot,\sfrac{n}{2}+it\big)\,dt.
\end{align*}
\end{lem}

\begin{proof}
Given $\psi\in C_c^{\infty}(\R_{\geq0})$ we let
\begin{align*}
 k(P,Q):=\psi\bigg(\frac{|P-Q|^2}{2y_Py_Q}\bigg),\qquad (P,Q)\in\H^{n+1}\times\H^{n+1},
\end{align*}
be the associated point pair invariant. The corresponding integral operator,
\begin{align*}
 L_kg(P)=\int_{\H^{n+1}}k(P,Q)g(Q)\,dv(Q),
\end{align*}
is a bounded linear operator on $L^2(\Gamma\setminus\H^{n+1})$. According to a general lemma of Selberg, \cite{selberg}, if $g\in C^2(\H^{n+1})$ is a solution to the equation $-\Delta g=\lambda g$ ($g$ need not be $\Gamma$-automorphic), then
\begin{align*}
 \int_{\H^{n+1}}k(P,Q)g(Q)\,dv(Q)=h_k(\lambda)g(P),
\end{align*}
where $h_k(\lambda)$ is independent of $g$.

We let $k_1(P,Q):=-\Delta_1k(P,Q)$, where $\Delta_1$ indicates that $\Delta$ operates in the first argument. Then $k_1$ is also a point pair invariant and if $g\in C^2(\H^{n+1})\cap L^2(\Gamma\setminus\H^{n+1})$ is such that $\Delta g\in L^2(\Gamma\setminus\H^{n+1})$ then
\begin{align}\label{lk}
L_k(-\Delta g)=L_{k_1}g
\end{align}
(cf. \cite[pp.\ 51-52 (with $\mu=Id$)]{selberg}; in particular note that $\Delta_1k(P,Q)=\Delta_2k(P,Q)$, and apply Green's formula). Applying \eqref{lk} with (e.g.) $g(P)=y_P^s$ (where $s(n-s)=\lambda$), we conclude $h_{k_1}(\lambda)= \lambda h_k(\lambda)$. Now, if we write the spectral decomposition of $-\Delta f$ as
\begin{align*}
 -\Delta f=\sum_{m\geq0}d_m \phi_m + \sum_{\ell=1}^{\kappa}\int_{0}^{\infty}j_{\ell}(t) E_{\ell}\big(\cdot,\sfrac{n}{2}+it\big)\,dt,
\end{align*}
and apply \eqref{lk} (with $g=f$), we get by comparing spectral coefficients:
\begin{align}\label{cm}
 c_m\lambda_mh_{k}(\lambda_m)=d_mh_{k}(\lambda_m),
\end{align}
and
\begin{align}\label{gl}
 g_{\ell}(t)\big((\sfrac{n}{2})^2+t^2\big)h_{k}\big((\sfrac{n}{2})^2+t^2\big)=j_{\ell}(t)h_{k}\big((\sfrac{n}{2})^2+t^2\big),
\end{align}
for all $\ell\in\{1,\ldots,\kappa\}$ and almost all $t\geq0$.

To conclude the proof, we point out that for any given $M>0$ it is possible to find a point pair invariant $k$ such that $h_k(\lambda)\neq0$ for all $\lambda\in[0,M]$; cf., e.g., \cite[Lemma 7.5]{cs}. Applying \eqref{cm} and \eqref{gl} for such choices of $k$ (letting $M\to\infty$), we conclude $d_m=c_m\lambda_m$ for all $m$ and $j_{\ell}(t)=g_{\ell}(t)\big((\sfrac{n}{2})^2+t^2\big)$ for almost all $t\geq0$, as desired.
\end{proof}

\begin{remark}\label{H2}
For $f$ as in Lemma \ref{deltaspec} we obtain, using Parseval's formula \eqref{norm}:
\begin{align*}
  \sum_{m\geq0}|c_m|^2\big|(\sfrac{n}{2})^2+r_m^2\big|^2+2\pi\sum_{\ell=1}^{\kappa}\int_{0}^{\infty}|g_{\ell}(t)|^2\big|(\sfrac{n}{2})^2+t^2\big|^2\,dt<\infty.
\end{align*}
Thus $f\in H^2(\Gamma\setminus\H^{n+1})$.
\end{remark}

\begin{remark}
Let $k_0=\lfloor\sfrac{n+1}{4}\rfloor+1$. By repeated use of the argument in Lemma \ref{deltaspec} and Remark \ref{H2} we find that if $f\in C^{2k_0}(\H^{n+1})\cap L^2(\Gamma\setminus\H^{n+1})$ is such that $\Delta^{\ell} f\in L^2(\Gamma\setminus\H^{n+1})$ for all $0\leq \ell\leq k_0$, then $f\in H^{2k_0}(\Gamma\setminus\H^{n+1})$. Since $2k_0>\sfrac{n+1}{2}$ it follows from Proposition \ref{pointwise} that the spectral expansion of any such $f$ converges uniformly and absolutely on compact subsets of $\H^{n+1}$.
\end{remark}

\section{Rankin-Selberg type bounds}\label{RSsec}

In this section we prove Rankin-Selberg type bounds on sums of absolute squares of the Fourier coefficients of the Eisenstein series or a cusp form at a cusp.

From now on we assume that $\infty\in\{\eta_1,...,\eta_{\kappa}\}$. Any discrete subgroup of $\PSL(2,C_n)$ can be brought into such a form by an auxiliary conjugation. Of course we may assume that $\eta_1=\infty$. By a further conjugation we may also assume that $A_1=\matr{1}{0}{0}{1}$. We define $\mu_0$ by $\mu_0:=\min_{\vecmu\in\Lambda_1^*\setminus{\{\vec0\}}}|\vecmu|$.

\subsection{First versions}

We first consider the Eisenstein series $E_k(P,\frac{n}{2}+iT)$ for $T\in\R_{\geq0}$, and its Fourier expansion at the cusp at $\infty$.

\begin{prop}\label{rankin1}
Let $Y>0$ be fixed and let $k\in\{1,...,\kappa\}$. Consider the Fourier expansion of $E_k(P,\frac{n}{2}+iT)$ at the cusp $\eta_1=\infty$:
\begin{multline}\label{fseries}
E_k\big(\vecx+yi_n,\sfrac{n}{2}+iT\big)=\delta_{k1}y^{\frac{n}{2}+iT}+\varphi_{k1}\big(\sfrac{n}{2}+iT\big)y^{\frac{n}{2}-iT}\\
+\sum_{\vec0\neq\vecmu\in\Lambda_1^*}a_{\vecmu}\big(\sfrac{n}{2}+iT\big)y^{n/2}K_{iT}(2\pi|\vecmu|y)e^{2\pi i\langle\vecmu,\vecx\rangle}.
\end{multline}
Then we have, uniformly over $X\geq\frac{\mu_0}{2}$ and $ 0\leq T \leq Y$:
\begin{align*}
\sum_{0<|\vecmu|\leq X}\big|a_{\vecmu}\big(\sfrac{n}{2}+iT\big)\big|^2=O\Big(X^{n}\big(1+\log^+(X)\big)\Big),
\end{align*}
where the implied constant depends only on $Y$ and $\Gamma$.
\end{prop}

\begin{proof}
 We mimic \cite[Prop.\ 4.1]{andreas}. Let $\mathcal{P}_1$ be the interior of a fundamental parallelogram for the lattice $\Lambda_1$. We will now for varying $R$ and $H$, satisfying $0<R<H$, study the integral
 \begin{align*}
  \mathcal{J}:=\int_{\mathcal P_1\times(R,H)}\big|E_k\big(P,\sfrac{n}{2}+iT\big)\big|^2 \,d\nu(P).
 \end{align*} 
By the automorphy of the Eisenstein series we get
 \begin{align}\label{Jint}
  \mathcal{J}=\sum_{W\in\Gamma}\int_{\F}I\big(W(P)\in\mathcal{P}_1\times(R,H)\big)\big|E_k\big(P,\sfrac{n}{2}+iT\big)\big|^2\, d\nu(P),
 \end{align}
where $I(\cdot)$ is the indicator function.

We let $\tau_0:=\max_{\ell\in\{1,...,\kappa\}}\tau_1\tau_{\ell}$ and define $\widehat{B}:=\max(B_0,H,\tau_0R^{-1})$ (recall the definitions of $B_0$ and $\tau_{\ell}$ given in Section \ref{notation}). For each $P\in\F\setminus\F_{\widehat{B}}$ there exists
$\ell\in\{1,...,\kappa\}$ such that $P\in\F_{\ell}(\widehat{B})$, i.e.\
$y_{A_{\ell}(P)}\geq \widehat{B}$. Using Lemma \ref{cuspineq} we get that for
each $W\in\Gamma$
\begin{align*}
 y_{W(P)}=y_{A_1W(P)}\leq \frac{\tau_1\tau_{\ell}}{y_{A_{\ell}(P)}}\leq\frac{\tau_0}{\widehat{B}}\leq R, 
 \end{align*}
unless $\ell=1$, $W\in\Gammainfty$ and $y_{W(P)}=y_{P}\geq\widehat{B}\geq H$.
In either case we have $W(P)\notin\mathcal{P}_1\times(R,H)$. This
proves that the integrand in \eqref{Jint} is equal to zero for all
$P\in\F\setminus\F_{\widehat{B}}$ and all $W\in\Gamma$ and we conclude that
\begin{align}\label{JJ}
 \mathcal{J}=\sum_{W\in\Gamma}\int_{\F_{\BB}}I\big(W(P)\in\mathcal{P}_1\times(R,H)\big)\big|E_k\big(P,\sfrac{n}{2}+iT\big)\big|^2\, d\nu(P).
\end{align}

We now note that both $\F_{\BB}$ and $\mathcal{P}_1\times(R,H)$ are
bounded regions which, since $\Gamma$ is discrete, implies that
\begin{align*}
 \#\big\{W\in\Gamma\mid W(\F_{\BB})\cap(\mathcal{P}_1\times(R,H))\neq\emptyset\big\}<\infty.
\end{align*}
This in turn implies that there are only finitely many non-zero terms in the sum \eqref{JJ}. Hence we may interchange the order of summation and integration to get
\begin{align*}
 \mathcal{J}=\int_{\F_{\BB}}\#\big\{W\in\Gamma\mid W(P)\in\mathcal{P}_1\times(R,H)\big\}\big|E_k\big(P,\sfrac{n}{2}+iT\big)\big|^2\, d\nu(P).
\end{align*}
Using Lemma \ref{counting}, and the fact that
$|\Gammainfty'\setminus\Gammainfty|<\infty$, we get
\begin{align*}
 \#\big\{W\in\Gamma\mid W(P)\in\mathcal{P}_1\times(R,H)\big\}\leq \#\big\{W\in\Gammainfty'\setminus\Gamma\mid y_{W(P)}>R\big\}=O\Big(1+\frac{1}{R^n}\Big),
\end{align*}
where the bound is uniform for all $P\in\H^{n+1}$ and all $R>0$. Furthermore, by Lemma \ref{estim} we have $\big|E_k\big(P,\sfrac{n}{2}+iT\big)\big|=O\big((\mathcal{Y}_{\Gamma}(P))^{n/2}\big)$ uniformly for $0\leq T\leq Y$ and $P\in\H^{n+1}$. Hence it follows from Lemma \ref{integral} that
\begin{align}\label{resulta}
 \mathcal{J}=O\Big(\Big(1+\frac{1}{R^n}\Big)\log\big(2\BB\big)\Big).
\end{align}

Next we substitute the Fourier expansion given in
\eqref{fseries} in the definition of $\mathcal{J}$. Using
Parseval's relation we find that $\mathcal{J}$ equals
\begin{multline*}
|\mathcal{P}_1|\int_{R}^H \Big(\big| \delta_{k1}y^{\frac{n}{2}+iT}+\varphi_{k1}\big(\sfrac{n}{2}+iT\big)y^{\frac{n}{2}-iT}\big|^2+\sum_{\vec0\neq\vecmu\in\Lambda_1^*}\big|a_{\vecmu}\big(\sfrac{n}{2}+iT\big)\big|^2y^n\big|K_{iT}(2\pi|\vecmu|y)\big|^2 \Big)\, \frac{dy}{y^{n+1}}\\
\geq|\mathcal{P}_1|\sum_{\vec0\neq\vecmu\in\Lambda_1^*}\big|a_{\vecmu}\big(\sfrac{n}{2}+iT\big)\big|^2\int_{R}^HK_{iT}(2\pi|\vecmu|y)^2\, \frac{dy}{y}\\
=|\mathcal{P}_1|\sum_{\vec0\neq\vecmu\in\Lambda_1^*}\big|a_{\vecmu}\big(\sfrac{n}{2}+iT\big)\big|^2\int_{2\pi|\vecmu|R}^{2\pi|\vecmu|H}
K_{iT}(t)^2\, \frac{dt}{t}.
\end{multline*}
We take $H=\frac{1}{\mu_0}$ and $R=(2\pi X)^{-1}$. Note that for such $R,H$ we have $[1,2\pi]\subset[2\pi |\vecmu|R,2\pi |\vecmu|H]$ for all $\vecmu\in\Lambda_1^*\setminus\{\vec0\}$ with $|\vecmu|\leq X$. Thus
\begin{align}\label{resultb}
\sum_{0<|\vecmu|\leq X}\big|a_{\vecmu}\big(\sfrac{n}{2}+iT\big)\big|^2\leq C^{-1}\mathcal{J},
\end{align}
where $C$ is defined by
\begin{align*}
C=C(Y)=|\mathcal{P}_1|\inf_{T\in[0,Y]}\int_{1}^{2\pi}K_{iT}(t)^2\,\frac{dt}{t}.
\end{align*}
Finally we note that with these choices we get $\BB=\max(B_0,\mu_0^{-1},2\pi \tau_0X)$, and the proposition follows from \eqref{resulta} and \eqref{resultb}.
\end{proof}

We continue with a similar result concerning (certain) cusp forms on $\Gamma\setminus\H^{n+1}$, where we also keep control of the dependence on the eigenvalue.

\begin{prop}\label{cusprankin}
Let $\phi$ be a cusp form on $\Gamma\setminus\H^{n+1}$ with eigenvalue $\lambda=(\sfrac{n}{2})^2+T^2$, $T\geq0$
, normalized in such a way that $\int_{\F}|\phi(P)|^2\,d\nu(P)=1$. Consider the Fourier expansion of $\phi$ at the cusp $\eta_1=\infty$:
\begin{align}\label{cuspexpansion}
\phi(\vecx+yi_n)=\sum_{\vec0\neq\vecmu\in\Lambda_1^*}c_{\vecmu}y^{n/2}K_{iT}(2\pi|\vecmu|y)e^{2\pi i\langle\vecmu,\vecx\rangle}.
\end{align}
Then we have, uniformly over $X\geq\frac{\mu_0}{2}$:
\begin{align*}
\sum_{0<|\vecmu|\leq X}|c_{\vecmu}|^2=O\Big(e^{\pi T}\Big(T+\frac{X^n}{(T+1)^{n-1}}\Big)\Big),
\end{align*}
where the implied constant depends only on $\Gamma$.
\end{prop}

\begin{proof}
We let
\begin{align*}
 \mathfrak{J}:=\int_{\mathcal P_1\times(R,H)}|\phi(P)|^2\,d\nu(P),
\end{align*}
where $0<R<H$.
By the same argument as in the proof of Proposition \ref{rankin1} we
obtain
\begin{align*}
\mathfrak{J}=O\Big(1+\frac{1}{R^n}\Big)\int_{\F_{\BB}}|\phi(P)|^2\, d\nu(P)=
O\Big(1+\frac{1}{R^n}\Big).
\end{align*}
We let $R=\frac{T+1}{8\pi X}$ and $H=\frac{T+1}{4\pi\mu_0}$. For such $R$ and $H$ we get
\begin{align}\label{shortint}
\mathfrak{J}=O\Big(1+\frac{X^n}{(T+1)^n}\Big).
\end{align}

If we instead substitute the Fourier expansion \eqref{cuspexpansion} in the definition of $\mathfrak{J}$ and apply Parseval's relation we get
\begin{align*}
 \mathfrak{J}=|\mathcal{P}_1|\sum_{\vec0\neq\vecmu\in\Lambda_1^*}|c_{\vecmu}|^2\int_{2\pi|\vecmu|R}^{2\pi|\vecmu|H}
K_{iT}(t)^2 \,\frac{dt}{t}.
\end{align*}
With the above choice of $R$ and $H$ we have
$[\frac{T+1}{4},\frac{T+1}{2}]\subset[2\pi |\vecmu|R,2\pi |\vecmu|H]$ for all $\vecmu\in\Lambda_1^*\setminus\{\vec0\}$ with $|\vecmu|\leq X$. Thus
\begin{align*}
 \sum_{0<|\vecmu|\leq X}|c_{\vecmu}|^2\leq C^{-1}\mathfrak{J},
\end{align*}
where $C$ is defined by
\begin{align}\label{C}
C=|\mathcal{P}_1|\int_{\frac{T+1}{4}}^{\frac{T+1}{2}}K_{iT}(t)^2\,\frac{dt}{t}.
\end{align}
By a minor modification of \cite[Lemma 3.1.2]{andreas2} (cf.\ also \cite{bal}) we find that $C^{-1}=O\big((T+1)e^{\pi T}\big)$ holds for all $T\geq0$. This fact together with \eqref{shortint} gives the desired result.
\end{proof}

\subsection{Strong version for the Eisenstein series}\label{strongversionsection}

Finally we prove a Rankin-Selberg bound for the Eisenstein series with explicit control on the dependence on the eigenvalue. First we recall some properties of the spectral majorant function $W(t)$ which is defined in \cite[eq.\ (7.10)]{cs}. $W(t)$ is an even function that depends only on $\Gamma$, and $W(t)\geq 1$ for all $t\in\R$. As in \cite[p.\ 315]{dennis} (cf.\ also \cite[Prop.\ 7.12]{cs}) we have
\begin{align}\label{westimate}
\sum_{k=1}^{\kappa}\sum_{j=1}^{\kappa}\varphi_{kj}'(\sfrac{n}{2}+iT)\overline{\varphi_{kj}(\sfrac{n}{2}+iT)}=O\big(W(T)\big),
\end{align}
for all $T\in\R$. We also recall from \cite[Thm.\ 7.14]{cs} that
\begin{align}\label{wt}
 \int_0^TW(t)\,dt=O(T^{n+1})\:\:\:\textrm{as $T\to\infty$}.
\end{align}

We also need the ``cut-off'' Eisenstein series $E_k^{B}(P,s)$ which is defined, for any $B\geq B_0$ and $P\in\F$, by
\begin{align*}
  E_k^{B}(P,s):=\begin{cases}
    E_k(P,s)& \textrm{ if $P\in\F_{B}$,}\\
    E_k(P,s)-\delta_{k\ell}(y_{A_{\ell}(P)})^{s}-\varphi_{k\ell}(s)(y_{A_{\ell}(P)})^{n-s} & \textrm{ if $P\in\F_{\ell}(B)$.}
    \end{cases}
\end{align*}
Using the appropriate Maass-Selberg relation (\cite[6.1.62]{cs}) we obtain, for all $s\in\C\setminus\R$ with $\Re s>\sfrac{n}{2}$:
\begin{multline}\label{MS}
 \int_{\F}\big|E_k^{B}(P,s)\big|^2\,d\nu(P)\\=\frac{B^{s+\bar{s}-n}-\sum_{j=1}^{\kappa}|\varphi_{kj}(s)|^2B^{n-s-\bar{s}}}{s+\bar{s}-n}+\frac{\overline{\varphi_{kk}(s)}B^{s-\bar{s}}-\varphi_{kk}(s)B^{\bar{s}-s}}{s-\bar{s}}.
\end{multline}
We want to consider $s$-values of the form $s=\sfrac{n}{2}+it$ with $t\in\R$, but we notice that for such $s$ the right hand side of \eqref{MS} is not well-defined since $s+\bar{s}-n=0$. In order to overcome this problem we let $s=\sigma+iT$ with $T\neq 0$ fixed and let $\sigma\to \sfrac{n}{2}$ from the right. The following computations correspond to \cite[Prop.\ 7.12 (ii)]{cs} but keep track also of the dependence on $B$.

In terms of  $\sigma$ and $T$, the right hand side of \eqref{MS} equals
\begin{align*}
\frac{B^{2\sigma-n}-B^{n-2\sigma}\sum_{j=1}^{\kappa}|\varphi_{kj}(\sigma+iT)|^2}{2\sigma-n}+\frac{\overline{\varphi_{kk}(\sigma+iT)}B^{2iT}-\varphi_{kk}(\sigma+iT)B^{-2iT}}{2iT}.
\end{align*}
Using \eqref{unitary} we note that
\begin{multline*}
\frac{B^{2\sigma-n}-B^{n-2\sigma}\sum_{j=1}^{\kappa}|\varphi_{kj}(\sigma+iT)|^2}{2\sigma-n}\\
 \to 2\log B-\frac{1}{2}\sum_{j=1}^{\kappa}\frac{\partial}{\partial\sigma}\Big(\big|\varphi_{kj}(\sigma+iT)\big|^2\Big)_{\big|\sigma=\sfrac{n}{2}}\:\:\:\:\textrm{as $\sigma\to\sfrac{n}{2}$}.
\end{multline*}
Furthermore, it follows from \eqref{diff} and \eqref{transpose} that
\begin{align*}
 &\sum_{j=1}^{\kappa}\frac{\partial}{\partial\sigma}\Big(\big|\varphi_{kj}(\sigma+iT)\big|^2\Big)_{\big|\sigma=\sfrac{n}{2}}=2\sum_{j=1}^{\kappa}\varphi_{kj}'(\sfrac{n}{2}+iT)\overline{\varphi_{kj}(\sfrac{n}{2}+iT)}
\end{align*}
and hence
\begin{multline*}
 \frac{B^{2\sigma-n}-B^{n-2\sigma}\sum_{j=1}^{\kappa}|\varphi_{kj}(\sigma+iT)|^2}{2\sigma-n}\\\to
 2\log B-\sum_{j=1}^{\kappa}\varphi_{kj}'(\sfrac{n}{2}+iT)\overline{\varphi_{kj}(\sfrac{n}{2}+iT)}\:\:\:\:\textrm{as $\sigma\to\sfrac{n}{2}$}.
\end{multline*}
We also note that
\begin{align*}
 \frac{\overline{\varphi_{kk}(\sigma+iT)}B^{2iT}-\varphi_{kk}(\sigma+iT)B^{-2iT}}{2iT}
 \to2\text{Re}\bigg(\frac{\overline{\varphi_{kk}(\sfrac{n}{2}+iT)}B^{2iT}}{2iT}\bigg) \:\:\:\textrm{as $\sigma\to\sfrac{n}{2}$}.
\end{align*}
We can thus conclude that for $T\neq 0$ we have
\begin{multline*}
 \int_{\F}\big|E_k^{B}(P,\sfrac{n}{2}+iT)\big|^2\,d\nu(P)\\=2\log B-\sum_{j=1}^{\kappa}\varphi_{kj}'(\sfrac{n}{2}+iT)\overline{\varphi_{kj}(\sfrac{n}{2}+iT)}+\text{Re}\bigg(\frac{\overline{\varphi_{kk}(\sfrac{n}{2}+iT)}B^{2iT}}{iT}\bigg).
\end{multline*}
Finally, using \eqref{westimate}, we obtain that for all $T\geq 1$ we have
\begin{align}\label{eisensteinest}
 \sum_{k=1}^{\kappa}\int_{\F}\big|E_k^{B}(P,\sfrac{n}{2}+iT)\big|^2\,d\nu(P)=2\kappa\log B+O\big(W(T)\big).
\end{align}
We are now ready to prove the following proposition.

\begin{prop}\label{FinalRS}
 In the Fourier series \eqref{fseries} we have, uniformly over $X\geq\frac{\mu_0}{2}$ and $T\geq0$,
 \begin{align*}
\sum_{0<|\vecmu|\leq X}\big|a_{\vecmu}\big(\sfrac{n}{2}+iT\big)\big|^2=O\Big(e^{\pi T}\Big(T+\frac{X^n}{(T+1)^{n-1}}\Big)\Big)\Big\{\log^+\Big(\frac{X}{T+1}+T\Big)+W(T)\Big\},
\end{align*}
where the implied constant depends only on $\Gamma$.
\end{prop}

\begin{proof}
 We recall from the proof of Proposition \ref{rankin1} that
 \begin{align}\label{jest}
 \mathcal{J}=O\Big(1+\frac{1}{R^n}\Big)\int_{\F_{\BB}}\big|E_k\big(P,\sfrac{n}{2}+iT\big)\big|^2\, d\nu(P).
\end{align}
Using \eqref{eisensteinest}, we obtain for all $T\geq 1$:
\begin{align}\label{finaljest}
 \mathcal{J}=O\Big(1+\frac{1}{R^n}\Big)\int_{\F}\big|E_k^{\widehat{B}}(P,\sfrac{n}{2}+iT)\big|^2\, d\nu(P)=O\Big(1+\frac{1}{R^n}\Big)\Big\{\log\widehat{B}+W(T)\Big\},
\end{align}
where the implied constant depends only on $\Gamma$. The bound \eqref{finaljest} holds also for $0\leq T\leq 1$, as follows directly from \eqref{jest} (cf.\ \eqref{resulta}).

We also recall that
\begin{align*}
\mathcal{J}\geq |\mathcal{P}_1|\sum_{\vec0\neq\vecmu\in\Lambda_1^*}\big|a_{\vecmu}\big(\sfrac{n}{2}+iT\big)\big|^2\int_{2\pi|\vecmu|R}^{2\pi|\vecmu|H}
K_{iT}(t)^2\, \frac{dt}{t}.
\end{align*}
As in the proof of Proposition \ref{cusprankin} we let $R=\frac{T+1}{8\pi X}$, $H=\frac{T+1}{4\pi\mu_0}$ and note that with this choice we have
\begin{align*}
 \sum_{0<|\vecmu|\leq X}\big|a_{\vecmu}\big(\sfrac{n}{2}+iT\big)\big|^2\leq C^{-1}\mathcal{J},
\end{align*}
where $C$ is given by \eqref{C}. Recalling the definition of $\widehat B$ as well as the facts $C^{-1}=O\big((T+1)e^{\pi T}\big)$ and $W(T)\geq 1$ we get, uniformly for $X\geq \frac{\mu_0}{2}$ and $T\geq 0$,
\begin{align*}
 &\sum_{0<|\vecmu|\leq X}\big|a_{\vecmu}\big(\sfrac{n}{2}+iT\big)\big|^2\\
&=O\bigg((T+1)e^{\pi T}\Big(1+\Big(\frac{8\pi X}{T+1}\Big)^n\Big)\bigg)\Big\{\log^+\Big(\frac{X}{T+1}+T\Big)+W(T)\Big\},
\end{align*}
which is the desired bound.
\end{proof}

\section{The horosphere integral}

As discussed in the introduction, our main objective in this paper is to prove statements of the form:
\begin{align}\label{errortermigen}
  \frac{1}{\delta_1\cdots\delta_n}\int_{\R^n}\chi_{\vecdel,\vecgam}(\vecu)f(u_1\vecom_1+\cdots+u_n\vecom_n+yi_n)\,d\vecu\\
\longrightarrow\frac{\langle\chi\rangle}{\nu(\Gamma\setminus\H^{n+1})}\int_{\Gamma\setminus\H^{n+1}}f(P)\,d\nu(P)\nonumber
 \end{align}
as $y\to0$. Here $f$ is an arbitrary (compactly supported) test function on $\Gamma\setminus\H^{n+1}$ and
$\vecom_1,\ldots,\vecom_n$ is a basis of the lattice $\Lambda_1$ corresponding to the cusp $\eta_1=\infty$; also
\begin{align}\label{chidelta}
\chi_{\vecdel,\vecgam}(\vecu):=\chi\big(\sfrac{u_1-\gamma_1}{\delta_1},\ldots,\sfrac{u_n-\gamma_n}{\delta_n}\big),
\end{align}
where $\chi:\R^n\to\R$ is a fixed measurable function with compact support, $\vecgam=(\gamma_1,\ldots,\gamma_n) \in\R^n$, $\vecdel=(\delta_1,\ldots,\delta_n)\in(0,1]^n$, and the numbers $\delta_1,\ldots,\delta_n$ are allowed to shrink with $y$ as $y\to0$. We are interested in what restrictions are needed on $\delta_1,\ldots,\delta_n$ as $y\to0$ in order that \eqref{errortermigen} holds. Furthermore, when  \eqref{errortermigen} holds, we wish to prove precise results on the rate of convergence in  \eqref{errortermigen}.

We will call the left hand side of  \eqref{errortermigen} "the horosphere integral". We will focus on the situation where $\chi$ is either smooth or the characteristic function of the rectangle $[-1/2,1/2]^n$. Recall that in the latter situation we are in fact studying ``horosphere pieces'' of the form
\begin{align*}
 \Big\{u_1\vecom_1+\cdots+u_n\vecom_n+yi_n\,\big|\, u_i\in[\alpha_i,\beta_i]\:\textrm{for}\:i=1,\ldots,n\Big\}
\end{align*}
where $\beta_i-\alpha_i=\delta_i$ for $i=1,\ldots,n$. (Here $\alpha_i=\gamma_i-\delta_i/2$, $\beta_i=\gamma_i+\delta_i/2$.) 

To prove \eqref{errortermigen} we will use the spectral expansion of $f$ and Fourier expansion of $\chi_{\vecdel,\vecgam}(\vecu)$. The first step is to note that, since $\vecu\mapsto f(u_1\vecom_1+\ldots+u_n\vecom_n+yi_n)$ only depends on $\vecu$ modulo $\Z^n$, we have
\begin{align}\label{hoppasattdetsnartarklart}
&\frac{1}{\delta_1\cdots\delta_n}\int_{\R^n}\chi_{\vecdel,\vecgam}(\vecu)f(u_1\vecom_1+\cdots+u_n\vecom_n+yi_n)\,d\vecu\\
&=\frac{1}{\delta_1\cdots\delta_n}\int_{\R^n/\Z^n}\Psi_{\vecdel,\vecgam}(\vecu)f(u_1\vecom_1+\cdots+u_n\vecom_n+yi_n)\,d\vecu,\nonumber
\end{align}
where
\begin{align}\label{psi}
 \Psi_{\vecdel,\vecgam}(\vecu):=\underset{\vecm\in\Z^n}{\sum}\chi_{\vecdel,\vecgam}(\vecu+\vecm).
\end{align}
Note that $\Psi_{\vecdel,\vecgam}(\vecu)$ depends only on $\vecu$ modulo $\Z^n$. Let us suppose that $\chi$ is smooth. Then also $\Psi_{\vecdel,\vecgam}$ is smooth, and may be expanded as a Fourier series
\begin{align}\label{PSIUFIRSTFORMULA}
\Psi_{\vecdel,\vecgam}(\vecu)=\sum_{\vecnu\in\Z^n}
\biggl(\int_{\R^n/\Z^n}\Psi_{\vecdel,\vecgam}(\tilde\vecu)
e^{-2\pi i\langle\vecnu,\tilde\vecu\rangle}\,d\tilde\vecu\biggr)
\,e^{2\pi i\langle\vecnu,\vecu\rangle}.
\end{align}
Substituting \eqref{psi} in \eqref{PSIUFIRSTFORMULA} and changing order of summation and integration we get
\begin{align}\label{chiseries}
 \Psi_{\vecdel,\vecgam}(\vecu)=\sum_{\vecnu\in\Z^n}\widehat{\chi_{\vecdel,\vecgam}}(\vecnu)e^{2\pi i\langle\vecnu,\vecu\rangle},
\end{align}
where $\widehat{\chi_{\vecdel,\vecgam}}$ is the Fourier transform of $\chi_{\vecdel,\vecgam}$, viz.
\begin{align*}
\widehat{\chi_{\vecdel,\vecgam}}(\vecnu)=\int_{\R^n}\chi_{\vecdel,\vecgam}(\vecu)e^{-2\pi i\langle\vecnu,\vecu\rangle}\,d\vecu.
\end{align*}
Note also that (via the substitution $u_i=\gamma_i+\delta_ix_i$)
\begin{align}\label{transf2}
\widehat{\chi_{\vecdel,\vecgam}}(\vecnu)
=\Big(\prod_{i=1}^n\delta_i\Big)e^{-2\pi i\langle\vecnu,\vecgam\rangle}\widehat{\chi}(\delta_1\nu_1,\ldots,\delta_n\nu_n).
\end{align}
The following formula for $\widehat{\chi_{\vecdel,\vecgam}}$ will also be useful.

\begin{lem}\label{partint}
Let $\vecnu\in\R^n\setminus\{\vec0\}$ and $m\in\Z_{\geq0}$. Then
\begin{align*}
 \widehat{\chi_{\vecdel,\vecgam}}(-\vecnu)=\Big(\frac{i}{2\pi}\Big)^m\frac{1}{|\vecnu|^{2m}}\int_{\R^n}\bigg(\Big(\nu_1\frac{\partial}{\partial u_1}+\cdots+\nu_n\frac{\partial}{\partial u_n}\Big)^m\chi_{\vecdel,\vecgam}(\vecu)\bigg)e^{2\pi i\langle\vecnu,\vecu\rangle}\,d\vecu.
\end{align*}
 \end{lem}

\begin{proof}
The lemma follows from a straightforward application of \cite[p.\ 4, Thm.\ 1.8]{SW}.
\end{proof}

Using Lemma \ref{partint} we find that for all $\vecnu\in\R^n\setminus\{\vec0\}$, \begin{align*}
\widehat{\chi}(\vecnu)=O\big(\|\chi\|_{m,1}|\vecnu|^{-m}\big),
\end{align*}
where the implied constant only depends on $m\in\Z_{\geq0}$ and where
\begin{align*}
\|\chi\|_{m,1}:=\underset{|\ell|\leq m}{\sum}\|D^{\ell}\chi\|_{L^1(\R^n)}
\end{align*}
is the norm in the Sobolev space $W^{m,1}(\R^n)$. (Recall that for any multi-index $\ell=(\ell_1,\ldots,\ell_n)$ of length $|\ell|$ we have $D^{\ell}=\frac{\partial^{|\ell|}}{\partial x_1^{\ell_1}\cdots\partial x_n^{\ell_n}}$.) Using this estimate (with $m=0$ and $m=k$) together with $|(\delta_1\nu_1,\ldots,\delta_n\nu_n)|\geq\big(\underset{i\in\{1,\ldots,n\}}{\min}\delta_i\big)|\vecnu|$ in \eqref{transf2} we conclude
\begin{align}\label{transfestim}
\widehat{\chi_{\vecdel,\vecgam}}(\vecnu)=O\Big(\|\chi\|_{k,1}\prod_{i=1}^n\delta_i\Big)
\begin{cases}
 1 & \text{if $|\vecnu|\leq\big(\underset{i\in\{1,\ldots,n\}}{\min}\delta_i\big)^{-1}$}\\
\Big(\big(\underset{i\in\{1,\ldots,n\}}{\min}\delta_i\big)|\vecnu|\Big)^{-k} & \text{if $|\vecnu|\geq\big(\underset{i\in\{1,\ldots,n\}}{\min}\delta_i\big)^{-1}$,}
\end{cases}
\end{align}
where the implied constant only depends on $k\in\Z_{\geq0}$.

\subsection{The horosphere integral for cusp forms}\label{cuspchap}

In this section we study the horosphere integral when $f$ is a cusp form eigenfunction.
We begin with a bound on sums of the Fourier coefficients of cusp forms, twisted with an additive character, generalizing \cite[Thm. 3]{hafner}.

Recall that $\Omega=\{\vecom_1,\vecom_2,\ldots,\vecom_n\}$ is a basis of the lattice $\Lambda_1$. For $i=1,\ldots,n$ we define $ \vecom_i^*$ to be the unique element in $\Lambda_1^*$ satisfying the relations $\langle\vecom_i^*,\vecom_i\rangle=1$ and $\langle\vecom_i^*,\vecom_j\rangle=0$ for $i\neq j$.
Then $\Omega^*=\{\vecom_1^*,\vecom_2^*,\ldots,\vecom_n^*\}$ is a basis for $\Lambda_1^*$, i.e.
\begin{align*}
\Lambda_1^*=\Big\{m_1\vecom_1^*+m_2\vecom_2^*+\cdots +m_n\vecom_n^*\,\big|\, m_1,\ldots m_n\in\Z\Big\}.
\end{align*}

\begin{prop}\label{Lemma1}
 Let $\phi$ be a cusp form with eigenvalue $\lambda>0$. We define $s$ via $\lambda=s(n-s)$, $s\in[n/2,n)\cup [n/2,n/2+i\infty)$. Let the Fourier expansion of $\phi$ at the cusp $\eta_1=\infty$ be
 \begin{align}\label{smallcuspexp}
  \phi(\vecx+yi_n)=\sum_{\vec0\neq\vecmu\in\Lambda_1^*}c_{\vecmu}y^{n/2}K_{s-n/2}(2\pi|\vecmu|y)e^{2\pi i\langle\vecmu,\vecx\rangle}.
 \end{align}
 Let $c_{\vec0}:=0$. Then the following holds, uniformly over $\vecM=(M_1,\ldots,M_n)\in\Z_{\geq0}^n\setminus\{\vec0\}$ and $(\alpha_1,\ldots,\alpha_n)\in\R^n$:
 \begin{align*}
  \sum_{m_1=0}^{M_1}\cdots\sum_{m_n=0}^{M_n}c_{m_1\vecom_1^*+\cdots+m_n\vecom_n^*}e^{2\pi i(m_1\alpha_1+\cdots+m_n\alpha_n)}=O\Big(|\vecM|^{n/2}\big(\log(2|\vecM|)\big)^{n+1}\Big),
 \end{align*}
 where the implied constant depends only on $\Gamma$, $\phi$ and $\Omega$.
\end{prop}

\begin{proof}
We let $I_{\vecal}:=\prod_{i=1}^n[\alpha_i-\sfrac{1}{2},\alpha_i+\sfrac{1}{2}]$. We will study the following integral, for varying \mbox{$(M_1,\ldots,M_n)\in\Z_{\geq0}^n\setminus\{\vec0\}$,} $(\alpha_1,\ldots,\alpha_n)\in\R^n$ and $\delta\in[0,1)$:
\begin{multline*}
 \mathcal{K}:=\int_0^{\infty}\int_{I_{\vecal}}\phi(u_1\vecom_1+\cdots+u_n\vecom_n+yi_n)\\
\times\bigg(\sum_{m_1=0}^{M_1}\cdots\sum_{m_n=0}^{M_n}e^{2\pi i(m_1(\alpha_1-u_1)+\cdots+m_n(\alpha_n-u_n))}\bigg)\,\frac{du_1\ldots du_ndy}{y^{\delta}}.
\end{multline*}
To begin with we recall that $\phi$ decays exponentially in each cusp, which implies that $\phi$ is bounded on $\H^{n+1}$. In particular this means that we can find constants $A,B>0$ such that $|\phi(\vecx+yi_n)|\leq Ae^{-By}$ for all $\vecx+yi_n\in\H^{n+1}$. Using this observation we get
\begin{align}\label{firstK}
 |\mathcal{K}|&\leq \int_0^{\infty}\int_{I_{\vecal}}Ae^{-By}\bigg|\sum_{m_1=0}^{M_1}\cdots\sum_{m_n=0}^{M_n}e^{2\pi i(m_1(\alpha_1-u_1)+\cdots+m_n(\alpha_n-u_n))}\bigg|\,\frac{du_1\ldots du_ndy}{y^{\delta}}\\
 &=O\Big(1+\int_0^1\frac{dy}{y^{\delta}}\Big)\int_{I_{\vecal}}\bigg|\sum_{m_1=0}^{M_1}\cdots\sum_{m_n=0}^{M_n}e^{2\pi i(m_1(\alpha_1-u_1)+\cdots+m_n(\alpha_n-u_n))}\bigg|\,du_1\ldots du_n\nonumber\\
 &=O\Big(\frac{1}{1-\delta}\Big)\prod_{j=1}^n\int_{-\frac{1}{2}}^{\frac{1}{2}}\Big|\sum_{m_j=0}^{M_j}e^{2\pi im_jv_j}\Big|\,dv_j=O\bigg(\frac{\log(2M_{i_1})\cdots\log(2M_{i_k})}{1-\delta}\bigg),\nonumber
\end{align}
where $M_{i_1},\ldots,M_{i_k}$ are the nonzero coordinates in $(M_1,\ldots,M_n)$.

On the other hand, substituting the Fourier expansion \eqref{smallcuspexp} into the definition of $\mathcal{K}$, we get
\begin{multline}\label{Kint}
 \mathcal{K}=\int_0^{\infty}\int_{I_{\vecal}}\sum_{\vec0\neq\vecmu\in\Lambda_1^*}c_{\vecmu}y^{n/2}K_{s-n/2}(2\pi|\vecmu|y)e^{2\pi i\langle\vecmu,u_1\vecom_1+\cdots+u_n\vecom_n\rangle}\\
 \times\bigg(\sum_{m_1=0}^{M_1}\cdots\sum_{m_n=0}^{M_n}e^{2\pi i(m_1(\alpha_1-u_1)+\cdots+m_n(\alpha_n-u_n))}\bigg)\,\frac{du_1\ldots du_ndy}{y^{\delta}}.
\end{multline}
Here, for each fixed $y>0$, the inner integral equals
\begin{align*}
&\sum_{\vec0\neq\vecmu\in\Lambda_1^*}c_{\vecmu}y^{n/2-\delta}K_{s-n/2}(2\pi|\vecmu|y)\\
&\times\bigg(\sum_{m_1=0}^{M_1}\cdots\sum_{m_n=0}^{M_n}e^{2\pi i(m_1\alpha_1+\cdots+m_n\alpha_n)}\int_{I_{\vecal}}e^{2\pi i\langle\vecmu -m_1\vecom_1^*-\cdots-m_n\vecom_n^*,u_1\vecom_1+\cdots+u_n\vecom_n\rangle}\,du_1\ldots du_n\bigg)\\
&=\underset{m_1+\cdots+m_n>0}{\sum_{m_1=0}^{M_1}\cdots\sum_{m_n=0}^{M_n}}c_{m_1\vecom_1^*+\cdots+m_n\vecom_n^*}y^{n/2-\delta}K_{s-n/2}\big(2\pi|m_1\vecom_1^*+\cdots+m_n\vecom_n^*|y\big)e^{2\pi i(m_1\alpha_1+\cdots+m_n\alpha_n)},
\end{align*}
and since this is a finite sum it follows that we can change the order between integration and summation in \eqref{Kint}. Using  \cite[p.\ 388(8)]{watson} we obtain
\begin{align}\label{Kequal}
 \mathcal{K}
 &=\underset{m_1+\cdots+m_n>0}{\sum_{m_1=0}^{M_1}\cdots\sum_{m_n=0}^{M_n}}c_{m_1\vecom_1^*+\cdots+m_n\vecom_n^*}e^{2\pi i(m_1\alpha_1+\cdots+m_n\alpha_n)}\\
&\times\big(2\pi|m_1\vecom_1^*+\cdots+m_n\vecom_n^*|\big)^{\delta-n/2-1}\int_0^{\infty} u^{n/2-\delta}K_{s-n/2}(u)\,du\nonumber\\
 &=\bigg(\underset{m_1+\cdots+m_n>0}{\sum_{m_1=0}^{M_1}\cdots\sum_{m_n=0}^{M_n}}c_{m_1\vecom_1^*+\cdots+m_n\vecom_n^*}|m_1\vecom_1^*+\cdots+m_n\vecom_n^*|^{\delta-n/2-1}e^{2\pi i(m_1\alpha_1+\cdots+m_n\alpha_n)}\bigg)\nonumber\\
&\times\frac{\pi^{\delta-n/2-1}}{4}\Gamma\Big(\frac{n+1}{2}-\frac{\delta+s}{2}\Big)\Gamma\Big(\frac{1+s}{2}-\frac{\delta}{2}\Big).\nonumber
\end{align}
Combining \eqref{firstK}, \eqref{Kequal} and the observation 
\begin{align*}
 \inf_{\delta\in[0,1)}\Big|\frac{\pi^{\delta-n/2-1}}{4}\Gamma\Big(\frac{n+1}{2}-\frac{\delta+s}{2}\Big)\Gamma\Big(\frac{1+s}{2}-\frac{\delta}{2}\Big)\Big|>0
\end{align*}
yields
\begin{multline}\label{Sestimate}
\underset{m_1+\cdots+m_n>0}{\sum_{m_1=0}^{M_1}\cdots\sum_{m_n=0}^{M_n}}c_{m_1\vecom_1^*+\cdots+m_n\vecom_n^*}|m_1\vecom_1^*+\cdots+m_n\vecom_n^*|^{\delta-n/2-1}e^{2\pi i(m_1\alpha_1+\cdots+m_n\alpha_n)}\\
=O\bigg(\frac{\log(2M_{i_1})\cdots\log(2M_{i_k})}{1-\delta}\bigg)=O\bigg(\frac{\big(\log(2|\vecM|)\big)^n}{1-\delta}\bigg),
\end{multline}
for all $\vecM=(M_1,\ldots,M_n)\in\Z_{\geq0}^n\setminus\{\vec0\}$, $(\alpha_1,\ldots,\alpha_n)\in\R^n$ and $\delta\in[0,1)$. 

We call the sum in the left hand side of \eqref{Sestimate} $S(M_1,\ldots,M_n)$ and we define $S(0,\ldots,0):=0$. We also define
\begin{align*}
g(x_1,\ldots,x_n):=|x_1\vecom_1^*+\cdots+x_n\vecom_n^*|^{n/2+1-\delta},
\end{align*}
and note that $g$ is smooth in $\R^n\setminus\{\vec0\}$. Using Corollary \ref{corsumbp} and \eqref{Sestimate} we get
\begin{align}\label{ibp}
 &\sum_{m_1=0}^{M_1}\cdots\sum_{m_n=0}^{M_n}c_{m_1\vecom_1^*+\cdots+m_n\vecom_n^*}e^{2\pi i(m_1\alpha_1+\cdots+m_n\alpha_n)}\nonumber\\
&=\underset{A\subset N}{\sum}(-1)^{|A|}\underset{\prod_{j\in A}[0,M_j]}{\int} g_{A,N\setminus A}(\vecx)S_{A,N\setminus A}(\vecx)\,d\vecx\\
&=O\bigg(\frac{\big(\log(2|\vecM|)\big)^n}{1-\delta}\bigg)\bigg\{\underset{A\subset N}{\sum}\underset{\prod_{j\in A}[0,M_j]}{\int} |g_{A,N\setminus A}(\vecx)|\,d\vecx \bigg\},\nonumber
\end{align}
where we recall that $N=\{1,\ldots,n\}$. 

In order to bound the right hand side of \eqref{ibp} we need to estimate the derivatives of $g$. We note that
\begin{align*}
 g(x_1,\ldots,x_n)=Q(x_1,\ldots,x_n)^{\sfrac{n}{4}+\sfrac{1}{2}-\sfrac{\delta}{2}},
\end{align*}
for some positive definite quadratic form $Q$ that only depends on $\Omega^*$. We let $k_1,k_2>0$ be such that 
\begin{align}\label{quadraticform}
k_1|\vecx|^2\leq Q(\vecx)\leq k_2|\vecx|^2
\end{align}
for all $\vecx=(x_1,\ldots,x_n)\in\R^n$. Differentiating $g$ and using \eqref{quadraticform} we obtain, for all $\vecx\in\R^n\setminus\{\vec0\}$, $\delta\in[0,1)$ and multi-indices $\ell=(\ell_1,\ldots,\ell_n)$ of length $|\ell|$,
\begin{align}\label{gder}
\frac{\partial^{|\ell|}}{\partial x_1^{\ell_1}\cdots\partial x_n^{\ell_n}}g(\vecx)=O\big(|\vecx|^{n/2+1-|\ell|-\delta}\big).
\end{align}
(The implied constant depends only on $\Omega$ and $|\ell|$.)

Using \eqref{gder} we can now estimate the integrals in \eqref{ibp}.
If $n/2+1-|A|-\delta<0$ then
\begin{align}\label{aa}
\underset{\prod_{j\in A}[0,M_j]}{\int} |g_{A,N\setminus A}(\vecx)|\,d\vecx=O(1)\underset{\substack{|\vecx|\leq|\vecM|\\x_1,\ldots,x_{|A|}>0}}{\int}|\vecx|^{n/2+1-|A|-\delta}\,d\vecx
=O\big(|\vecM|^{n/2+1-\delta}\big),
\end{align}
and if $n/2+1-|A|-\delta\geq0$ then
\begin{align}\label{bb}
 \underset{\prod_{j\in A}[0,M_j]}{\int} |g_{A,N\setminus A}(\vecx)|\,d\vecx=O\big(|\vecM|^{n/2+1-|A|-\delta}\big)\underset{\prod_{j\in A}[0,M_j]}{\int}\,d\vecx=O\big(|\vecM|^{n/2+1-\delta}\big).
\end{align}
Combining \eqref{ibp}, \eqref{aa} and \eqref{bb} we get
\begin{align*}
 \sum_{m_1=0}^{M_1}\cdots\sum_{m_n=0}^{M_n}c_{m_1\vecom_1^*+\cdots+m_n\vecom_n^*}e^{2\pi i(m_1\alpha_1+\cdots+m_n\alpha_n)}
 =O\bigg(\frac{|\vecM|^{n/2+1-\delta}\big(\log(2|\vecM|)\big)^n}{1-\delta}\bigg).
\end{align*}
Finally we choose $\delta=1-\big(\log(3|\vecM|)\big)^{-1}$, and conclude that
\begin{align*}
 \sum_{m_1=0}^{M_1}\cdots\sum_{m_n=0}^{M_n}c_{m_1\vecom_1^*+\cdots+m_n\vecom_n^*}e^{2\pi i(m_1\alpha_1+\cdots+m_n\alpha_n)}
 =O\Big(|\vecM|^{n/2}\big(\log(2|\vecM|)\big)^{n+1}\Big),
\end{align*}
for all $\vecM=(M_1,\ldots,M_n)\in\Z_{\geq0}^n\setminus\{\vec0\}$ and all $(\alpha_1,\ldots,\alpha_n)\in\R^n$.
\end{proof}

\begin{remark}\label{remark}
 In the same way we find that for all $(\ve_1,\ldots,\ve_n)\in\{±1\}^n$, $\vecM=(M_1,\ldots,M_n)\in\Z_{\geq0}^n\setminus\{\vec0\}$ and $(\alpha_1,\ldots,\alpha_n)\in\R^n$,
 \begin{multline*}
  \sum_{m_1=0}^{M_1}\cdots\sum_{m_n=0}^{M_n}c_{\ve_1m_1\vecom_1^*+\cdots+\ve_nm_n\vecom_n^*}e^{2\pi i(\ve_1m_1\alpha_1+\cdots+\ve_nm_n\alpha_n)}\\=O\Big(|\vecM|^{n/2}\big(\log(2|\vecM|)\big)^{n+1}\Big),
 \end{multline*}
where the implied constant depends only on $\Gamma$, $\phi$ and $\Omega$.
\end{remark}

We continue to consider the cusp form $\phi$ described in Proposition \ref{Lemma1} having a Fourier expansion at infinity given by \eqref{smallcuspexp}. We are now ready to study the horosphere integral of $\phi$:
\begin{align*}
\mathcal{I}:=\frac{1}{\delta_1\cdots\delta_n}\int_{\R^n/\Z^n}\Psi_{\vecdel,\vecgam}(\vecu)\phi(u_1\vecom_1+\cdots+u_n\vecom_n+yi_n)\,d\vecu.
\end{align*}
Cf.\ \eqref{hoppasattdetsnartarklart}; here $\Psi_{\vecdel,\vecgam}$ is given by \eqref{psi}, where we assume that $\chi:\R^n\to\R$ is a smooth function of compact support. Using the expansions \eqref{chiseries} and \eqref{smallcuspexp} and integrating term by term yields
\begin{align}\label{Ifirst}
 \mathcal{I}
=\frac{1}{\delta_1\cdots\delta_n}\underset{\vecm\in\Z^n\setminus\{\vec0\}}{\sum}c_{\vecmu}y^{n/2}K_{s-n/2}(2\pi|\vecmu|y)\widehat{\chi_{\vecdel,\vecgam}}(-\vecm),
\end{align}
where in the term corresponding to $\vecm=(m_1,\ldots,m_n)\in \Z^n\setminus\{\vec0\}$ we have $\vecmu=m_1\vecom_1^*+\cdots+m_n\vecom_n^*$. For each nonempty $D\subset\{1,\ldots,n\}$ we define
\begin{align}\label{RD}
 R_D:=\big\{\vecm\in\Z^n\mid m_j\neq0\text{ iff } j\in D\big\}\subset\Z^n.
\end{align}
Hence we have
\begin{align}\label{chicuspint}
\mathcal{I}=\frac{1}{\delta_1\cdots\delta_n}\underset{\substack{D\subset\{1,\ldots,n\}\\D\neq\emptyset}}{\sum}\underset{\vecm\in R_D}{\sum}c_{\vecmu}y^{n/2}K_{s-n/2}(2\pi|\vecmu|y)\widehat{\chi_{\vecdel,\vecgam}}(-\vecm).
\end{align}

\begin{prop}\label{chicuspprop}
 Let $\chi:\R^n\to \R$ be a smooth function with compact support. Let $\ve>0$ and let $\phi$ be a cusp form with eigenvalue $\lambda$. Define $s$ via $\lambda=s(n-s)$, $s\in[n/2,n)\cup [n/2,n/2+i\infty)$. Then the following holds, uniformly over all $0<y<1$, all $\vecgam=(\gamma_1,\ldots,\gamma_n)\in\R^n$ and all $\delta_1,\ldots,\delta_n$ satisfying $0<\delta_1,\ldots,\delta_n\leq 1$:
\begin{multline*}
 \frac{1}{\delta_1\cdots\delta_n}\int_{\R^n}\chi_{\vecdel,\vecgam}(\vecu)\phi(u_1\vecom_1+\cdots+u_n\vecom_n+yi_n)\,d\vecu\\
 =O\Big(\|\chi\|_{n,1}y^{n-\Re s-\ve}\big(\underset{i\in\{1,\ldots,n\}}{\min}\delta_i\big)^{\Re s-n}\Big),
\end{multline*}
 where the implied constant depends only on $\Gamma$, $\phi$, $\Omega$ and $\ve$.
\end{prop}

\begin{proof}
We estimate each inner sum in \eqref{chicuspint} separately. That
is, given a nonempty
$D=\{j_1,\ldots,j_{|D|}\}\subset\{1,\ldots,n\}$ we want to
estimate
\begin{align*}
 S_D:=\frac{1}{\delta_1\cdots\delta_n}\underset{\vecm\in R_D}{\sum}c_{\vecmu}y^{n/2}K_{s-n/2}(2\pi|\vecmu|y)\widehat{\chi_{\vecdel,\vecgam}}(-\vecm).
\end{align*}
We let $R_D^+:=R_D\cap(\Z_{\geq0})^n$. It will be sufficient to estimate
\begin{align}\label{chiSD}
 & S_D^+:=\frac{1}{\delta_1\cdots\delta_n}\underset{\vecm\in R_D^+}{\sum}c_{\vecmu}y^{n/2}K_{s-n/2}(2\pi|\vecmu|y)\widehat{\chi_{\vecdel,\vecgam}}(-\vecm).
\end{align}
The remaining parts of $S_D$ can be estimated in the same way using Remark \ref{remark}.

We now fix a monotone function $\psi\in C^{\infty}(\R_{>0})$ satisfying
\begin{align*}
 \psi(x)=\begin{cases}
1 & \text{if $x\in(0,1]$}\\
0 & \text{if $x\in[2,\infty)$}
         \end{cases}
\end{align*}
and let $Y:=\underset{i\in\{1,\ldots,n\}}{\max}{\delta_i}^{-1}$. We split $S_D^+$ as
\begin{align}\label{SD12}
 S_D^+=S_D^1+S_D^2
\end{align}
with
\begin{align*}
S_D^1:=\frac{1}{\delta_1\cdots\delta_n}\underset{\vecm\in R_D^+}{\sum}\psi(Y^{-1}|\vecm|)c_{\vecmu}y^{n/2}K_{s-n/2}(2\pi|\vecmu|y)\widehat{\chi_{\vecdel,\vecgam}}(-\vecm)
\end{align*}
and
\begin{align*}
S_D^2:=\frac{1}{\delta_1\cdots\delta_n}\underset{\vecm\in R_D^+}{\sum}\big(1-\psi(Y^{-1}|\vecm|)\big)c_{\vecmu}y^{n/2}K_{s-n/2}(2\pi|\vecmu|y)\widehat{\chi_{\vecdel,\vecgam}}(-\vecm).
\end{align*}

Next we discuss some estimates that will be used repeatedly below. As in \eqref{gder} we have, for each multi-index $\ell=(\ell_1,\ldots,\ell_{|D|})$ of length $|\ell|$,
\begin{align}\label{chia}
 \frac{\partial^{|\ell|}}{\partial x_{j_1}^{\ell_1}\cdots\partial x_{j_{|D|}}^{\ell_{|D|}}}\big(|x_{j_1}\vecom_{j_1}^*+\cdots+x_{j_{|D|}}\vecom_{j_{|D|}}^*|\big)=O\big(|\vecx|^{1-|\ell|}\big).
\end{align}
We also need estimates of the K-bessel function. If $\sigma=\Re s$ we have
\begin{align}\label{chic}
K_{s-n/2}(u)=O(u^{n/2-\sigma-\ve}),\hspace{8pt} K_{s-n/2}^{(m)}(u)=O(u^{n/2-\sigma-m})
\end{align}
for $u>0$ and $m\geq 1$. For further details see \cite[pp.\ 77-80, 202]{watson}.

We first consider $S_D^1$. We note that
\begin{align}\label{sd1}
 S_D^1=\frac{1}{\delta_1\cdots\delta_n}\underset{\R^n}{\int}\Big(\underset{\vecm\in R_D^+}{\sum}c_{\vecmu}e^{2\pi i\langle\vecm,\vecu\rangle}\psi(Y^{-1}|\vecm|)y^{n/2}K_{s-n/2}(2\pi|\vecmu|y)\chi_{\vecdel,\vecgam}(\vecu)\Big)\,d\vecu.
\end{align}
We now fix $\vecu\in \R^n$ and estimate
\begin{align}\label{S1}
 \mathfrak{S}_D^1(\vecu):=\underset{\vecm\in R_D^+}{\sum}c_{\vecmu}e^{2\pi i\langle\vecm,\vecu\rangle}\psi(Y^{-1}|\vecm|)y^{n/2}K_{s-n/2}(2\pi|\vecmu|y)\chi_{\vecdel,\vecgam}(\vecu).
\end{align}
In order to apply summation by parts to \eqref{S1} we define
\begin{align}\label{afirst}
a(m_{j_1},\ldots,m_{j_{|D|}})
:=
c_{m_{j_1}\vecom_{j_1}^*+\cdots+m_{j_{|D|}}\vecom_{j_{|D|}}^*}e^{2\pi i(m_{j_1}u_{j_1}+\cdots+m_{j_{|D|}}u_{j_{|D|}})}
\end{align}
(recall that $c_{\vec0}=0$) and
\begin{align*}
 g(\vecx):=\psi(Y^{-1}|\vecx|)y^{n/2}K_{s-n/2}\big(2\pi|x_{j_1}\vecom_{j_1}^*+\cdots+x_{j_{|D|}}\vecom_{j_{|D|}}^*|y\big)\chi_{\vecdel,\vecgam}(\vecu),
\end{align*}
where $\vecx=(x_{j_1},\ldots,x_{j_{|D|}})$.
We also define
\begin{align}\label{Sfirst}
S(X_{j_1},\ldots,X_{j_{|D|}}):=\underset{0\leq m_{j_1}\leq X_{j_1}}{\sum}\cdots\underset{0\leq m_{j_{|D|}}\leq X_{j_{|D|}}}{\sum}a(m_{j_1},\ldots,m_{j_{|D|}}).
\end{align}
Applying formula \eqref{intbp*} we get
\begin{align}\label{chiSintbp}
 \mathfrak{S}_D^1(\vecu)=(-1)^{|D|}\underset{A\subset D}{\sum}\underset{\prod_{j\in A}[1,\infty)}{\int} g_{A,\emptyset}(\vecx)S_{A,\emptyset}(\vecx)\,d\vecx.
\end{align}
It follows from Proposition \ref{Lemma1} that $S_{A,\emptyset}(\vecx)=O(|\vecx|^{n/2+\ve})$. Here, and in all estimates in the rest of the proof, the implied constant depends only on $\Gamma$, $\phi$, $\Omega$ and $\ve$. We also have $S_{\emptyset,\emptyset}=c_{\vec0}=0$ and hence the corresponding term in \eqref{chiSintbp} is zero. It remains to consider nonempty $A$.

We now turn to the derivatives of $g$. Recall that the derivatives of $g$ in \eqref{chiSintbp} correspond to multi-indices of the type $\ell=(\ell_1,\ldots,\ell_{|D|})$ satisfying $\ell_i\leq1$. For each such $\ell$, $\frac{\partial^{|\ell|}}{\partial x_{j_1}^{\ell_1}\cdots\partial x_{j_{|D|}}^{\ell_{|D|}}}g(\vecx)$ is the (finite) sum of all expressions of the form
\begin{align}\label{chigderi}
&y^{n/2}\chi_{\vecdel,\vecgam}(\vecu)\frac{\partial^{|\ell^1|}}{\partial x_{j_1}^{\ell_1^1}\cdots\partial x_{j_{|D|}}^{\ell_{|D|}^1}}\big(\psi(Y^{-1}|\vecx|)\big)\\
&\hspace{50pt}\times\frac{\partial^{|\ell^2|}}{\partial x_{j_1}^{\ell_1^2}\cdots\partial x_{j_{|D|}}^{\ell_{|D|}^2}}\Big(K_{s-n/2}\big(2\pi|x_{j_1}\vecom_{j_1}^*+\cdots+x_{j_{|D|}}\vecom_{j_{|D|}}^*|y\big)\Big),\nonumber
\end{align}
where $\ell^1,\ell^2$ are multi-indices satisfying $\ell^1+\ell^2=\ell$. Using \eqref{chia} (but with $\{\vecom_{j_1}^*,\ldots,\vecom_{j_{|D|}}^*\}$ replaced with an orthonormal basis) and noting that all derivatives of $\psi(Y^{-1}|\vecx|)$ are zero except when $Y<|\vecx|<2Y$, we get
\begin{align}\label{psider}
 \frac{\partial^{|\ell^1|}}{\partial x_{j_1}^{\ell_1^1}\cdots\partial x_{j_{|D|}}^{\ell_{|D|}^1}}\big(\psi(Y^{-1}|\vecx|)\big)=O\big(|\vecx|^{-|\ell^1|}\big).
\end{align}
Applying \eqref{chia} and \eqref{chic} (modifying the second bound in \eqref{chic} into $K_{s-n/2}^{(m)}(u)=O(u^{n/2-\sigma-m-\ve})$; this being allowed since $K_{s-n/2}^{(m)}(u)$ anyway decays exponentially as $u\to\infty$) yields
\begin{align}\label{kder}
 \frac{\partial^{|\ell^2|}}{\partial x_{j_1}^{\ell_1^2}\cdots\partial x_{j_{|D|}}^{\ell_{|D|}^2}}\Big(K_{s-n/2}\big(2\pi|x_{j_1}\vecom_{j_1}^*+\cdots+x_{j_{|D|}}\vecom_{j_{|D|}}^*|y\big)\Big)=O\big(y^{n/2-\sigma-\ve}|\vecx|^{n/2-\sigma-|\ell^2|-\ve}\big).
\end{align}
Combining \eqref{chigderi}, \eqref{psider} and \eqref{kder} we obtain
\begin{align*}
 \frac{\partial^{|\ell|}}{\partial x_{j_1}^{\ell_1}\cdots\partial x_{j_{|D|}}^{\ell_{|D|}}}g(\vecx)=O\big(y^{n-\sigma-\ve}|\chi_{\vecdel,\vecgam}(\vecu)||\vecx|^{n/2-\sigma-|\ell|-\ve}\big).
\end{align*}

Returning to \eqref{chiSintbp} we have
\begin{multline*}
\underset{\prod_{j\in A}[1,\infty)}{\int} g_{A,\emptyset}(\vecx)S_{A,\emptyset}(\vecx)\,d\vecx\\
 =O\big(y^{n-\sigma-\ve}|\chi_{\vecdel,\vecgam}(\vecu)|\big)\underset{\substack{\prod_{j\in A}[1,\infty)\\|\vecx|\leq2Y}}{\int}|\vecx|^{n-\sigma-|A|}\,d\vecx
=O\big(y^{n-\sigma-\ve}Y^{n-\sigma}|\chi_{\vecdel,\vecgam}(\vecu)|\big),
\end{multline*}
and hence
\begin{align*}
 \mathfrak{S}_D^1(\vecu)=O\big(y^{n-\sigma-\ve}Y^{n-\sigma}|\chi_{\vecdel,\vecgam}(\vecu)|\big).
\end{align*}
Using this estimate in \eqref{sd1} we conclude that
\begin{align}\label{firstchi}
 S_D^1=O\Big(\frac{y^{n-\sigma-\ve}Y^{n-\sigma}}{\delta_1\cdots\delta_n}\Big)\int_{\R^n}|\chi_{\vecdel,\vecgam}(\vecu)|\,d\vecu
=O\Big(\|\chi\|_{0,1}y^{n-\sigma-\ve}\big(\underset{i\in\{1,\ldots,n\}}{\min}\delta_i\big)^{\sigma-n}\Big).
\end{align}
We stress that the implied constant neither depends on $\delta_1,\ldots,\delta_n$ nor on $\chi$.

Next we consider $S_D^2$. Applying Lemma \ref{partint} with $m=n$ we get
\begin{multline}\label{sd2}
 S_D^2=\frac{1}{\delta_1\cdots\delta_n}\Big(\frac{i}{2\pi}\Big)^n\int_{\R^n}\bigg(\underset{\vecm\in R_D^+}{\sum}c_{\vecmu}e^{2\pi i\langle\vecm,\vecu\rangle}\big(1-\psi(Y^{-1}|\vecm|)\big)y^{n/2}\\
\times K_{s-n/2}(2\pi|\vecmu|y)\frac{1}{|\vecm|^{2n}}\Big(m_1\frac{\partial}{\partial u_1}+\cdots+m_n\frac{\partial}{\partial u_n}\Big)^n\chi_{\vecdel,\vecgam}(\vecu)\bigg)\,d\vecu.
\end{multline}
We fix $\vecu\in \R^n$ and estimate
\begin{multline}\label{S2}
 \mathfrak{S}_D^2(\vecu):=\underset{\vecm\in R_D^+}{\sum}c_{\vecmu}e^{2\pi i\langle\vecm,\vecu\rangle}\big(1-\psi(Y^{-1}|\vecm|)\big)y^{n/2}\\
\times K_{s-n/2}(2\pi|\vecmu|y)\frac{1}{|\vecm|^{2n}}\Big(m_1\frac{\partial}{\partial u_1}+\cdots+m_n\frac{\partial}{\partial u_n}\Big)^n\chi_{\vecdel,\vecgam}(\vecu).
\end{multline}
In order to apply summation by parts to $\mathfrak{S}_D^2$ we
define
\begin{multline*}
g(\vecx):=\big(1-\psi(Y^{-1}|\vecx|)\big)y^{n/2}K_{s-n/2}\big(2\pi|x_{j_1}\vecom_{j_1}^*+\cdots+x_{j_{|D|}}\vecom_{j_{|D|}}^*|y\big)\\
\times\frac{1}{|\vecx|^{2n}}\Big(x_{j_1}\frac{\partial}{\partial
u_{j_1}}+\cdots+x_{j_{|D|}}\frac{\partial}{\partial
u_{j_{|D|}}}\Big)^n\chi_{\vecdel,\vecgam}(\vecu),
\end{multline*}
and let $a(m_{j_1},\ldots,m_{j_{|D|}})$ and $S(X_{j_1},\ldots,X_{j_{|D|}})$ be as above (cf.\ \eqref{afirst} and \eqref{Sfirst}). Using formula \eqref{intbp*} we get
\begin{align}\label{chiSintbp2}
 \mathfrak{S}_D^2(\vecu)=(-1)^{|D|}\underset{A\subset D}{\sum}\underset{\prod_{j\in A}[1,\infty)}{\int} g_{A,\emptyset}(\vecx)S_{A,\emptyset}(\vecx)\,d\vecx.
\end{align}
Again $S_{\emptyset,\emptyset}=0$ (which makes the corresponding
term in \eqref{chiSintbp2} zero) and for nonempty $A$ we have
$S_{A,\emptyset}(\vecx)=O(|\vecx|^{n/2+\ve})$.
For each multi-index $\ell$, $\frac{\partial^{|\ell|}}{\partial x_{j_1}^{\ell_1}\cdots\partial x_{j_{|D|}}^{\ell_{|D|}}}g(\vecx)$ is the (finite) sum of all expressions of the form
\begin{align}\label{chigderi2}
&y^{n/2}\frac{\partial^{|\ell^1|}}{\partial x_{j_1}^{\ell_1^1}\cdots\partial x_{j_{|D|}}^{\ell_{|D|}^1}}\big(1-\psi(Y^{-1}|\vecx|)\big)\\
&\times\frac{\partial^{|\ell^2|}}{\partial x_{j_1}^{\ell_1^2}\cdots\partial x_{j_{|D|}}^{\ell_{|D|}^2}}\Big(K_{s-n/2}\big(2\pi|x_{j_1}\vecom_{j_1}^*+\cdots+x_{j_{|D|}}\vecom_{j_{|D|}}^*|y\big)\Big)\nonumber\\
&\times\frac{\partial^{|\ell^3|}}{\partial x_{j_1}^{\ell_1^3}\cdots\partial x_{j_{|D|}}^{\ell_{|D|}^3}}\Big(\frac{1}{|\vecx|^{2n}}\Big)\frac{\partial^{|\ell^4|}}{\partial x_{j_1}^{\ell_1^4}\cdots\partial x_{j_{|D|}}^{\ell_{|D|}^4}}\Big(\Big(x_{j_1}\frac{\partial}{\partial
u_{j_1}}+\cdots+x_{j_{|D|}}\frac{\partial}{\partial
u_{j_{|D|}}}\Big)^n\chi_{\vecdel,\vecgam}(\vecu)\Big),\nonumber
\end{align}
where $\ell^1,\ell^2,\ell^3,\ell^4$ are multi-indices satisfying $\ell^1+\ell^2+\ell^3+\ell^4=\ell$. Bounds for the first two derivatives in \eqref{chigderi2} are given by \eqref{psider} and \eqref{kder} respectively.
Using \eqref{chia} (slightly modified) we also find that
\begin{align}\label{moreder}
 \frac{\partial^{|\ell^3|}}{\partial x_{j_1}^{\ell_1^3}\cdots\partial x_{j_{|D|}}^{\ell_{|D|}^3}}\Big(\frac{1}{|\vecx|^{2n}}\Big)=O\Big(\frac{1}{|\vecx|^{2n+|\ell^3|}}\Big).
\end{align}
Recalling that we only consider multi-indices of length $|\ell_4|\leq|\ell|\leq n$ as well as the definition \eqref{chidelta} we get
\begin{align}\label{lastder}
&\frac{\partial^{|\ell^4|}}{\partial x_{j_1}^{\ell_1^4}\cdots\partial x_{j_{|D|}}^{\ell_{|D|}^4}}\Big(\Big(x_{j_1}\frac{\partial}{\partial
u_{j_1}}+\cdots+x_{j_{|D|}}\frac{\partial}{\partial
u_{j_{|D|}}}\Big)^n\chi_{\vecdel,\vecgam}(\vecu)\Big)\nonumber\\
&=O\Big(|\vecx|^{n-|\ell^4|}\underset{|\vecal|=n}{\max}\Big|\frac{\partial^{|\vecal|}}{\partial u_{j_1}^{\alpha_1}\cdots\partial u_{j_{|D|}}^{\alpha_{|D|}}}\chi_{\vecdel,\vecgam}(\vecu)\Big|\Big)\\
&=O\Big(|\vecx|^{n-|\ell^4|}Y^n\underset{|\vecal|=n}{\max}\Big|\Big(\frac{\partial^{|\vecal|}}{\partial u_{j_1}^{\alpha_1}\cdots\partial u_{j_{|D|}}^{\alpha_{|D|}}}\chi\Big)\big(\sfrac{u_1-\gamma_1}{\delta_1},\ldots,\sfrac{u_n-\gamma_n}{\delta_n}\big)\Big|\Big).\nonumber
\end{align}
Combining \eqref{chigderi2} with \eqref{psider}, \eqref{kder}, \eqref{moreder} and \eqref{lastder} we obtain
\begin{multline*}
\frac{\partial^{|\ell|}}{\partial x_{j_1}^{\ell_1}\cdots\partial x_{j_{|D|}}^{\ell_{|D|}}}g(\vecx)\\
=O\Big(y^{n-\sigma-\ve}Y^n|\vecx|^{-n/2-\sigma-|\ell|-\ve}\underset{|\vecal|=n}{\max}\Big|\Big(\frac{\partial^{|\vecal|}}{\partial u_{j_1}^{\alpha_1}\cdots\partial u_{j_{|D|}}^{\alpha_{|D|}}}\chi\Big)\big(\sfrac{u_1-\gamma_1}{\delta_1},\ldots,\sfrac{u_n-\gamma_n}{\delta_n}\big)\Big|\Big).
\end{multline*}

Returning to \eqref{chiSintbp2} we have
\begin{align*}
&\underset{\prod_{j\in A}[1,\infty)}{\int} g_{A,\emptyset}(\vecx)S_{A,\emptyset}(\vecx)\,d\vecx\\
 &=O\Big(y^{n-\sigma-\ve}Y^n\underset{|\vecal|=n}{\max}\Big|\Big(\frac{\partial^{|\vecal|}}{\partial u_{j_1}^{\alpha_1}\cdots\partial u_{j_{|D|}}^{\alpha_{|D|}}}\chi\Big)\big(\sfrac{u_1-\gamma_1}{\delta_1},\ldots,\sfrac{u_n-\gamma_n}{\delta_n}\big)\Big|\Big)\underset{\substack{\prod_{j\in A}[1,\infty)\\|\vecx|\geq Y}}{\int}|\vecx|^{-\sigma-|A|}\,d\vecx\\
&=O\Big(y^{n-\sigma-\ve}Y^{n-\sigma}\underset{|\vecal|=n}{\max}\Big|\Big(\frac{\partial^{|\vecal|}}{\partial u_{j_1}^{\alpha_1}\cdots\partial u_{j_{|D|}}^{\alpha_{|D|}}}\chi\Big)\big(\sfrac{u_1-\gamma_1}{\delta_1},\ldots,\sfrac{u_n-\gamma_n}{\delta_n}\big)\Big|\Big),
\end{align*}
and it follows that 
\begin{align*}
 \mathfrak{S}_D^2(\vecu)=O\Big(y^{n-\sigma-\ve}Y^{n-\sigma}\underset{|\vecal|=n}{\max}\Big|\Big(\frac{\partial^{|\vecal|}}{\partial u_{j_1}^{\alpha_1}\cdots\partial u_{j_{|D|}}^{\alpha_{|D|}}}\chi\Big)\big(\sfrac{u_1-\gamma_1}{\delta_1},\ldots,\sfrac{u_n-\gamma_n}{\delta_n}\big)\Big|\Big).
\end{align*}
Using this bound in \eqref{sd2} we conclude that
\begin{align*}
 S_D^2&=O\Big(\frac{y^{n-\sigma-\ve}Y^{n-\sigma}}{\delta_1\cdots\delta_n}\Big)\int_{\R^n}\underset{|\vecal|=n}{\max}\Big|\Big(\frac{\partial^{|\vecal|}}{\partial u_{j_1}^{\alpha_1}\cdots\partial u_{j_{|D|}}^{\alpha_{|D|}}}\chi\Big)\big(\sfrac{u_1-\gamma_1}{\delta_1},\ldots,\sfrac{u_n-\gamma_n}{\delta_n}\big)\Big|\,d\vecu\nonumber\\
&=O\Big(\|\chi\|_{n,1}y^{n-\sigma-\ve}\big(\underset{i\in\{1,\ldots,n\}}{\min}\delta_i\big)^{\sigma-n}\Big).
\end{align*}
Together with \eqref{firstchi} this proves the proposition.
\end{proof}

We continue with a result that shows how to keep control of the dependence on the eigenvalue in the horosphere integral for cusp forms with $\lambda\geq(\sfrac{n}{2})^2$. We use the Rankin-Selberg bound from Section \ref{RSsec}. However, this method gives weaker $y,\delta$-exponents than the previous proposition.

\begin{prop}\label{cuspformintprop}
Let $\chi:\R^n\to \R$ be a smooth function with compact support. Let $\ve>0$ and let $\phi$ be a cusp form on $\Gamma\setminus\H^{n+1}$ with eigenvalue $\lambda=(\sfrac{n}{2})^2+T^2$, $T\geq0$, normalized so that $\int_{\F}|\phi(P)|^2\,d\nu(P)=1$. Then the following holds, uniformly over all $0<y<1$, all $\vecgam=(\gamma_1,\ldots,\gamma_n)\in\R^n$ and all $\delta_1,\ldots,\delta_n$ satisfying $0<\delta_1,\ldots,\delta_n\leq 1$:
\begin{multline*}
 \frac{1}{\delta_1\cdots\delta_n}\int_{\R^n}\chi_{\vecdel,\vecgam}(\vecu)\phi(u_1\vecom_1+\cdots+u_n\vecom_n+yi_n)\,d\vecu\\
 =O\Big(\|\chi\|_{n,1}(T+1)^{1/6+\ve}y^{n/2-\ve}\big(\underset{i\in\{1,\ldots,n\}}{\min}\delta_i\big)^{-n}\Big),
\end{multline*}
 where the implied constant depends only on $\Gamma$, $\Omega$ and $\ve$.
\end{prop}

\begin{proof}
 It is enough to consider $0<\ve<\sfrac{1}{2}$. Given such an $\ve$ we have the following bound on the K-bessel function (cf.\ \cite{bal} and \cite[pp.\ 525-526]{andreas}), which is uniform for all $T\geq 0$ and $t>0$:
 \begin{align}\label{kbesselestim}
  K_{iT}(t)=O\Big(e^{-(\pi/2)T}(T+1)^{-1/3+\ve}t^{-\ve}\min\big(1,e^{(\pi/2)T-t}\big)\Big).
 \end{align}
Applying this bound and (a weak version of) \eqref{transfestim} with $k=n$ in equation \eqref{Ifirst} we get
\begin{multline}\label{Iest}
\mathcal{I}=O\Big(\|\chi\|_{n,1}y^{n/2-\ve}Y^ne^{-(\pi/2)T}(T+1)^{-1/3+\ve}\Big)\\
\times\sum_{\vec0\neq\vecmu\in\Lambda_1^*}|c_{\vecmu}||\vecmu|^{-n-\ve}\min\big(1,e^{(\pi/2)T-2\pi|\vecmu|y}\big).
\end{multline}
(Recall that $Y=\underset{i\in\{1,\ldots,n\}}{\max}\delta_i^{-1}$.) We call the sum in the second line above $\Sigma$. In order to estimate $\Sigma$ we introduce the functions
\begin{align*}
f(X):=X^{-n-\ve}\min\big(1,e^{(\pi/2)T-2\pi yX}\big)
\end{align*}
and
\begin{align*}
 S(X):=\underset{0<|\vecmu|\leq X}{\sum}|c_{\vecmu}|.
\end{align*}
Using the Cauchy-Schwarz inequality and Proposition \ref{cusprankin} we obtain
\begin{align}\label{SB}
 S(X)=O\bigg(e^{(\pi/2)T}\Big(T+\frac{X^n}{(T+1)^{n-1}}\Big)^{1/2}X^{n/2}\bigg)
\end{align}
for all $X\geq\sfrac{\mu_0}{2}$. We further note that for fixed $T\geq0$ and $0<y<1$ the function $f$ is continuous and piecewise smooth. 

Now, by integration by parts,
\begin{align*}
 \Sigma=\int_{\frac{\mu_0}{2}}^{\infty}f(X)\,dS(X)
 =-\int_{\frac{\mu_0}{2}}^{\infty}f'(X)S(X)\,dX.
\end{align*}
Using the straightforward bound
\begin{align}\label{fder}
 f'(X)=\begin{cases}
O\big(X^{-n-1-\ve}\big) & \text{if $X<\frac{T}{4y}$}\\
O\big((1+yX)X^{-n-1-\ve}e^{(\pi/2)T-2\pi yX}\big) &
\text{if $X>\frac{T}{4y}$}
       \end{cases}
\end{align}
together with \eqref{SB} we get 
\begin{align}\label{intermediatebound}
 \Sigma&
 =O\big(e^{(\pi/2)T}\big)\bigg\{\int_{\frac{\mu_0}{2}}^{\max(\frac{\mu_0}{2},\frac{T}{4y})}X^{-n/2-1-\ve}\Big(T+\frac{X^n}{(T+1)^{n-1}}\Big)^{1/2}\,dX\\
 &+\int_{\max(\frac{\mu_0}{2},\frac{T}{4y})}^{\infty}(1+yX)X^{-n/2-1-\ve}\Big(T+\frac{X^n}{(T+1)^{n-1}}\Big)^{1/2}e^{(\pi/2)T-2\pi yX}\,dX\bigg\}.\nonumber
\end{align}
Notice that since $\big(T+\frac{X^n}{(T+1)^{n-1}}\big)^{1/2}\leq
T^{1/2}+X^{n/2}(T+1)^{(1-n)/2}$ the first integral in
\eqref{intermediatebound} is bounded by $O\big((T+1)^{1/2}\big)$.
When $X>\max(\frac{\mu_0}{2},\frac{T}{4y})$ we have $\big(T+\frac{X^n}{(T+1)^{n-1}}\big)^{1/2}=O(X^{n/2})$ and hence the second integral in \eqref{intermediatebound} is bounded by
\begin{align}\label{intb'}
O(y^{\ve})\int_{\max(\frac{\mu_0y}{2},\frac{T}{4})}^{\infty}(1+u^{-1})u^{-\ve}e^{(\pi/2)T-2\pi u}\,du=O(1).
\end{align}
Hence $\Sigma=O\big(e^{(\pi/2)T}(T+1)^{1/2}\big)$ and using this estimate in \eqref{Iest} gives the desired result.
\end{proof}

Note that Proposition \ref{cuspformintprop} in particular holds when $\phi=\phi_m$ with $\lambda_m\geq(\sfrac{n}{2})^2$. We next collect a bound on the contribution from all such cusp forms $\phi_m$ to the horosphere integral \eqref{hoppasattdetsnartarklart} of a general test function $f$.

\begin{prop}\label{cuspimportant}
 Let $\chi:\R^n\to \R$ be a smooth function with compact support. Let $\ve>0$, $k>\frac{n+1}{2}+\frac{1}{6}$ and $f\in H^k(\Gamma\setminus\H^{n+1})$. Let the spectral expansion of $f$ be
\begin{align}\label{spectraldec}
 f(P)=\sum_{m\geq0}c_m\phi_m(P) + \sum_{\ell=1}^{\kappa}\int_{0}^{\infty}g_{\ell}(t)E_{\ell}\big(P,\sfrac{n}{2}+it\big)\,dt.
\end{align}
Then the following holds, for all $0<y<1$, all $\vecgam=(\gamma_1,\ldots,\gamma_n)\in\R^n$ and all $\delta_1,\ldots,\delta_n$ satisfying $0<\delta_1,\ldots,\delta_n\leq 1$:
\begin{multline}\label{important1}
 \frac{1}{\delta_1\cdots\delta_n}\int_{\R^n}\chi_{\vecdel,\vecgam}(\vecu)\Big\{\sum_{\lambda_m\geq(\frac{n}{2})^2}c_m\phi_m(u_1\vecom_1+\cdots+u_n\vecom_n+yi_n)\Big\}\,d\vecu\\
=O\Big(\|f\|_{H^k}\|\chi\|_{n,1}y^{n/2-\ve}\big(\underset{i\in\{1,\ldots,n\}}{\min}\delta_i\big)^{-n}\Big),
 \end{multline}
where the implied constant depends only on $\Gamma$, $\Omega$, $k$ and $\ve$.
\end{prop}

\begin{proof}
 It suffices to consider $0<\ve<k-\frac{n+1}{2}-\frac{1}{6}$. Let $\lambda_m=(\sfrac{n}{2})^2+r_m^2$. Changing order of integration and summation and using the Cauchy-Schwarz inequality we get
\begin{align*}
 &\bigg|\frac{1}{\delta_1\cdots\delta_n}\int_{\R^n}\chi_{\vecdel,\vecgam}(\vecu)\Big\{\sum_{\lambda_m\geq(\frac{n}{2})^2}c_m \phi_m(u_1\vecom_1+\cdots+u_n\vecom_n+yi_n)\Big\}\,d\vecu\bigg|\\
&\leq
 \Big(\sum_{\lambda_m\geq(\frac{n}{2})^2}|c_m|^2(r_m+1)^{2k}\Big)^{1/2}\\
&\times\bigg(\sum_{\lambda_m\geq(\frac{n}{2})^2}\bigg|\frac{1}{(r_m+1)^{k}\delta_1\cdots\delta_n}\int_{\R^n}\chi_{\vecdel,\vecgam}(\vecu)
\phi_m(u_1\vecom_1+\cdots+u_n\vecom_n+yi_n)\,d\vecu\bigg|^2\bigg)^{1/2}.
\end{align*}
By definition 
\begin{align*}
 \Big(\sum_{\lambda_m\geq(\frac{n}{2})^2}|c_m|^2(r_m+1)^{2k}\Big)^{1/2}\leq\|f\|_{H^k},
\end{align*}
and by Proposition \ref{cuspformintprop} the second factor above is bounded by
\begin{align*}
 O\Big(\|\chi\|_{n,1}y^{n/2-\ve}\big(\underset{i\in\{1,\ldots,n\}}{\min}\delta_i\big)^{-n}\Big)\Big\{\sum_{\lambda_m\geq(\frac{n}{2})^2}(r_m+1)^{-2k+1/3+2\ve} \Big\}^{1/2}.
\end{align*}
Finally we use Weyl's law (cf. \cite[Thm.\ 7.33]{cs}) to conclude
that this last sum is convergent (since $k>\frac{n+1}{2}+\frac{1}{6}+\ve$).
\end{proof}

\subsection{The horosphere integral for non-cuspidal eigenfunctions}\label{noncuspchap}

The treatment of a non-cuspidal eigenfunction $\phi$ is considerably more involved than the previous treatment of cusp forms, due to the fact that $\phi$ is no longer uniformly bounded throughout $\H^{n+1}$. Our central technical result is the following bound on linear combinations of Fourier coefficients of $\phi$ of the same type as in Proposition \ref{Lemma1}.

\begin{thm}\label{Lemma1'}
 Let $\ve>0$ and let $\phi$ be a non-cuspidal eigenfunction with eigenvalue $\lambda>0$. Define $s$ via $\lambda=s(n-s)$, $s\in(n/2,n)$. Let the Fourier expansion of $\phi$ at the cusp $\eta_1=\infty$ be
 \begin{align}\label{noncusp}
  \phi(\vecx+yi_n)=c_{\vec0}y^{n-s}+\sum_{\vec0\neq\vecmu\in\Lambda_1^*}c_{\vecmu}y^{n/2}K_{s-n/2}(2\pi|\vecmu|y)e^{2\pi i\langle\vecmu,\vecx\rangle}.
 \end{align}
Given $\vecM=(M_1,\ldots,M_n)\in\Z_{\geq0}^n\setminus\{\vec0\}$ we order the coordinates as $M_{j_1}\geq\ldots\geq M_{j_n}$. Then the following holds, uniformly over $\vecM=(M_1,\ldots,M_n)\in\Z_{\geq0}^n\setminus\{\vec0\}$ and $(\alpha_1,\ldots,\alpha_n)\in\R^n$:
 \begin{multline}\label{lincomb'}
  \sum_{m_1=0}^{M_1}\cdots\sum_{m_n=0}^{M_n}c_{m_1\vecom_1^*+\cdots+m_n\vecom_n^*}e^{2\pi i(m_1\alpha_1+\cdots+m_n\alpha_n)}\\=O\Big(|\vecM|^{n/2+\ve}(M_{j_{i_0}}+1)^{i_0-s}\prod_{k=i_0+1}^n(M_{j_k}+1)\Big),
 \end{multline}
where $i_0$ is the smallest integer $>s$ (viz., $i_0=s+1$ if $s\in\Z$, otherwise $i_0=\lceil s\rceil$). The implied constant depends only on $\Gamma$, $\phi$, $\Omega$ and $\ve$.
\end{thm}

\begin{remark}
If $n=1$ then we always get $i_0=1$, and hence \eqref{lincomb'} says 
\begin{align*}
 \sum_{m=0}^{M}c_me^{2\pi im\alpha}=O\big(M^{3/2-s+\ve}\big),\qquad\forall M\geq1,
\end{align*}
which agrees with \cite[Prop.\ 5.1]{andreas} except for the extra $\ve$ in the exponent. If $n=2$ then we always get $i_0=2$, and hence \eqref{lincomb'} says
\begin{align*}
&\sum_{m_1=0}^{M_1}\sum_{m_2=0}^{M_2}c_{m_1\vecom_1^*+m_2\vecom_2^*}e^{2\pi i(m_1\alpha_1+m_2\alpha_2)}\\
&=O\Big(\max(M_1,M_2)^{1+\ve}\big(\min(M_1,M_2)+1\big)^{2-s}\Big),\qquad\forall(M_1,M_2)\in\Z_{\geq0}^2\setminus\{\vec0\}.
\end{align*}
\end{remark}

The proof of Theorem \ref{Lemma1'} will occupy the next 10 pages.

\begin{proof}[Proof of Theorem \ref{Lemma1'}]
 We let $I_{\vecal}:=\prod_{i=1}^n[\alpha_i-\sfrac{1}{2},\alpha_i+\sfrac{1}{2}]$ and $\vecu:=(u_1,\ldots,u_n)$. For varying $\vecM=(M_1,\ldots,M_n)\in\Z_{\geq0}^n\setminus\{\vec0\}$, $(\alpha_1,\ldots,\alpha_n)\in\R^n$ and $\delta\in(0,1)$ satisfying $n>s+1-\delta$ and $\delta-s\notin\Z$ we let
\begin{multline*}
 \widehat{\mathcal{K}}:=\int_0^{\infty}\int_{I_{\vecal}}\phi(u_1\vecom_1+\cdots+u_n\vecom_n+yi_n)\\
\times\bigg(\underset{m_1+\cdots+m_n>0}{\sum_{m_1=0}^{M_1}\cdots\sum_{m_n=0}^{M_n}}e^{2\pi i(m_1(\alpha_1-u_1)+\cdots+m_n(\alpha_n-u_n))}\bigg)\,\frac{du_1\ldots du_ndy}{y^{\delta}}.
\end{multline*}
This integral is not absolutely convergent because of the first term in the expansion \eqref{noncusp}. However, using \eqref{noncusp} we see that for each $y>0$ the inner integral above equals
\begin{align}\label{rearrange}
\underset{m_1+\cdots+m_n>0}{\sum_{m_1=0}^{M_1}\cdots\sum_{m_n=0}^{M_n}}c_{m_1\vecom_1^*+\cdots+m_n\vecom_n^*}y^{n/2}K_{s-n/2}\big(2\pi|m_1\vecom_1^*+\cdots+m_n\vecom_n^*|y\big)e^{2\pi i(m_1\alpha_1+\cdots+m_n\alpha_n)},
\end{align}
which is a finite sum. Hence, by the estimate \eqref{chic} and the exponential decay of the K-Bessel function at infinity (cf.\ \cite[p.\ 202(1)]{watson}), we find that
\begin{multline}\label{finite}
 \int_0^{\infty}\bigg|\int_{I_{\vecal}}\phi(u_1\vecom_1+\cdots+u_n\vecom_n+yi_n)\\
\times\bigg(\underset{m_1+\cdots+m_n>0}{\sum_{m_1=0}^{M_1}\cdots\sum_{m_n=0}^{M_n}}e^{2\pi i(m_1(\alpha_1-u_1)+\cdots+m_n(\alpha_n-u_n))}\bigg)\,d\vecu\bigg|\frac{dy}{y^{\delta}}<\infty.
\end{multline}
Furthermore, using \eqref{rearrange} and \cite[p.\ 388(8)]{watson}, we obtain 
\begin{multline}\label{khat}
 \widehat{\mathcal{K}}=\bigg(\underset{m_1+\cdots+m_n>0}{\sum_{m_1=0}^{M_1}\cdots\sum_{m_n=0}^{M_n}}c_{m_1\vecom_1^*+\cdots+m_n\vecom_n^*}|m_1\vecom_1^*+\cdots+m_n\vecom_n^*|^{\delta-n/2-1}e^{2\pi i(m_1\alpha_1+\cdots+m_n\alpha_n)}\bigg)\\
\times\frac{\pi^{\delta-n/2-1}}{4}\Gamma\Big(\frac{n+1}{2}-\frac{\delta+s}{2}\Big)\Gamma\Big(\frac{1+s}{2}-\frac{\delta}{2}\Big).
\end{multline}

The goal is now to estimate $\widehat{\mathcal{K}}$ from above. For every $k\in\{1,...,\kappa\}$, we obtain from the Fourier expansion of $\phi(P)$ at $\eta_k$ that
\begin{align}\label{fi1}
 \phi(P)=c_{\vec0}^{(k)}(y_{A_k(P)})^{n-s}+O\big(e^{-\pi \mu_0 y_{A_k(P)}}\big)\qquad\text{ as }\,\,y_{A_k(P)}\to\infty.
\end{align}
Here $c_{\vec0}^{(1)}=c_{\vec0}$. Furthermore we know that $\phi$ is bounded on $\mathcal{F}_B$ for any $B\geq B_0$. We now use the $\Gamma$-invariance of $\phi$, Lemma \ref{estim} and the notation
\begin{eqnarray*}
 \lfloor t \rfloor_M:=\begin{cases}
                                  t & \textrm{ if $t\geq M$}\\
                  0 & \textrm{ otherwise,}
                                 \end{cases}
\end{eqnarray*}
to conclude that
\begin{align}\label{fi2}
 \phi(P)=O\big(\mathcal{Y}_{\Gamma}(P)^{n-s}\big)=O\big(1+\lfloor\mathcal{Y}_{\Gamma}(P)\rfloor_1^{n-s}\big) \qquad\forall P\in\H^{n+1}.
\end{align}
For $y\geq1$ we substitute \eqref{fi1}, with $k=1$, into the definition of $\widehat{\mathcal{K}}$ and for $y\leq1$ we use \eqref{fi2} in the same definition. This results in the following estimate:
\begin{align*}
 |\widehat{\mathcal{K}}|&\leq\int_1^{\infty}O(e^{-\pi \mu_0 y})\,\frac{dy}{y^{\delta}}\bigg(1+\prod_{j=1}^n\int_{\alpha_j-1/2}^{\alpha_j+1/2}\Big|\sum_{m_j=0}^{M_j}e^{2\pi im_j(\alpha_j-u_j)}\Big|\,du_j\bigg)\\
 &+\int_0^{1}\int_{I_{\vecal}}O\Big(1+\lfloor\mathcal{Y}_{\Gamma}(u_1\vecom_1+\cdots+u_n\vecom_n+yi_n)\rfloor_1^{n-s}\Big)\\
&\times\bigg(1+\prod_{j=1}^n\Big|\sum_{m_j=0}^{M_j}e^{2\pi im_j(\alpha_j-u_j)}\Big|\bigg)\,\frac{du_1\ldots du_ndy}{y^{\delta}}.
\end{align*}
We note, by e.g.\ using the formula for a geometric sum, that $\Big|\sum_{m_j=0}^{M_j}e^{2\pi im_j(\alpha_j-u_j)}\Big|=O\big(\min(M_j+1,|u_j-\alpha_j|^{-1})\big)$ for all $u_j\in[\alpha_j-1/2,\alpha_j+1/2]$. Hence
\begin{align*}
 \int_{\alpha_j-1/2}^{\alpha_j+1/2}\Big|\sum_{m_j=0}^{M_j}e^{2\pi im_j(\alpha_j-u_j)}\Big|\,du_j=O\big(\log(2(M_j+1))\big)
\end{align*}
for all $j\in\{1,\ldots,n\}$, and it follows that
\begin{multline}\label{Khatestim}
 |\widehat{\mathcal{K}}|\leq O\Big(\big(\log(2|\vecM|)\big)^n\Big)
 +O(1)\int_0^{1}\int_{I_{\vecal}}\lfloor\mathcal{Y}_{\Gamma}(u_1\vecom_1+\cdots+u_n\vecom_n+yi_n)\rfloor_1^{n-s}\\
\times\prod_{j=1}^n\min\big(M_j+1,|u_j-\alpha_j|^{-1}\big)\,\frac{du_1\ldots du_ndy}{y^{\delta}}.
\end{multline}
Here, and in all estimates in the rest of the proof, the implied constants depend only on $\Gamma$, $\phi$, $\Omega$ and $\delta$. 

We now use the definition of $\mathcal{Y}_{\Gamma}(P)$ to see that the integral in \eqref{Khatestim} is bounded from above by
\begin{multline}\label{readytoestimate}
\sum_{k=1}^{\kappa}\sum_{W\in\Gamma_{\eta_k}\backslash\Gamma}\int_0^{1}\int_{I_{\vecal}}\lfloor y_{A_kW(u_1\vecom_1+\cdots+u_n\vecom_n+yi_n)}\rfloor_1^{n-s}\\\times\prod_{j=1}^n\min\big(M_j+1,|u_j-\alpha_j|^{-1}\big)\,\frac{du_1\ldots du_ndy}{y^{\delta}}.
\end{multline}
We study each term in this double sum separately. In order to do so we fix some $k\in\{1,...,\kappa\}$ and $W\in\Gamma_{\eta_k}\backslash\Gamma$ which give a nonzero contribution to \eqref{readytoestimate}. Hence $y_{A_kW(u_1\vecom_1+\cdots+u_n\vecom_n+yi_n)}\geq1$ for some $(u_1,\ldots,u_n,y)\in {I_{\vecal}}^{\circ}\times(0,1)$ and we conclude that $A_kW\notin\Gammainfty$. We write $A_kW=\matr{a}{b}{c}{d}$ and use Lemma \ref{shim} to see that $|c|\geq1/\sqrt{\tau_1\tau_k}$. Now $y_{A_kW(u_1\vecom_1+\cdots+u_n\vecom_n+yi_n)}\geq1$ implies that $u_1\vecom_1+\cdots+u_n\vecom_n+yi_n$ belongs to the horoball tangent to $\R^n$ at $-c^{-1}d$ with Euclidean radius $R:=1/(2|c|^2)$. In particular we have $|u_1\vecom_1+\cdots+u_n\vecom_n+c^{-1}d|\leq R\leq\frac{\tau_1\tau_k}{2}$. This, together with
$|u_j-\alpha_j|<1/2$ for all $j\in\{1,\ldots,n\}$, yields
\begin{multline*}
 \big|\alpha_1\vecom_1+\cdots+\alpha_n\vecom_n+c^{-1}d\big|\leq\big|(\alpha_1-u_1)\vecom_1+\cdots+(\alpha_n-u_n)\vecom_n\big|\\+\big|u_1\vecom_1+\cdots+u_n\vecom_n+c^{-1}d\big|
 <\frac{1}{2}\Big(\sum_{i=1}^n|\vecom_i|+\tau_1\tau_k\Big).
\end{multline*}
Thus we get a nonzero contribution in \eqref{readytoestimate} from $W=A_k^{-1}\matr{a}{b}{c}{d}\in\Gamma_{\eta_k}\backslash\Gamma$ only if $|c|\geq1/\sqrt{\tau_1\tau_k}$ and $-c^{-1}d$ belongs to the open Euclidean ball with centre $\alpha_1\vecom_1+\cdots+\alpha_n\vecom_n$ and radius
\begin{align}\label{rk}
\widetilde R_k:=\frac{1}{2}\Big(\sum_{i=1}^n|\vecom_i|+\tau_1\tau_k\Big).
\end{align}

We write $c^{-1}d=\sum_{j=1}^ne_j\vecom_j$ and let $\alpha_j':=\alpha_j+e_j$ for $j\in\{1,\ldots,n\}$. We note that there exist $A,B>0$ such that $A|\vecu|\leq\big|u_1\vecom_1+\cdots+u_n\vecom_n\big|\leq B|\vecu|$ holds for all $\vecu\in\R^n$ and recall from \eqref{ycoord} that
\begin{align*}
y_{A_kW(u_1\vecom_1+\cdots+u_n\vecom_n+yi_n)}=\frac{y}{|c|^2\big(|u_1\vecom_1+\cdots+u_n\vecom_n+c^{-1}d|^2+y^2\big)}.
\end{align*}
Changing variables, $u_j'=u_j+e_j$, in \eqref{readytoestimate} and using $\big|u_j'-\alpha_j'\big|\geq\big||u_j'|-|\alpha_j'|\big|$, we get 
\begin{align}\label{firststep}
 &\int_0^{1}\int_{I_{\vecal}}\lfloor y_{A_kW(u_1\vecom_1+\cdots+u_n\vecom_n+yi_n)}\rfloor_1^{n-s}
 \prod_{j=1}^n\min\big(M_j+1,|u_j-\alpha_j|^{-1}\big)\,\frac{du_1\ldots du_ndy}{y^{\delta}}\\
 &\leq|c|^{2(s-n)}\int_0^{2R}\underset{[-R/A,R/A]^n}{\int}\Big(\frac{y}{|u_1'\vecom_1+\cdots+u_n'\vecom_n|^2+y^2}\Big)^{n-s}\nonumber\\
&\hspace{145pt}\times\prod_{j=1}^n\min\big(M_j+1,\big||u_j'|-|\alpha_j'|\big|^{-1}\big)\,\frac{du_1'\ldots du_n'dy}{y^{\delta}}.\nonumber
\end{align}

We now fix $y\in(0,2R]$ and study the inner integral in \eqref{firststep}. For $m\geq1$ we note that when $|u_1'\vecom_1+\cdots+u_n'\vecom_n|\geq2^{m-1}y$ we have
\begin{align}\label{m1}
 \frac{y}{|u_1'\vecom_1+\cdots+u_n'\vecom_n|^2+y^2}=O\Big(\frac{1}{2^{2m}y}\Big).
\end{align}
Furthermore, we always have
\begin{align}\label{m0}
\frac{y}{|u_1'\vecom_1+\cdots+u_n'\vecom_n|^2+y^2}=O\Big(\frac{1}{y}\Big).
\end{align}
Using \eqref{m1} and \eqref{m0} we see that the inner integral in \eqref{firststep} is bounded by
\begin{align}\label{estim1}
&O(1)\underset{1\leq2^m< 2R/y}{\sum}\hspace{10pt}\underset{[-2^my/A,2^my/A]^n}{\int}\Big(\frac{1}{2^{2m}y}\Big)^{n-s}\prod_{j=1}^n\min\big(M_j+1,\big||u_j'|-|\alpha_j'|\big|^{-1}\big)\,d\vecu'\\
&=O(1)\underset{1\leq2^m< 2R/y}{\sum}2^{2m(s-n)}y^{s-n}\prod_{j=1}^n\int_0^{2^{m}y/A}\min\big(M_j+1,\big|u_j'-|\alpha_j'|\big|^{-1}\big)\,du_j',\nonumber
\end{align}
where we sum over all integers $m$ satisfying $1\leq2^m< 2R/y$. We now note the following general bound.

\begin{lem}\label{int1}
 For all $S>0$, $M\geq1$ and $\alpha\in\R$ we have
\begin{align}\label{nuardenflyttad}
\int_0^{S}\min\big(M,\big|u-|\alpha|\big|^{-1}\big)\,du=
\begin{cases}
O\big(1+\log^+(SM)\big) & \text{always,}\\
O\big(S\min\big(M,1/|\alpha|\big)\big)  & \text{if $S<\sfrac{1}{2\min(M,1/|\alpha|)}$.}
\end{cases}
\end{align}
\end{lem}

\begin{proof}
If $|\alpha|>S$ then the integral increases if the range of integration is replaced by $\big[|\alpha|-S,|\alpha|\big]$, since $u\mapsto |u-|\alpha||^{-1}$ is an increasing function of $u\in[0,|\alpha|)$. Hence the left hand side of \eqref{nuardenflyttad} is always
\begin{align*}
<\int_{|\alpha|-S}^{|\alpha|+S}\min\big(M,\big|u-|\alpha|\big|^{-1}\big)\,du
=2\int_0^S\min\big(M,u^{-1}\big)\,du=O\big(1+\log^+(SM)\big).
\end{align*}

Next assume $S<\sfrac{1}{2\min(M,1/|\alpha|)}$. If $1/|\alpha|<M$ then $S<\frac12|\alpha|$ and thus the integrand in \eqref{nuardenflyttad} is $\leq2|\alpha|^{-1}$ for all $u\in[0,S]$; in the remaining case we simply use that the integrand is everywhere $\leq M$; we thus obtain the bound $O\big(S\min\big(M,1/|\alpha|\big)\big)$, as stated.
\end{proof}

Continuing onwards with the proof of Theorem \ref{Lemma1'}, we define $\beta_j:=\min\big(M_j+1,1/|\alpha_j'|\big)$ for all $j\in\{1,\ldots,n\}$ and order the indices so that $\beta_{j_k}$ are decreasing, i.e.\ $\beta_{j_1}\geq\beta_{j_2}\geq\cdots\geq\beta_{j_n}$. Using Lemma \ref{int1} and recalling that $R=1/(2|c|^2)$ we find that \eqref{estim1} is bounded by
\begin{align*}
 &O(1)\underset{1\leq2^m< 1/(|c|^2y)}{\sum}2^{2m(s-n)}y^{s-n}\prod_{j=1}^n
\left\{\begin{array}{ll}2^my\beta_j  & \textrm{if $2^{m}<\sfrac{A}{2y\beta_j}$}\\
1+\log^+\big(\sfrac{2^my}{A}(M_j+1)\big) & \textrm{if $2^{m}\geq\sfrac{A}{2y\beta_j}$}
\end{array}\right\}\nonumber\\
&=O(1)\underset{1\leq2^m<\min\big(\sfrac{A}{2y\beta_{j_1}},\sfrac{1}{|c|^2y}\big)}{\sum}2^{m(2s-n)}y^s\prod_{k=1}^n\beta_{j_k}\nonumber\\
&+O(1)\sum_{i=1}^n\bigg(\underset{R_i\leq2^m< R_{i+1}}{\sum}2^{m(2s-n-i)}y^{s-i}\prod_{k=1}^i\Big(1+\log^+\big(\sfrac{2^my}{A}(M_{j_k}+1)\big)\Big)\prod_{k=i+1}^n\beta_{j_k}\bigg),
\end{align*}
where
$R_i:=\max\big(1,\min\big(\sfrac{A}{2y\beta_{j_i}},\sfrac{1}{|c|^2y}\big)\big)$
for $i\in\{1,\ldots,n\}$ and $R_{n+1}:=\sfrac{1}{|c|^2y}$. It
follows that \eqref{firststep} is bounded by
\begin{align}\label{i13}
 &O\big(|c|^{2(s-n)}\big)\int_0^{1/|c|^2}\bigg\{\underset{1\leq2^m<\min\big(\sfrac{A}{2y\beta_{j_1}},\sfrac{1}{|c|^2y}\big)}{\sum}2^{m(2s-n)}y^{s-\delta}\prod_{k=1}^n\beta_{j_k}\\
&+\sum_{i=1}^n\bigg(\underset{R_i\leq2^m< R_{i+1}}{\sum}2^{m(2s-n-i)}y^{s-i-\delta}\prod_{k=1}^i\Big(1+\log^+\big(\sfrac{2^my}{A}(M_{j_k}+1)\big)\Big)\prod_{k=i+1}^n\beta_{j_k}\bigg)\bigg\}\,dy.\nonumber
\end{align}

We call the sum in the first line of \eqref{i13} $S_0$ and the $i$:th term in the sum in the last line of  \eqref{i13} $S_i$. In order to estimate $S_0$ we first assume that $\min\big(\sfrac{A}{2y\beta_{j_1}},\sfrac{1}{|c|^2y}\big)=\sfrac{1}{|c|^2y}$, i.e.\ $|c|^2\geq\sfrac{2\beta_{j_1}}{A}$. Note that in this situation $S_j=0$ for all $j\in\{1,\ldots,n\}$. Using $2s-n>0$ and estimating the geometric sum involved, we get
\begin{align*}
 S_0&=y^{s-\delta}\Big(\prod_{k=1}^n\beta_{j_k}\Big)\underset{1\leq2^m<\sfrac{1}{|c|^2y}}{\sum}2^{m(2s-n)}=O\Big(y^{n-s-\delta}|c|^{2n-4s}\prod_{k=1}^n\beta_{j_k}\Big).
\end{align*}
Integrating we find that the contribution to \eqref{i13} is
\begin{align}\label{i14}
&O\Big(|c|^{2(s-n)}|c|^{2n-4s}\prod_{k=1}^n\beta_{j_k}\Big)\int_0^{1/|c|^2}y^{n-s-\delta}\,dy=O\Big(|c|^{2(\delta-n-1)}\prod_{k=1}^n\beta_{j_k}\Big).
\end{align}
We now turn to the case
$\min\big(\sfrac{A}{2y\beta_{j_1}},\sfrac{1}{|c|^2y}\big)=\sfrac{A}{2y\beta_{j_1}}$,
i.e.\ $|c|^2\leq\sfrac{2\beta_{j_1}}{A}$. We note that in this case
$S_0$ is an empty sum if not $y<\sfrac{A}{2\beta_{j_1}}$. For such $y$
we get
\begin{align*}
S_0=y^{s-\delta}\Big(\prod_{k=1}^n\beta_{j_k}\Big)\underset{1\leq2^m<\sfrac{A}{2y\beta_{j_1}}}{\sum}2^{m(2s-n)}
=O\Big(y^{n-s-\delta}\beta_{j_1}^{\,n+1-2s}\prod_{k=2}^n\beta_{j_k}\Big),
\end{align*} 
and hence the contribution from $S_0$ to
\eqref{i13} is
\begin{align}\label{i15}
O\Big(|c|^{2(s-n)}\beta_{j_1}^{\,n+1-2s}\prod_{k=2}^n\beta_{j_k}\Big)\int_{0}^{\sfrac{A}{2\beta_{j_1}}}
y^{n-s-\delta}\,dy=O\Big(|c|^{2(s-n)}\beta_{j_1}^{\,\delta-s}\prod_{k=2}^n\beta_{j_k}\Big).
\end{align}

We next estimate $S_i$ for $i\in\{1,\ldots,n\}$. To begin with we
assume $R_{i}=\max\big(1,\sfrac{A}{2y\beta_{j_i}}\big)$ and
$R_{i+1}=\sfrac{1}{|c|^2y}$, i.e.\ $\sfrac{2\beta_{j_{i+1}}}{A}\leq|c|^2\leq\sfrac{2\beta_{j_i}}{A}$ where $\beta_{j_{n+1}}:=0$. In this
situation all $S_{\ell}$ with $\ell>i$ are zero. Estimating the
geometric sums involved yields
\begin{align}\label{i16}
&S_i=y^{s-i-\delta}\Big(\prod_{k=i+1}^n\beta_{j_k}\Big)\underset{\max\big(1,\sfrac{A}{2y\beta_{j_i}}\big)\leq2^m<\sfrac{1}{|c|^2y}}{\sum}
2^{m(2s-n-i)}\prod_{k=1}^i\Big(1+\log^+\big(\sfrac{2^my}{A}(M_{j_k}+1)\big)\Big)\\
&=O\bigg(y^{s-i-\delta}\prod_{k=1}^i\Big(1+\log^+\big(\sfrac{M_{j_k}+1}{A|c|^2}\big)\Big)\prod_{k=i+1}^n\beta_{j_k}\bigg)
\begin{cases}\max\big(1,\sfrac{1}{y\beta_{j_i}}\big)^{2s-n-i} &
\text{if $2s<n+i$}\\
1+\log\Big(\frac{1}{|c|^2\max\big(y,\sfrac{A}{2\beta_{j_i}}\big)}\Big)
& \text{if $2s=n+i$}\\
\big(\sfrac{1}{|c|^2y}\big)^{2s-n-i} & \text{if $2s>n+i$.}
\end{cases}\nonumber
\end{align}
When $2s<n+i$ we get the following contribution to \eqref{i13}:
\begin{align}\label{i17}
O\bigg(|c|^{2(s-n)}\prod_{k=1}^i\Big(1+\log^+\big(\sfrac{M_{j_k}+1}{A|c|^2}\big)\Big)\prod_{k=i+1}^n\beta_{j_k}\bigg)
\int_0^{1/|c|^2}y^{s-i-\delta}\max\big(1,\sfrac{1}{y\beta_{j_i}}\big)^{2s-n-i}\,dy.
\end{align}
We estimate the integral in \eqref{i17} as follows (using $\delta-s\notin\Z$):
\begin{align}\label{i175}
&\int_0^{1/|c|^2}y^{s-i-\delta}\max\big(1,\sfrac{1}{y\beta_{j_i}}\big)^{2s-n-i}\,dy\nonumber\\
&=\int_0^{\min\big(\sfrac{1}{\beta_{j_i}},\sfrac{1}{|c|^2}\big)}
y^{s-i-\delta}\Big(\frac{1}{y\beta_{j_i}}\Big)^{2s-n-i}\,dy+\int_{\min\big(\sfrac{1}{\beta_{j_i}},\sfrac{1}{|c|^2}\big)}^{1/|c|^2}y^{s-i-\delta}\,dy\\
&=O\Big(\max\big(\beta_{j_i}^{\,i+\delta-s-1},|c|^{2(i+\delta-s-1)}\big)\Big).\nonumber
\end{align}
Hence \eqref{i17} is bounded by
\begin{align}\label{i18}
O\Big(|c|^{2(s-n)}\max\big(\beta_{j_i}^{\,i+\delta-s-1},|c|^{2(i+\delta-s-1)}\big)\prod_{k=1}^i\Big(1+\log^+\big(\sfrac{M_{j_k}+1}{A|c|^2}\big)\Big)\prod_{k=i+1}^n\beta_{j_k}\Big).
\end{align}
When $2s=n+i$ we have $s-i=n-s>0$. In this situation the contribution from \eqref{i16} to \eqref{i13} is
\begin{multline}\label{i19}
O\bigg(|c|^{2(s-n)}\prod_{k=1}^i\Big(1+\log^+\big(\sfrac{M_{j_k}+1}{A|c|^2}\big)\Big)\prod_{k=i+1}^n\beta_{j_k}\bigg)\\\times\int_0^{1/|c|^2}y^{s-i-\delta}\Big(1+\log\Big(\frac{1}{|c|^2\max\big(y,\sfrac{A}{2\beta_{j_i}}\big)}\Big)\Big)\,dy.
\end{multline}
The integral in \eqref{i19} is
\begin{align}\label{i20}
=O\Big(\Big(1+\log\Big(\frac{2\beta_{j_i}}{A|c|^2}\Big)\Big)\beta_{j_i}^{\,i+\delta-s-1}\Big)+O\Big(|c|^{2(i+\delta-s-1)}\Big)+\int_{\sfrac{A}{2\beta_{j_i}}}^{1/|c|^2}y^{s-i-\delta}\log\Big(\frac{1}{|c|^2y}\Big)\,dy.
\end{align}
Changing variables, $y=e^{-t}/|c|^2$, we find that the integral in \eqref{i20} is bounded by
\begin{align}\label{i205}
 |c|^{2(i+\delta-s-1)}\int_0^{\infty}te^{t(i+\delta-s-1)}\,dt=O\big(|c|^{2(i+\delta-s-1)}\big).
\end{align}
Hence \eqref{i20} is $O\big(|c|^{2(i+\delta-s-1)}\big)$
and it follows that \eqref{i19} is bounded by
\begin{align}\label{i21}
O\Big(|c|^{2(s-n)}|c|^{2(i+\delta-s-1)}\prod_{k=1}^i\Big(1+\log^+\big(\sfrac{M_{j_k}+1}{A|c|^2}\big)\Big)\prod_{k=i+1}^n\beta_{j_k}\Big).
\end{align}
Finally, when $2s>n+i$ we find that the contribution from \eqref{i16} to \eqref{i13} is
\begin{align}\label{i22}
&O\Big(|c|^{2(s-n)}\prod_{k=1}^i\Big(1+\log^+\big(\sfrac{M_{j_k}+1}{A|c|^2}\big)\Big)\prod_{k=i+1}^n\beta_{j_k}\Big)
\int_0^{1/|c|^2}y^{s-i-\delta}\Big(\frac{1}{|c|^2y}\Big)^{2s-n-i}\,dy\\
&=O\Big(|c|^{2(s-n)}|c|^{2(i+\delta-s-1)}\prod_{k=1}^i\Big(1+\log^+\big(\sfrac{M_{j_k}+1}{A|c|^2}\big)\Big)\prod_{k=i+1}^n\beta_{j_k}\Big).\nonumber
\end{align}
It follows from \eqref{i18}, \eqref{i21} and \eqref{i22} that in all cases the  contribution from \eqref{i16} to \eqref{i13} is
\begin{align}\label{i23}
O\Big(|c|^{2(s-n)}\max\big(\beta_{j_i}^{\,i+\delta-s-1},|c|^{2(i+\delta-s-1)}\big)\prod_{k=1}^i\Big(1+\log^+\big(\sfrac{M_{j_k}+1}{A|c|^2}\big)\Big)\prod_{k=i+1}^n\beta_{j_k}\Big).
\end{align}

It remains to estimate $S_i$ for $i\in\{1,\ldots,n-1\}$ under the assumption that $R_{i}=\max\big(1,\sfrac{A}{2y\beta_{j_i}}\big)$ and
$R_{i+1}=\max\big(1,\sfrac{A}{2y\beta_{j_{i+1}}}\big)$, i.e.\
$|c|^2\leq\sfrac{2\beta_{j_{i+1}}}{A}$. We note that $S_i$ is an empty sum if not $y<\sfrac{A}{2\beta_{j_{i+1}}}$. For such $y$ we get
\begin{align}\label{i24}
S_i&=O\bigg(y^{s-i-\delta}\prod_{k=1}^i\Big(1+\log^+\big(\sfrac{M_{j_k}+1}{2\beta_{j_{i+1}}}\big)\Big)\prod_{k=i+1}^n\beta_{j_k}\bigg)\\
&\hspace{100pt}\times
\begin{cases}
\max\big(1,\sfrac{1}{y\beta_{j_i}}\big)^{2s-n-i} &
\text{if $2s<n+i$}\\
1+\log\Big(\frac{A}{2\beta_{j_{i+1}}\max\big(y,\sfrac{A}{2\beta_{j_i}}\big)}\Big)
& \text{if $2s=n+i$}\\
\big(\frac{1}{y\beta_{j_{i+1}}}\big)^{2s-n-i} & \text{if $2s>n+i$.}
\end{cases}\nonumber
\end{align}
By a similar case by case analysis as above it follows that in all cases the contribution from \eqref{i24} to \eqref{i13} is
\begin{align}\label{i31}
O\Big(|c|^{2(s-n)}\max\big(\beta_{j_i}^{\,i+\delta-s-1},\beta_{j_{i+1}}^{\,i+\delta-s-1}\big)\prod_{k=1}^i\Big(1+\log^+\big(\sfrac{M_{j_k}+1}{2\beta_{j_{i+1}}}\big)\Big)\prod_{k=i+1}^n\beta_{j_k}\Big).
\end{align}

We now collect the estimates above to get a resulting bound on \eqref{firststep}. When $\sfrac{2\beta_{j_1}}{A}\leq|c|^2$ it follows from \eqref{i13} and \eqref{i14} that \eqref{firststep} is bounded by
\begin{align}\label{secondstep1}
O\Big(|c|^{2(\delta-n-1)}\prod_{k=1}^n\beta_{j_k}\Big).
\end{align}
Furthermore,  when $\sfrac{2\beta_{j_{\ell+1}}}{A}\leq|c|^2<\sfrac{2\beta_{j_{\ell}}}{A}$ and $\ell\in\{1,\ldots,n\}$ (recall that $\beta_{j_{n+1}}=0$), it follows from \eqref{i13}, \eqref{i15}, \eqref{i23} and \eqref{i31} that \eqref{firststep} is bounded by
\begin{align}\label{secondstep2}
 O(1)&\bigg\{|c|^{2(s-n)}\beta_{j_1}^{\,\delta-s}\prod_{k=2}^n\beta_{j_k}\\
&+\sum_{i=1}^{\ell-1}\bigg(|c|^{2(s-n)}\max\big(\beta_{j_i}^{\,i+\delta-s-1},\beta_{j_{i+1}}^{\,i+\delta-s-1}\big)\prod_{k=1}^i\Big(1+\log^+\big(\sfrac{M_{j_k}+1}{2\beta_{j_{i+1}}}\big)\Big)\prod_{k=i+1}^n\beta_{j_k}\bigg)\nonumber\\
&+|c|^{2(s-n)}\max\big(\beta_{j_{\ell}}^{\,\ell+\delta-s-1},|c|^{2(\ell+\delta-s-1)}\big)\prod_{k=1}^{\ell}\Big(1+\log^+\big(\sfrac{M_{j_k}+1}{A|c|^2}\big)\Big)\prod_{k=\ell+1}^n\beta_{j_k}\bigg\}.\nonumber
\end{align}
We define $i_0$ to be the smallest integer $>s+1-\delta$ (thus $\frac n2<i_0\leq n$) and set $i_\ell:=\min(\ell,i_0)$ for $\ell\in\{1,\ldots,n\}$. Comparing the sizes of all terms we find that \eqref{secondstep2} is 
\begin{align}\label{secondstep3}
O\Big(|c|^{2(s-n)}\big(\log(2|\vecM|)\big)^{\ell}\prod_{k=i_{\ell}+1}^n\beta_{j_k}\Big)
\begin{cases}
|c|^{2(\ell+\delta-s-1)} & \text{if $\ell<i_0$}\\
\beta_{j_{i_0}}^{\,i_0+\delta-s-1} & \text{if $\ell\geq i_0$.}
\end{cases}
\end{align}
Combining \eqref{firststep}, \eqref{secondstep1} and \eqref{secondstep3} we conclude that
\begin{align}\label{secondstep4}
 &\int_0^{1}\int_{I_{\vecal}}\lfloor y_{A_kW(u_1\vecom_1+\cdots+u_n\vecom_n+yi_n)}\rfloor_1^{n-s}\prod_{j=1}^n\min\big(M_j+1,|u_j-\alpha_j|^{-1}\big)\,\frac{du_1\ldots du_ndy}{y^{\delta}}\nonumber\\
&=O\Big(\big(\log(2|\vecM|)\big)^{n}\Big)
\begin{cases}
|c|^{2(\delta-n-1)}\prod_{k=1}^n\beta_{j_k} & \text{if $\sfrac{2\beta_{j_1}}{A}\leq|c|^2$}\\
|c|^{2(\ell+\delta-n-1)}\prod_{k=\ell+1}^n\beta_{j_k} & \text{if $\sfrac{2\beta_{j_{\ell+1}}}{A}\leq|c|^2<\sfrac{2\beta_{j_{\ell}}}{A}$, $1\leq\ell<i_0$}\\
|c|^{2(s-n)}\beta_{j_{i_0}}^{\,i_0+\delta-s-1}\prod_{k=i_0+1}^n\beta_{j_k} & \text{if $|c|^2<\sfrac{2\beta_{j_{i_0}}}{A}$}.
\end{cases}
\end{align}

We now add up all nonzero contributions to \eqref{readytoestimate} for one fixed $k\in\{1,...,\kappa\}$. We recall from the discussion above \eqref{rk} that we get such a nonzero contribution from $W=A_k^{-1}\matr{a}{b}{c}{d}\in\Gamma_{\eta_k}\backslash\Gamma$ only if $|c|\geq1/\sqrt{\tau_1\tau_k}$ and $-c^{-1}d$ belongs to the ball with centre $\alpha_1\vecom_1+\cdots+\alpha_n\vecom_n$ and radius $\widetilde R_k$. We will split the summation over these $W$ into dyadic boxes. To fix the notation, let us agree that
each integer vector $\vecm=(m_1,\ldots,m_n) \in\Z^n$
corresponds to the dyadic box of all $W=A_k^{-1}\matr abcd
\in\Gamma_{\eta_k}\backslash\Gamma$ (with $|c|\geq 1/\sqrt{\tau_1\tau_k}$)
such that, for all $j\in\{1,\ldots,n\}$,
\begin{align} \notag
& \alpha_j'\in \bigl[-2/(M_j+1),2/(M_j+1)\bigr]
&& \text{if } \: m_j=0;
\\ \label{DYADICBOX}
& \alpha_j'\in \bigl[2^{m_j}/(M_j+1),2^{m_j+1}/(M_j+1)\bigr]
&& \text{if } \: m_j>0;
\\ \notag
& \alpha_j'\in \bigl[-2^{|m_j|+1}/(M_j+1),-2^{|m_j|}/(M_j+1)\bigr]
&& \text{if } \: m_j<0.
\end{align}
We will consider these dyadic boxes for each integer vector
$\vecm$ with
\begin{align*}
|m_j|\leq \log_2\bigl(\tilde{R}_k(M_j+1)/A\bigr),\qquad
j=1,\ldots,n.
\end{align*}
Note that in this way all the non-zero contributions to \eqref{readytoestimate}
are certainly accounted for.
Furthermore, note that for $W=A_k^{-1}\matr abcd$ belonging to the
dyadic box corresponding to $\vecm$, we have
\begin{align} \label{BETAIDYADICFACT}
\beta_j=\min\big(M_j+1,1/|\alpha_j'|\big)\in
[2^{-|m_j|-1}(M_j+1),2^{-|m_j|}(M_j+1)],\qquad j=1,\ldots,n.
\end{align}

Recall that $-c^{-1}d=\sum_{j=1}^n(\alpha_j-\alpha_j')\vecom_j$. Let us fix a vector $\vecm$ as above, and let
$\widehat{\mathfrak{B}}_{\vecm}\subset\R^n$ be the
corresponding box for $-c^{-1}d$ given by \eqref{DYADICBOX}.
This box has sides parallel with the basis vectors
$\vecom_1,\ldots,\vecom_n$,
of lengths $\leq 2^{|m_j|+2}|\vecom_j|/(M_j+1)$,
$j=1,\ldots,n$.
Now there is some constant $\mathcal{C}>0$, which depends only on $\Omega$,
such that we can find a right-angled box
$\mathfrak{B}_{\vecm}\subset \R^n$
which contains $\widehat{\mathfrak{B}}_{\vecm}$,
and which has sides of lengths
\begin{align}\label{bdef}
b_{\vecm,j}:=\mathcal{C}2^{|m_j|}/(M_j+1),
\qquad j=1,\ldots,n.
\end{align}
(For example, we might let the sides of $\mathfrak{B}_{\vecm}$
be parallel with the orthonormal basis vectors obtained by performing the
Gram-Schmidt process on the vectors $\{\vecom_j\}_{j=1}^{n}$,
ordered by decreasing values of
$2^{|m_j|+2}|\vecom_j|/(M_j+1)$.
This always works with
$\mathcal{C}=4n\max_{1\leq j\leq n}|\vecom_j|$.)

Given $\vecm$ we fix $j_1,\ldots,j_n$ to be a permutation of
the indices $1,\ldots,n$ such that $b_{\vecm,j_1}\leq
b_{\vecm,j_2}
\leq\ldots\leq b_{\vecm,j_n}$.
By \eqref{BETAIDYADICFACT} we have
$\mathcal{C}/(2b_{\vecm,j})\leq\beta_j\leq
\mathcal{C}/b_{\vecm,j}$ for each $j$, and hence
if $j_1',\ldots,j_n'$ is a permutation of $1,\ldots,n$ such that
$\beta_{j_1'}\geq\beta_{j_2'}\geq\ldots\geq\beta_{j_n'}$
then $\beta_{j_i'}/\beta_{j_i}\in[\frac 12,2]$ for each $i$. Hence, by \eqref{secondstep4}, we get the following bound on the total contribution to
\eqref{readytoestimate} from the $W$'s in our dyadic box:
\begin{align}\label{thirdstep}
 &O\Big(\big(\log(2|\vecM|)\big)^{n}\Big)\bigg\{\Big(2^{-|m_{j_{i_0}}|}(M_{j_{i_0}}+1)\Big)^{i_0+\delta-s-1}\\
&\times\Big(\prod_{k=i_0+1}^n\big(2^{-|m_{j_{k}}|}(M_{j_{k}}+1)\big)\Big)\underset{\substack{W\\1/\sqrt{\tau_1\tau_k}\leq|c|<D/\sqrt{b_{\vecm,j_{i_0}}}}}{\sum}|c|^{2(s-n)}\nonumber\\
&+\sum_{\ell=1}^{i_0-1}\Big(\prod_{k=\ell+1}^n\big(2^{-|m_{j_{k}}|}(M_{j_{k}}+1)\big)\Big)\underset{\substack{W\\C/\sqrt{b_{\vecm,j_{\ell+1}}}\leq|c|<D/\sqrt{b_{\vecm,j_{\ell}}}}}{\sum}|c|^{2(\ell+\delta-n-1)}\nonumber\\
&+\Big(\prod_{k=1}^n\big(2^{-|m_{j_{k}}|}(M_{j_{k}}+1)\big)\Big)\underset{\substack{W\\C/\sqrt{b_{\vecm,j_{1}}}\leq|c|}}{\sum}|c|^{2(\delta-n-1)}\bigg\},\nonumber
\end{align}
with $C=\sqrt{\mathcal{C}/2A}$ and $D=2\sqrt{\mathcal{C}/A}$. Here the sums are taken over a set of representatives $W=A_k^{-1}\matr abcd\in\Gamma_{\eta_k}\backslash\Gamma$ restricted by
$-c^{-1}d\in\mathfrak{B}_{\vecm}$ and the stated bounds on
$|c|$. To estimate the sums in \eqref{thirdstep} we use Lemma \ref{csum}.
For the first $W$-sum in \eqref{thirdstep}, set $\nu=2(n-s)$,
and note that $\nu>2(n-i)$ for all $i\in\{i_0,\ldots,n\}$
since $i_0-s>1-\delta>0$. Hence
\begin{align}\label{csum1}
 &\underset{\substack{W\\1/\sqrt{\tau_1\tau_k}\leq|c|<D/\sqrt{b_{\vecm,j_{i_0}}}}}{\sum}|c|^{2(s-n)}\nonumber\\
 &=\underset{\substack{W\\1/\sqrt{\tau_1\tau_k}\leq|c|<D/\sqrt{b_{\vecm,j_{n}}}}}{\sum}|c|^{2(s-n)}+\sum_{i=i_0}^{n-1}\underset{\substack{W\\D/\sqrt{b_{\vecm,j_{i+1}}}\leq|c|<D/\sqrt{b_{\vecm,j_{i}}}}}{\sum}|c|^{2(s-n)}\\
&=O(1)+\sum_{i=i_0}^{n-1}O\Big((b_{\vecm,j_{i+1}})^{\,i-s}\prod_{k=i+1}^nb_{\vecm,j_k}\Big)=O(1),\nonumber
\end{align}
where the implied constant depends only on $\Gamma$, $\phi$ and $\Omega$.
Next, for any $\ell\in\{1,\ldots,i_0-1\}$, if we set
$\nu=2(n+1-\ell-\delta)$
then $\nu>2(n-\ell)$, and hence 
\begin{align}\label{csum2}
 \underset{\substack{W\\C/\sqrt{b_{\vecm,j_{\ell+1}}}\leq|c|<D/\sqrt{b_{\vecm,j_{\ell}}}}}{\sum}|c|^{2(\ell+\delta-n-1)}=O\Big((b_{\vecm,j_{\ell+1}})^{1-\delta}\prod_{k=\ell+1}^nb_{\vecm,j_k}\Big),
\end{align}
where the implied constant depends only on $\Gamma$, $\Omega $ and $\delta$. Finally, using $\nu=2(n+1-\delta)>2n$, we get
\begin{align}\label{csum3}
 \underset{\substack{W\\C/\sqrt{b_{\vecm,j_{1}}}\leq|c|}}{\sum}|c|^{2(\delta-n-1)}=O\Big((b_{\vecm,j_{1}})^{1-\delta}\prod_{k=1}^nb_{\vecm,j_k}\Big),
\end{align}
where again the implied constant depends only on $\Gamma$, $\Omega $ and $\delta$. Using \eqref{bdef}, \eqref{csum1}, \eqref{csum2} and \eqref{csum3} we find that \eqref{thirdstep} is bounded by
\begin{align}\label{finalthirdstep}
 &O\bigg(\Big(2^{-|m_{j_{i_0}}|}(M_{j_{i_0}}+1)\Big)^{i_0+\delta-s-1}\Big(\prod_{k=i_0+1}^n\big(2^{-|m_{j_{k}}|}(M_{j_{k}}+1)\big)\Big)\big(\log(2|\vecM|)\big)^{n}\bigg).
\end{align}
In particular, when $\vecm=(0,\ldots,0)$ the permutation
$j_1,\ldots,j_n$ is such that $M_{j_1}\geq\ldots\geq M_{j_n}$,
and \eqref{finalthirdstep} equals
\begin{align}\label{fourthstep}
O\Big((M_{j_{i_0}}+1)^{i_0+\delta-s-1}\Big(\prod_{k=i_0+1}^n(M_{j_k}+1)\Big)\big(\log(2|\vecM|)\big)^{n}\Big).
\end{align}
For any other $\vecm$ we find by inspection that \eqref{finalthirdstep}
(for $j_1,\ldots,j_n$
with $b_{\vecm,j_1}\leq\ldots\leq b_{\vecm,j_n}$)
is majorized by \eqref{fourthstep}
(for $j_1,\ldots,j_n$ with $M_{j_1}\geq\ldots\geq M_{j_n}$). Since there are $O\big(\big(\log(2|\vecM|)\big)^{n}\big)$ such $\vecm$ we find that \eqref{readytoestimate} is bounded by
\begin{align}\label{fourthstep1}
O\Big((M_{j_{i_0}}+1)^{i_0+\delta-s-1}\Big(\prod_{k=i_0+1}^n(M_{j_k}+1)\Big)\big(\log(2|\vecM|)\big)^{2n}\Big),
\end{align}
where the indices are ordered so that $M_{j_1}\geq\ldots\geq M_{j_n}$. From \eqref{Khatestim} we now get
that also $|\widehat{\mathcal{K}}|$ is bounded by \eqref{fourthstep1},
and from \eqref{khat} it follows that
\begin{multline}\label{news}
\underset{m_1+\cdots+m_n>0}{\sum_{m_1=0}^{M_1}\cdots\sum_{m_n=0}^{M_n}}c_{m_1\vecom_1^*+\cdots+m_n\vecom_n^*}|m_1\vecom_1^*+\cdots+m_n\vecom_n^*|^{\delta-n/2-1}e^{2\pi i(m_1\alpha_1+\cdots+m_n\alpha_n)}\\
=O\bigg((M_{j_{i_0}}+1)^{i_0+\delta-s-1}\Big(\prod_{k=i_0+1}^n(M_{j_k}+1)\Big)\big(\log(2|\vecM|)\big)^{2n}\bigg),
\end{multline}
for all $\vecM\in\Z_{\geq0}^n\setminus\{\vec0\}$, $(\alpha_1,\ldots,\alpha_n)\in\R^n$ and all admissible $\delta\in(0,1)$. 

We call the sum in the left hand side of \eqref{news} $S(M_1,\ldots,M_n)$ and we define $S(0,\ldots,0):=0$. We furthermore define
\begin{align*}
g(x_1,\ldots,x_n):=|x_1\vecom_1^*+\cdots+x_n\vecom_n^*|^{n/2+1-\delta},
\end{align*}
and note that $g$ is smooth in $\R^n\setminus\{\vec0\}$. Using Corollary \ref{corsumbp} and \eqref{news} we get
\begin{align}\label{IBP}
 &\underset{m_1+\cdots+m_n>0}{\sum_{m_1=0}^{M_1}\cdots\sum_{m_n=0}^{M_n}}c_{m_1\vecom_1^*+\cdots+m_n\vecom_n^*}e^{2\pi i(m_1\alpha_1+\cdots+m_n\alpha_n)}\\
&=\underset{A\subset N}{\sum}(-1)^{|A|}\underset{\prod_{j\in A}[0,M_j]}{\int} g_{A,N\setminus A}(\vecx)S_{A,N\setminus A}(\vecx)\,d\vecx\nonumber\\
&=O\bigg((M_{j_{i_0}}+1)^{i_0+\delta-s-1}\Big(\prod_{k=i_0+1}^n(M_{j_k}+1)\Big)\big(\log(2|\vecM|)\big)^{2n}\bigg)\nonumber\\
&\hspace{119pt}\times\bigg\{\underset{A\subset N}{\sum}\underset{\prod_{j\in A}[0,M_j]}{\int} |g_{A,N\setminus A}(\vecx)|\,d\vecx \bigg\}.\nonumber
\end{align}
(Recall that $N=\{1,\ldots,n\}$.) In order to get a bound on the last factor we recall from \eqref{aa} and \eqref{bb} that for all $A$ we have
\begin{align*}
 \underset{\prod_{j\in A}[0,M_j]}{\int} |g_{A,N\setminus A}(\vecx)|\,d\vecx=O\big(|\vecM|^{n/2+1-\delta}\big).
\end{align*}
Using this estimate in \eqref{IBP} we arrive at
\begin{multline*}
 \underset{m_1+\cdots+m_n>0}{\sum_{m_1=0}^{M_1}\cdots\sum_{m_n=0}^{M_n}}c_{m_1\vecom_1^*+\cdots+m_n\vecom_n^*}e^{2\pi i(m_1\alpha_1+\cdots+m_n\alpha_n)}\\
 =O\bigg(|\vecM|^{n/2+1-\delta}(M_{j_{i_0}}+1)^{i_0+\delta-s-1}\Big(\prod_{k=i_0+1}^n(M_{j_k}+1)\Big)\big(\log(2|\vecM|)\big)^{2n}\bigg).
\end{multline*}
Finally we add $c_{\vec0}$, choose $\delta=1-\ve$ (with $\ve$ small) and conclude that
\begin{multline*}
 \sum_{m_1=0}^{M_1}\cdots\sum_{m_n=0}^{M_n}c_{m_1\vecom_1^*+\cdots+m_n\vecom_n^*}e^{2\pi i(m_1\alpha_1+\cdots+m_n\alpha_n)}\\
 =O\Big(|\vecM|^{n/2+2\ve}(M_{j_{i_0}}+1)^{i_0-s}\prod_{k=i_0+1}^n(M_{j_k}+1)\Big),
\end{multline*}
for all $\vecM\in\Z_{\geq0}^n\setminus\{\vec0\}$ and all $(\alpha_1,\ldots,\alpha_n)\in\R^n$.
\end{proof}

\begin{remark}
 In the same way we get that for all $(\ve_1,\ldots,\ve_n)\in\{±1\}^n$, all $\vecM=(M_1,\ldots,M_n)\in\Z_{\geq0}^n\setminus\{\vec0\}$ and all $(\alpha_1,\ldots,\alpha_n)\in\R^n$,
\begin{multline*}
  \sum_{m_1=0}^{M_1}\cdots\sum_{m_n=0}^{M_n}c_{\ve_1m_1\vecom_1^*+\cdots+\ve_nm_n\vecom_n^*}e^{2\pi i(\ve_1m_1\alpha_1+\cdots+\ve_nm_n\alpha_n)}\\=O\Big(|\vecM|^{n/2+\ve}(M_{j_{i_0}}+1)^{i_0-s}\prod_{k=i_0+1}^n(M_{j_k}+1)\Big),
 \end{multline*}
where the implied constant depends only on $\Gamma$, $\phi$, $\Omega$ and $\ve$.
\end{remark}

\begin{remark}
 Note that the right hand side in \eqref{lincomb'} is $O\big(|\vecM|^{3n/2-s+\ve}\big)$. Here the exponent $\frac{3n}{2}-s+\ve$ is essentially optimal, as the following proposition shows. (Compare \cite[Prop.\ 5.1$'$]{andreas} in the $2$-dimensional case.)
\end{remark}

\begin{prop}
 Let $\phi$ be as in Theorem \ref{Lemma1'}. Choose $k\in\{1,\ldots,\kappa\}$ such that $c_{\vec0}^{(k)}\neq0$ (cf.\ \eqref{fi1}), and let $\vecal=\alpha_1\vecom_1+\cdots+\alpha_n\vecom_n\in\R^n$ be any cusp equivalent to $\eta_k$. Then, for at least one choice of $(\ve_1,\ldots,\ve_n)\in\{±1\}^n$, there exists some $c>0$ and infinitely many $\vecM=(M_1,\ldots,M_n)\in\Z_{\geq0}^n$ such that
\begin{align}\label{cbound}
 \bigg|\sum_{m_1=0}^{M_1}\cdots\sum_{m_n=0}^{M_n}c_{\ve_1m_1\vecom_1^*+\cdots+\ve_nm_n\vecom_n^*}e^{2\pi i(\ve_1m_1\alpha_1+\cdots+\ve_nm_n\alpha_n)}\bigg|>c|\vecM|^{3n/2-s}.
\end{align}
\end{prop}

\begin{proof}
Assume that the conclusion above is not true, i.e.\ assume that for all \linebreak$(\ve_1,\ldots,\ve_n)\in\{±1\}^n$ and all $c>0$ there are only finitely many $\vecM\in\Z_{\geq0}^n$ such that \eqref{cbound} holds.
(Here, and in the rest of this proof, $\vecal$ is fixed.) We write
\begin{align}\label{bla}
  \phi(\vecal+yi_n)=c_{\vec0}y^{n-s}+\underset{\substack{D\subset\{1,\ldots,n\}\\D\neq\emptyset}}{\sum}\underset{\vecm\in R_D}{\sum}c_{\vecmu}y^{n/2}K_{s-n/2}(2\pi|\vecmu|y)e^{2\pi i\langle\vecmu,\vecal\rangle},
 \end{align}
where $R_D$ is as in \eqref{RD} and where $\vecmu=m_1\vecom_1^*+\cdots+m_n\vecom_n^*$ as usual. We estimate each inner sum in \eqref{bla} seperately. It will be sufficient to estimate, for each nonempty
$D=\{j_1,\ldots,j_{|D|}\}\subset\{1,\ldots,n\}$,
\begin{align}\label{bestSD}
 & S_D^+:=\underset{\vecm\in R_D^+}{\sum}c_{\vecmu}y^{n/2}K_{s-n/2}(2\pi|\vecmu|y)e^{2\pi i\langle\vecmu,\vecal\rangle},
\end{align}
where $R_D^+=R_D\cap(\Z_{\geq0})^n$. The remaining parts of the $R_D$-sum can then be estimated in the same way using our assumption with various $\{\pm1\}^n$-vectors $(\varepsilon_1,\ldots,\varepsilon_n)$. In order to apply summation by parts to \eqref{bestSD} we let $a(0,\ldots,0)=0$ and
\begin{align*}
a(m_{j_1},\ldots,m_{j_{|D|}})
:=
c_{m_{j_1}\vecom_{j_1}^*+\cdots+m_{j_{|D|}}\vecom_{j_{|D|}}^*}e^{2\pi i(m_{j_1}\alpha_{j_1}+\cdots+m_{j_{|D|}}\alpha_{j_{|D|}})}
\end{align*}
for $(m_{j_1},\ldots,m_{j_{|D|}})\neq\vec0$. We also introduce
\begin{align*}
 g(x_{j_1},\ldots,x_{j_{|D|}}):=y^{n/2}K_{s-n/2}\big(2\pi|x_{j_1}\vecom_{j_1}^*+\cdots+x_{j_{|D|}}\vecom_{j_{|D|}}^*|y\big)
\end{align*}
and
\begin{align*}
S(X_{j_1},\ldots,X_{j_{|D|}}):=\underset{0\leq m_{j_1}\leq X_{j_1}}{\sum}\cdots\underset{0\leq m_{j_{|D|}}\leq X_{j_{|D|}}}{\sum}a(m_{j_1},\ldots,m_{j_{|D|}}).
\end{align*}
According to \eqref{intbp*},
\begin{align}\label{bestSintbp}
 S_D^+=(-1)^{|D|}\underset{A\subset D}{\sum}\underset{\prod_{j\in A}[1,\infty)}{\int} g_{A,\emptyset}(\vecx)S_{A,\emptyset}(\vecx)\,d\vecx.
\end{align}
It is clear that $S_{\emptyset,\emptyset}=0$ and thus that the corresponding term in the sum in \eqref{bestSintbp} is zero. We now consider nonempty $A$. Let $c$ be given. By our assumption there exists a number $R>1$, depending on $c$, such that $|S(\vecX)|\leq c|\vecX|^{3n/2-s}$ for all $\vecX=(X_{j_1},\ldots,X_{j_{|D|}})$ satisfying $|\vecX|\geq R$. Let $T:=\sup_{|\vecX|\leq R}|S(\vecX)|$.

Recall that the derivatives of $g$ in \eqref{bestSintbp} correspond to multi-indices of the type $\ell=(\ell_1,\ldots,\ell_{|D|})$ satisfying $\ell_i\leq1$. One shows by induction that for each such $\ell$, $\frac{\partial^{|\ell|}}{\partial x_{j_1}^{\ell_1}\cdots\partial x_{j_{|D|}}^{\ell_{|D|}}}g(\vecx)$ is a finite sum of terms of the form
\begin{multline*}
y^{n/2}K_{s-n/2}^{(m)}\big(2\pi|x_{j_1}\vecom_{j_1}^*+\cdots+x_{j_{|D|}}\vecom_{j_{|D|}}^*|y\big)\nonumber\\
\times(2\pi y)^m\prod_{i=1}^m\frac{\partial^{|\ell^i|}}{\partial x_{j_1}^{l_1^i}\cdots\partial x_{j_{|D|}}^{l_{|D|}^i}}\big(|x_{j_1}\vecom_{j_1}^*+\cdots+x_{j_{|D|}}\vecom_{j_{|D|}}^*|\big)
\end{multline*}
where $0\leq m\leq|\ell|$, $\ell^1,\dots,\ell^m$ are multi-indices of length $\geq 1$, and $\sum_{i=1}^m|\ell_i|=|\ell|$. Using the bounds \eqref{chia} and \eqref{chic} and recalling  \eqref{quadraticform} we get
\begin{align}\label{bestg1}
\frac{\partial^{|\ell|}}{\partial x_{j_1}^{\ell_1}\cdots\partial x_{j_{|D|}}^{\ell_{|D|}}}g(\vecx)=O\Big(y^{n-s}|\vecx|^{n/2-s-|\ell|}\Big),\hspace{10pt} \forall\vecx\in(\R_{>0})^{|D|}.
\end{align}
By \eqref{quadraticform} there exists a positive constant $k_3$ such that $k_3|\vecx|\leq|x_{j_1}\vecom_{j_1}^*+\cdots+x_{j_{|D|}}\vecom_{j_{|D|}}^*|$ for all $\vecx\in\R^{|D|}$. 
Moreover, it follows from \cite[pp.\ 79,202]{watson} that $K_{s-n/2}^{(m)}(u)=O(u^{-1/2}e^{-u})$ for $u\geq 2\pi k_3$ and $m\geq0$. Hence, for $|\vecx|\geq\frac{1}{y}$ and the same multi-index $\ell$ as in \eqref{bestg1}, we get the stronger bound
\begin{align*}
&\frac{\partial^{|\ell|} }{\partial x_{j_1}^{\ell_1}\cdots\partial x_{j_{|D|}}^{\ell_{|D|}}}g(\vecx)=O\Big(y^{(n-1)/2+|\ell|}|\vecx|^{-1/2}e^{-2\pi k_3|\vecx|y}\Big).
\end{align*}

Returning to the integrals in \eqref{bestSintbp} we get, for $y<\frac{1}{R}$,
\begin{align*}
 &\underset{\prod_{j\in A}[1,\infty)}{\int} g_{A,\emptyset}(\vecx)S_{A,\emptyset}(\vecx)\,d\vecx\\
 &=O\big(y^{n-s}T\big)\underset{\substack{\prod_{j\in A}[1,\infty)\\|\vecx|\leq R}}{\int}|\vecx|^{n/2-s-|A|}\,d\vecx
 +O\big(y^{n-s}c\big)\underset{\substack{\prod_{j\in A}[1,\infty)\\R<|\vecx|\leq \frac{1}{y}}}{\int}|\vecx|^{2n-2s-|A|}\,d\vecx\\
&+O\big(y^{(n-1)/2+|A|}c\big)\underset{\substack{\prod_{j\in A}[1,\infty)\\\frac{1}{y}<|\vecx|}}{\int}|\vecx|^{(3n-1)/2-s}e^{-2\pi k_3|\vecx|y}\,d\vecx
=O\big(y^{n-s}T\big)+O\big(y^{s-n}c\big).
\end{align*}
Hence it follows from \eqref{bla} and \eqref{bestSintbp} that
\begin{align}\label{bestphi}
 |\phi(\vecal+yi_n)|=O\big(y^{n-s}\big)\big(|c_{\vec0}|+T\big)+O\big(y^{s-n}c\big)\qquad\text{as }\:y\to0,
\end{align}
where the implied constants only depend on $\Gamma$, $\phi$ and $\Omega$. Now, since $s-n<0<n-s$, $c>0$ is arbitrary and the implied constants in \eqref{bestphi} do not depend on $c$, it follows that
\begin{align}\label{best}
 \phi(\vecal+yi_n)=o(y^{s-n})\qquad\text{as } y\to0.
\end{align}

In order to reach a contradiction we choose $W\in\Gamma$ such that $\vecal=W(\eta_k)$ and write $A_kW^{-1}=\matr{a}{b}{c}{d}$. Since $A_kW^{-1}(\vecal)=\infty$ we have $c\vecal+d=0$, which shows that $c\neq0$. Thus, using \eqref{ycoord}, we find that
\begin{align*}
 y_{A_kW^{-1}(\vecal+yi_n)}=\frac{1}{|c|^2y}
\end{align*}
for all $y>0$. Finally it follows from \eqref{fi1} and the $\Gamma$-invariance of $\phi$ that
\begin{align*}
 \phi(\vecal+yi_n)\sim c_{\vec0}^{(k)}\big(y_{A_kW^{-1}(\vecal+yi_n)}\big)^{n-s}=c_{\vec0}^{(k)}|c|^{2(s-n)}y^{s-n}
\end{align*}
as $y\to0$. This contradicts the result in \eqref{best}.
\end{proof}

We next study the horosphere integral of non-cuspidal eigenfunctions.

\begin{prop}\label{chinoncuspprop}
 Let $\chi:\R^n\to \R$ be a smooth function with compact support. Let $\ve>0$ and let $\phi$ be a non-cuspidal eigenfunction with eigenvalue $\lambda>0$. Define $s$ via $\lambda=s(n-s)$, $s\in(n/2,n)$. Then the following holds, uniformly over all $0<y<1$, all $\vecgam=(\gamma_1,\ldots,\gamma_n)\in\R^n$ and all $\delta_1,\ldots,\delta_n$ satisfying $0<\delta_1,\ldots,\delta_n\leq 1$:
\begin{multline*}
 \frac{1}{\delta_1\cdots\delta_n}\int_{\R^n}\chi_{\vecdel,\vecgam}(\vecu)\phi(u_1\vecom_1+\cdots+u_n\vecom_n+yi_n)\,d\vecu\\
 =O\Big(\|\chi\|_{n,1}y^{n-s-\ve}\big(\underset{i\in\{1,\ldots,n\}}{\min}\delta_i\big)^{2(s-n)}\Big),
\end{multline*}
 where the implied constant depends only on $\Gamma$, $\phi$, $\Omega$ and $\ve$.
\end{prop}

\begin{proof}
 The proof is very similar to the proof of Proposition \ref{chicuspprop}. Using the Fourier expansions \eqref{chiseries}, \eqref{noncusp} and integrating term by term we get
 \begin{multline}\label{chinoncuspint}
  \frac{1}{\delta_1\cdots\delta_n}\int_{\R^n}\chi_{\vecdel,\vecgam}(\vecu)\phi(u_1\vecom_1+\cdots+u_n\vecom_n+yi_n)\,d\vecu=\frac{1}{\delta_1\cdots\delta_n}c_{\vec0}y^{n-s}\widehat{\chi_{\vecdel,\vecgam}}(\vec0)\\
+\frac{1}{\delta_1\cdots\delta_n}\underset{\substack{D\subset\{1,\ldots,n\}\\D\neq\emptyset}}{\sum}\underset{\vecm\in R_D}{\sum}c_{\vecmu}y^{n/2}K_{s-n/2}(2\pi|\vecmu|y)\widehat{\chi_{\vecdel,\vecgam}}(-\vecm).
 \end{multline}
(Recall the definition of $R_D$ from \eqref{RD}.) It follows from \eqref{transfestim} (with $k=0$) that
\begin{align*}
 \frac{1}{\delta_1\cdots\delta_n}c_{\vec0}y^{n-s}\widehat{\chi_{\vecdel,\vecgam}}(\vec0)=O\big(\|\chi\|_{0,1}y^{n-s}\big).
\end{align*}
We estimate each inner sum in \eqref{chinoncuspint} separately. As in the proof of Proposition \ref{chicuspprop}, given a nonempty $D=\{j_1,\ldots,j_{|D|}\}\subset\{1,\ldots,n\}$ it is enough to estimate
\begin{align}\label{SD*}
 & S_D^+:=\frac{1}{\delta_1\cdots\delta_n}\underset{\vecm\in R_D^+}{\sum}c_{\vecmu}y^{n/2}K_{s-n/2}(2\pi|\vecmu|y)\widehat{\chi_{\vecdel,\vecgam}}(-\vecm).
\end{align}
We split this sum into the parts $S_D^1$ and $S_D^2$ as in \eqref{SD12} and estimate these sums in the same way as in the proof of Proposition \ref{chicuspprop} except for the following minor difference: When estimating $S_{A,\emptyset}$ we use the bound
\begin{align*}
 S_{A,\emptyset}(\vecx)=O\big(|\vecx|^{3n/2-s+\ve}\big),
\end{align*}
 which is a weak form of the result in Theorem \ref{Lemma1'}. We find that
\begin{align*}
 S_D^+=O\Big(\|\chi\|_{n,1}y^{n-s-\ve}\big(\underset{i\in\{1,\ldots,n\}}{\min}\delta_i\big)^{2(s-n)}\Big),
\end{align*}
which proves the proposition.
\end{proof}

\subsection{The horosphere integral for the Eisenstein series}\label{Eisenchap}

The proofs here are similar to the proofs of Proposition \ref{cuspformintprop} and Proposition \ref{cuspimportant}.

\begin{prop}\label{eisensteinprop}
 Let $\chi:\R^n\to \R$ be a smooth function with compact support. Let $\varepsilon>0$ and $\ell\in\{1,...,\kappa\}$. Keep $0<y<1$ and $T\geq0$. Furthermore let $\vecgam=(\gamma_1,\ldots,\gamma_n)\in\R^n$ and let $\delta_1,\ldots,\delta_n$ be such that $0<\delta_1,\ldots,\delta_n\leq 1$. Then
 \begin{align*}
 &\frac{1}{\delta_1\cdots\delta_n}\int_{\R^n}\chi_{\vecdel,\vecgam}(\vecu)E_{\ell}\big(u_1\vecom_1+\cdots+u_n\vecom_n+yi_n,\sfrac{n}{2}+iT\big)\,d\vecu\\
 &\hspace{100pt}=O\Big(\|\chi\|_{n,1}(T+1)^{1/6+2\ve}\sqrt{W(T)}y^{n/2-\ve}\big(\underset{i\in\{1,\ldots,n\}}{\min}\delta_i\big)^{-n}\Big),
 \end{align*}
where the implied constant depends only on $\Gamma$, $\Omega$ and $\ve$.
\end{prop}

\begin{proof}
It is enough to consider $0<\ve<\frac{1}{2}$. We recall that the Eisenstein series has a Fourier expansion at infinity of the form
\begin{multline*}
E_{\ell}\big(\vecx+yi_n,\sfrac{n}{2}+iT\big)=\delta_{\ell1}y^{\sfrac{n}{2}+iT}+\varphi_{\ell1}\big(\sfrac{n}{2}+iT\big)y^{\sfrac{n}{2}-iT}\\
+\sum_{\vec0\neq\vecmu\in\Lambda_1^*}a_{\vecmu}\big(\sfrac{n}{2}+iT\big)y^{n/2}K_{iT}(2\pi|\vecmu|y)e^{2\pi
i\langle\vecmu,\vecx\rangle}.
\end{multline*}
Using this, \eqref{chiseries} and integrating term by term we get
\begin{align}\label{eis}
&\frac{1}{\delta_1\cdots\delta_n}\int_{\R^n}\chi_{\vecdel,\vecgam}(\vecu)E_{\ell}\big(u_1\vecom_1+\cdots+u_n\vecom_n+yi_n,\sfrac{n}{2}+iT\big)\,d\vecu\nonumber\\
&=\frac{1}{\delta_1\cdots\delta_n}\Big(\delta_{\ell1}y^{\sfrac{n}{2}+iT}+\varphi_{\ell1}\big(\sfrac{n}{2}+iT\big)y^{\sfrac{n}{2}-iT}\Big)\widehat{\chi_{\vecdel,\vecgam}}(\vec0)\\
&+\frac{1}{\delta_1\cdots\delta_n}\underset{\vecm\in
\Z^n\setminus{\{\vec0\}}}{\sum}a_{\vecmu}\big(\sfrac{n}{2}+iT\big)y^{n/2}K_{iT}(2\pi|\vecmu|y)\widehat{\chi_{\vecdel,\vecgam}}(-\vecm).\nonumber
\end{align}
We note that since $\Phi\big(\sfrac{n}{2}+iT\big)$ is unitary for
real $T$ we have $\big|\varphi_{\ell
1}\big(\sfrac{n}{2}+iT\big)\big|\leq1$. This fact together with
\eqref{transfestim} (with $k=0$) yields
\begin{align}\label{1eis}
 \frac{1}{\delta_1\cdots\delta_n}\Big(\delta_{\ell1}y^{\sfrac{n}{2}+iT}+\varphi_{\ell1}\big(\sfrac{n}{2}+iT\big)y^{\sfrac{n}{2}-iT}\Big)\widehat{\chi_{\vecdel,\vecgam}}(\vec0)=O\big(\|\chi\|_{0,1}y^{n/2}\big).
\end{align}
Using (a weak version of) \eqref{transfestim} with $k=n$ and the estimate of the K-Bessel function in \eqref{kbesselestim} we estimate the last line of \eqref{eis} by
\begin{align}\label{Iestim}
 &O\Big(\|\chi\|_{n,1}y^{n/2-\ve}Y^ne^{-(\pi/2)T}(T+1)^{-1/3+\ve}\Big)\\
&\hspace{45pt}\times\sum_{\vec0\neq\vecmu\in\Lambda_1^*}\big|a_{\vecmu}\big(\sfrac{n}{2}+iT\big)\big||\vecmu|^{-n-\ve}\min\big(1,e^{(\pi/2)T-2\pi|\vecmu|y}\big).\nonumber
\end{align}
(Recall that $Y=\underset{i\in\{1,\ldots,n\}}{\max}\delta_i^{-1}$.) 

We estimate the sum in \eqref{Iestim} (which we call $\Sigma$) in the same way as we estimate the corresponding sum in the proof of Proposition \ref{cuspformintprop}. The only difference is that we have to use Proposition \ref{FinalRS}
instead of Proposition \ref{cusprankin} when we estimate the sum $S(X)$.
Recalling that $W(T)\geq 1$ (cf.\ Sec.\ \ref{strongversionsection}) we get 
\begin{align*}
 S(X):&=\underset{0<|\vecmu|\leq X}{\sum}\big|a_{\vecmu}\big(\sfrac{n}{2}+iT\big)\big|\nonumber\\
&=O\bigg(e^{(\pi/2)T}(T+1)^{\ve}\sqrt{W(T)}\Big(T+\frac{X^n}{(T+1)^{n-1}}\Big)^{1/2}X^{n/2+\ve/2}\bigg)
\end{align*}
for all $X\geq\sfrac{\mu_0}{2}$. Using this bound together with \eqref{fder} we find that $ \Sigma$ is
\begin{multline*}
 O\Big(e^{(\pi/2)T}(T+1)^{\ve}\sqrt{W(T)}\Big)\bigg\{\int_{\frac{\mu_0}{2}}^{\max(\frac{\mu_0}{2},\frac{T}{4y})}X^{-n/2-1-\ve/2}\Big(T+\frac{X^n}{(T+1)^{n-1}}\Big)^{1/2}\,dX\nonumber\\
 +\int_{\max(\frac{\mu_0}{2},\frac{T}{4y})}^{\infty}(1+yX)X^{-n/2-1-\ve/2}\Big(T+\frac{X^n}{(T+1)^{n-1}}\Big)^{1/2}e^{(\pi/2)T-2\pi yX}\,dX\bigg\}.
\end{multline*}
Estimating these integrals as in the proof of Proposition \ref{cuspformintprop} we obtain
\begin{align*}
 \Sigma=O\Big(e^{(\pi/2)T}(T+1)^{1/2+\ve}\sqrt{W(T)}\Big),
\end{align*}
which together with \eqref{Iestim} gives the desired result.
\end{proof}

\begin{prop}\label{eisimportant}
Let $\chi:\R^n\to \R$ be a smooth function with compact support.
Let $\ve>0$, $k>\frac{n+1}{2}+\frac{1}{6}$ and
$f\in H^k(\Gamma\setminus\H^{n+1})$. Let the spectral expansion of
$f$ be
\begin{align}\label{spectraldecigen}
 f(P)=\sum_{m\geq0}c_m \phi_m(P) + \sum_{\ell=1}^{\kappa}\int_{0}^{\infty}g_{\ell}(t) E_{\ell}\big(P,\sfrac{n}{2}+it\big)\,dt.
\end{align}
Then the following holds, for all $0<y<1$, $\ell\in\{1,\ldots,\kappa\}$, $\vecgam=(\gamma_1,\ldots,\gamma_n)\in\R^n$ and all $\delta_1,\ldots,\delta_n$ satisfying $0<\delta_1,\ldots,\delta_n\leq 1$:
\begin{multline}\label{important2}
 \frac{1}{\delta_1\cdots\delta_n}\int_{\R^n}\chi_{\vecdel,\vecgam}(\vecu)
\bigg\{\int_{0}^{\infty}g_{\ell}(t) E_{\ell}\big(u_1\vecom_1+\cdots+u_n\vecom_n+yi_n,\sfrac{n}{2}+it\big)\,dt\bigg\}\,d\vecu\\
=O\Big(\|f\|_{H^k}\|\chi\|_{n,1}y^{n/2-\ve}\big(\underset{i\in\{1,\ldots,n\}}{\min}\delta_i\big)^{-n}\Big),
 \end{multline}
where the implied constant depends only on $\Gamma$, $\Omega$ , $k$ and $\ve$.
\end{prop}

\begin{proof}
 It is enough to consider $0<\ve<\frac{1}{2}\big(k-\frac{n+1}{2}-\frac{1}{6}\big)$. We change the order of integration and apply the Cauchy-Schwarz inequality to the left hand side of \eqref{important2} to obtain
 \begin{align*}
 &\bigg|\frac{1}{\delta_1\cdots\delta_n}\int_{\R^n}\chi_{\vecdel,\vecgam}(\vecu)\bigg\{\int_{0}^{\infty}g_{\ell}(t) E_{\ell}\big(u_1\vecom_1+\cdots+u_n\vecom_n+yi_n,\sfrac{n}{2}+it\big)\,dt\bigg\}\,d\vecu\bigg|\\
&\leq\bigg(\int_{0}^{\infty}\big|g_{\ell}(t)\big|^2(t+1)^{2k}\,dt\bigg)^{1/2}\\
&\times\bigg(\int_{0}^{\infty}(t+1)^{-2k}\bigg|\frac{1}{\delta_1\cdots\delta_n}\int_{\R^n}\chi_{\vecdel,\vecgam}(\vecu)
E_{\ell}\big(u_1\vecom_1+\cdots+u_n\vecom_n+yi_n,\sfrac{n}{2}+it\big)\,d\vecu\bigg|^2\,dt\bigg)^{1/2}.
 \end{align*}
By definition the first factor above is bounded by $\|f\|_{H^k}$.
Furthermore, using Proposition \ref{eisensteinprop}, we find that
the second factor is bounded by
\begin{align}\label{lastestimate}
O\Big(\|\chi\|_{n,1}y^{n/2-\ve}\big(\underset{i\in\{1,\ldots,n\}}{\min}\delta_i\big)^{-n}\Big)
\bigg(\int_{0}^{\infty}(t+1)^{-2k+1/3+4\ve}W(t)\,dt\bigg)^{1/2}.
\end{align}
Let $P(t)$ be defined by
\begin{align*}
 P(t):=\int_{0}^{t}W(r)\,dr.
\end{align*}
Then $P(t)=O(t^{n+1})$ for large $t$ (cf.\ \eqref{wt}) and the derivative of $P(t)$
equals $W(t)$. Integration by parts yields
\begin{align*}
 &\int_{0}^{K}(t+1)^{-2k+1/3+4\ve}W(t)\,dt=\int_{0}^{K}(t+1)^{-2k+1/3+4\ve}\,dP(t)\\
 & =(K+1)^{-2k+1/3+4\ve}P(K)+\big(2k-\sfrac{1}{3}-4\ve\big)\int_{0}^{K}(t+1)^{-2k-2/3+4\ve}P(t)\,dt.
\end{align*}
Note that the integrand above is positive, that
$(t+1)^{-2k-2/3+4\ve}P(t)=O(t^{n-2k+1/3+4\ve})$ for
large $t$, and that $(K+1)^{-2k+1/3+4\ve}P(K)\to0$ as $K\to\infty$.
Hence we conclude that the integral in \eqref{lastestimate} is convergent and we arrive at the desired
estimate.
\end{proof}

\section{The main theorems}

We now put together the results of sections \ref{cuspchap}, \ref{noncuspchap} and \ref{Eisenchap}.

\begin{thm}\label{chithm}
Let $\chi:\R^n\to \R$ be a smooth function with compact support. Let $\ve>0$ and $k>\frac{n+1}{2}+\frac{1}{6}$. We then have, for all $f\in H^k(\Gamma\setminus\H^{n+1})$, all $0<y<1$, all $\vecgam=(\gamma_1,\ldots,\gamma_n)\in\R^n$ and all $\delta_1,\ldots,\delta_n$ satisfying $\sqrt{y}\leq\delta_1,\ldots,\delta_n\leq 1$,
 \begin{align*}
  &\frac{1}{\delta_1\cdots\delta_n}\int_{\R^n}\chi_{\vecdel,\vecgam}(\vecu)f(u_1\vecom_1+\cdots+u_n\vecom_n+yi_n)\,d\vecu\\
  &=\frac{\langle\chi\rangle}{\nu(\Gamma\setminus\H^{n+1})}\int_{\Gamma\setminus\H^{n+1}}f(P)\,d\nu(P)+O\Big(\|f\|_{H^k}\|\chi\|_{n,1}y^{-\ve}\big(y/\delta_{\min}^2\big)^{n/2}\Big)\\
  &\hspace{170pt}+O\Big(\|f\|_{L^2}\|\chi\|_{n,1}y^{-\ve}\big(y/\delta_{\min}\big)^{n-s_1}\Big)\\
  &\hspace{170pt}+O\Big(\|f\|_{L^2}\|\chi\|_{n,1}y^{-\ve}\big(y/\delta_{\min}^2\big)^{n-\hat s_1}\Big),
 \end{align*}
where $\langle\chi\rangle=\int_{\R^n}\chi(\vecx)\,d\vecx$, $\delta_{\min}=\underset{i\in\{1,\ldots,n\}}{\min}\delta_i$, and the implied constants depend only on $\Gamma$, $\Omega$, $k$ and $\ve$. In the above $s_1\in(n/2,n)$ is the largest number such that there exists a cusp form on $\Gamma\setminus\H^{n+1}$ of eigenvalue $\lambda=s_1(n-s_1)$. If no such function exists on  $\Gamma\setminus\H^{n+1}$ the middle error term above is omitted. Also, $\hat s_1\in(n/2,n)$ is the largest number such that there exists a non-cuspidal eigenfunction on $\Gamma\setminus\H^{n+1}$ of eigenvalue $\lambda=\hat s_1(n-\hat s_1)$. Again, if no such function exists on  $\Gamma\setminus\H^{n+1}$ the third error term above is omitted.
\end{thm}

\begin{proof}
 Let $f$ have the spectral decomposition \eqref{spectraldecigen}. We substitute this expansion into the horosphere integral
 \begin{align}\label{int}
 \frac{1}{\delta_1\cdots\delta_n}\int_{\R^n}\chi_{\vecdel,\vecgam}(\vecu)f(u_1\vecom_1+\cdots+u_n\vecom_n+yi_n)\,d\vecu,                                             \end{align}
and integrate the result termwise. Recalling that $\phi_0=\nu(\Gamma\setminus\H^{n+1})^{-1/2}$, we note that $\phi_0$'s contribution to \eqref{int} equals
\begin{align*}
 \frac{\langle\chi\rangle}{\nu(\Gamma\setminus\H^{n+1})}\int_{\Gamma\setminus\H^{n+1}}f(P)\,d\nu(P).
\end{align*}
We also note that $|\langle f,\phi_m\rangle|\leq\|f\|_{L^2}$ and recall that there are only finitely many $m$ satisfying $\lambda_m<(\frac{n}{2})^2$. Now the theorem follows from Proposition \ref{chicuspprop}, Proposition \ref{cuspimportant}, Proposition \ref{chinoncuspprop} and Proposition \ref{eisimportant}.
\end{proof}

We now turn to the situation where $\chi$ is the characteristic function of the rectangle $I_{\vec0}:=[-1/2,1/2]^n$. To distinguish this situation from the one in Theorem \ref{chithm} we call this function $\chi_0$. To prove asymptotic equidistribution also in this case we will approximate $\chi_0$ by smooth functions and then apply Theorem \ref{chithm}.

\begin{thm}\label{chi0thm}
 Let $\ve'>0$ and $k>\frac{n+1}{2}+\frac{1}{6}$. We then have, for all bounded functions $f\in H^k(\Gamma\setminus\H^{n+1})$, all $0<y<1$ and all $\alpha_1,\ldots,\alpha_n,\beta_1,\ldots,\beta_n$ satisfying \mbox{$\sqrt{y}\leq\beta_1-\alpha_1,\ldots,\beta_n-\alpha_n\leq 1$},
 \begin{align}\label{finally}
 &\frac{1}{(\beta_1-\alpha_1)\cdots(\beta_n-\alpha_n)}\int_{\alpha_1}^{\beta_1}\cdots\int_{\alpha_n}^{\beta_n}f(u_1\vecom_1+\cdots+u_n\vecom_n+yi_n)\,du_1\ldots du_n\\
  &=\frac{1}{\nu(\Gamma\setminus\H^{n+1})}\int_{\Gamma\setminus\H^{n+1}}f(P)\,d\nu(P)+O\Big(\|f\|_{L^{\infty}}^{\frac{n-1}{n}}\|f\|_{H^k}^{1/n}y^{1/2-\ve'}\big(\underset{i\in\{1,\ldots,n\}}{\min}(\beta_i-\alpha_i)\big)^{-1}\Big)\nonumber\\
  &\hspace{120pt}+O\Big(\|f\|_{L^{\infty}}^{\frac{n-1}{n}}\|f\|_{L^2}^{1/n}y^{1-s_1/n-\ve'}\big(\underset{i\in\{1,\ldots,n\}}{\min}(\beta_i-\alpha_i)\big)^{s_1/n-1}\Big)\nonumber\\
  &\hspace{120pt}+O\Big(\|f\|_{L^{\infty}}^{\frac{n-1}{n}}\|f\|_{L^2}^{1/n}y^{1-\hat s_1/n-\ve'}\big(\underset{i\in\{1,\ldots,n\}}{\min}(\beta_i-\alpha_i)\big)^{2(\hat s_1/n-1)}\Big),\nonumber
 \end{align}
where the implied constants depend only on $\Gamma$, $\Omega$, $k$ and $\ve'$. In the above $s_1\in(n/2,n)$ is the largest number such that there exists a cusp form on $\Gamma\setminus\H^{n+1}$ of eigenvalue $\lambda=s_1(n-s_1)$. If no such function exists on  $\Gamma\setminus\H^{n+1}$ the middle error term above is omitted. Also, $\hat s_1\in(n/2,n)$ is the largest number such that there exists a non-cuspidal eigenfunction on $\Gamma\setminus\H^{n+1}$ of eigenvalue $\lambda=\hat s_1(n-\hat s_1)$. Again, if no such function exists on  $\Gamma\setminus\H^{n+1}$ the third error term above is omitted.
\end{thm}

\begin{proof}
Fix, once and for all, a function $\psi\in C^{\infty}(\R^n)$, with support contained in the closed unit ball, satisfying $\psi\geq0$ and $\int_{\R^n}\psi(\vecx)\,d\vecx=1$. For $h>0$ we define
\begin{align*}
 \psi_h(\vecx):=h^{-n}\psi(h^{-1}\vecx).
\end{align*}
We define $\chi_h$ as the convolution of $\chi_0$ and $\psi_h$:
\begin{align}\label{chih}
 \chi_{h}(\vecx):=\chi_0\ast\psi_h(\vecx)=h^{-n}\int_{I_{\vec0}}\psi\big(h^{-1}(\vecx-\vecy)\big)\,d\vecy.
\end{align}
By \cite[Thm.\ 1.3.2]{hor} we know that $\chi_h\in C^{\infty}(\R^n)$, with support in an $h$-neighbourhood of $I_{\vec0}$, and that $\chi_h$ tends to $\chi_0$ in $L^1(\R^n)$ (even uniformly outside an $h$-neighbourhood of the boundary of $I_{\vec0}$) as $h\to0$. Since each $\chi_h$ is smooth, Theorem \ref{chithm} holds with $\chi=\chi_h$.

Clearly 
\begin{align}\label{mean}
 \langle\chi_h\rangle=\int_{I_{\vec0}}\Big(h^{-n}\int_{\R^n}\psi\big(h^{-1}(\vecx-\vecy)\big)\,d\vecx\Big)\,d\vecy=1
\end{align}
for all $h>0$. By differentiating under the integral sign in \eqref{chih} it follows that $D^{\vecal}\chi_h=O(h^{-|\vecal|})$ for all multi-indices $\vecal$, where the implied constant depends only on $|\vecal|$. It also follows from \eqref{chih} that $\chi_h$ is (locally) constant except in an $h$-neighbourhood of the boundary of $I_{\vec0}$. Hence 
\begin{align}\label{sobolev}
 \|\chi_h\|_{n,1}=O(h^{1-n})
\end{align}
 for all $0<h<1$.

Now using \eqref{mean} and \eqref{sobolev} in Theorem \ref{chithm}, with $\delta_i=\beta_i-\alpha_i$ and $\gamma_i=\frac{1}{2}(\alpha_i+\beta_i)$ for $i\in\{1,\ldots,n\}$, and noting that \mbox{$\chi_0-\chi_h$} is zero except on an $h$-neighbourhood of the boundary of $I_{\vec0}$, we get, for $0<h<1$,
\begin{align}\label{approx}
 &\frac{1}{(\beta_1-\alpha_1)\cdots(\beta_n-\alpha_n)}\int_{\alpha_1}^{\beta_1}\cdots\int_{\alpha_n}^{\beta_n}f(u_1\vecom_1+\cdots+u_n\vecom_n+yi_n)\,du_1\ldots du_n\\
&=\frac{1}{\delta_1\cdots\delta_n}\int_{\R^n}\big((\chi_0)_{\vecdel,\vecgam}(\vecu)-(\chi_h)_{\vecdel,\vecgam}(\vecu)\big)f(u_1\vecom_1+\cdots+u_n\vecom_n+yi_n)\,d\vecu\nonumber\\
&+\frac{1}{\delta_1\cdots\delta_n}\int_{\R^n}(\chi_h)_{\vecdel,\vecgam}(\vecu)f(u_1\vecom_1+\cdots+u_n\vecom_n+yi_n)\,d\vecu\nonumber\\
&=\frac{1}{\nu(\Gamma\setminus\H^{n+1})}\int_{\Gamma\setminus\H^{n+1}}f(P)\,d\nu(P)+O\Big(\|f\|_{L^{\infty}}h\Big)
+O\Big(\|f\|_{H^k}y^{n/2-\ve}\big(\underset{i\in\{1,\ldots,n\}}{\min}\delta_i\big)^{-n}h^{1-n}\Big)\nonumber\\
&+O\Big(\|f\|_{L^2}y^{n-s_1-\ve}\big(\underset{i\in\{1,\ldots,n\}}{\min}\delta_i\big)^{s_1-n}h^{1-n}\Big)+O\Big(\|f\|_{L^2}y^{n-\hat s_1-\ve}\big(\underset{i\in\{1,\ldots,n\}}{\min}\delta_i\big)^{2(\hat s_1-n)}h^{1-n}\Big).\nonumber
\end{align}
Next we find an $h$ that minimizes the sum of the error terms in \eqref{approx}. When $n=1$ we let $h\to0$ in \eqref{approx} for any fixed $y$ and $\delta_1,\ldots,\delta_n$, and thus obtain the right hand side of \eqref{finally} (with $\ve'=\ve$). Now assume that $n>1$. Note that for any $A,B>0$, the function $h\mapsto Ah+Bh^{1-n}$ attains its minimum at $h_{\min}=\big(\frac{(n-1)B}{A}\big)^{1/n}$, and $Ah_{\min}+Bh_{\min}^{1-n}=O\big(A^{\frac{n-1}{n}}B^{\frac{1}{n}}\big)$, where the implied constant depends only on $n$. Using this fact with $A=\|f\|_{L^{\infty}}$ and
\begin{align*}
 B=\|f\|_{H^k}y^{n/2-\ve}\big(\underset{i\in\{1,\ldots,n\}}{\min}\delta_i\big)^{-n}&+\|f\|_{L^2}y^{n-s_1-\ve}\big(\underset{i\in\{1,\ldots,n\}}{\min}\delta_i\big)^{s_1-n}\\
&+\|f\|_{L^2}y^{n-\hat s_1-\ve}\big(\underset{i\in\{1,\ldots,n\}}{\min}\delta_i\big)^{2(\hat s_1-n)},
\end{align*}
gives the result \eqref{finally} (with $\ve'=\ve/n$) also for $n>1$,
so long as $h_{\text{min}}<1$ (so that \eqref{approx} can be applied with
$h=h_{\text{min}}$). In the remaining case, $h_{\text{min}}\geq 1$,
viz.\ $A\leq (n-1)B$, we use instead the trivial fact that the left hand side of \eqref{finally} equals
\begin{align*}
\frac{1}{\nu(\Gamma\setminus\H^{n+1})}\int_{\Gamma\setminus\H^{n+1}}f(P)\,d\nu(P) + O\bigl(\|f\|_{L^{\infty}}\bigr),
\end{align*}
and here
\begin{align*}
\|f\|_{L^{\infty}}=A\leq A^{\frac{n-1}n}\big((n-1)B\big)^{\frac 1n}
=O\bigl(A^{\frac{n-1}n}B^{\frac 1n}\bigr).
\end{align*}
Hence \eqref{finally} holds also in this case.
\end{proof}

Theorem \ref{chi0thm} is easily seen to imply Theorem \ref{thmA} in the introduction, that is
\begin{multline*}
 \frac{1}{(\beta_1-\alpha_1)\cdots(\beta_n-\alpha_n)}\int_{\alpha_1}^{\beta_1}\cdots\int_{\alpha_n}^{\beta_n}f(u_1\vecom_1+\cdots+u_n\vecom_n,y)\,du_1\ldots
du_n\nonumber\\
\to\frac{1}{\nu(\Gamma\setminus\H^{n+1})}\int_{\Gamma\setminus\H^{n+1}}f(P)\,d\nu(P),
\end{multline*}
uniformly as $y\to0$ so long as $\beta_1-\alpha_1,\ldots,\beta_n-\alpha_n\geq y^{1/2-\ve}$. The fact that Theorem \ref{thmA} holds for any \textit{continuous} function $f$ of compact support on $\Gamma\setminus\H^{n+1}$ follows because of the following standard approximation fact: Given any continuous function $f$ of compact support on $\Gamma\setminus\H^{n+1}$, then for any $\ve'>0$ there exists a $C^\infty$-function $f_1$ on $\H^{n+1}$ which is $\Gamma$-invariant and of compact support on $\Gamma\setminus\H^{n+1}$, and such that $\|f-f_1\|_{L^\infty}<\ve'$.

\begin{remark}\label{finalremark}
The exponent $\frac{1}{2}$ in "$\beta_1-\alpha_1,\ldots,\beta_n-\alpha_n\geq y^{1/2-\ve}$" is in fact the best possible.  To see this, note that it follows from \eqref{ycoord} that for any $W=\matr{a}{b}{c}{d}\in\PSL(2,C_n)$ with $c\neq0$ we have
\begin{align}\label{horoball}
 W^{-1}\big(\R^n\times[B,\infty)\big)=\bigg\{\vecx+yi_n\in\H^{n+1}\,\Big|\, y\leq\sfrac{1}{B|c|^2}, |\vecx+c^{-1}d|\leq\sqrt{\sfrac{y}{B|c|^2}-y^2}\bigg\}
\end{align}
(cf.\ \cite[Rem.\ 2.1.4]{andreas2}). We now fix $B>B_0$ and $W\in\Gamma$ with $c\neq0$. Then $\eta=-c^{-1}d$ is a cusp equivalent to infinity and $W^{-1}\big(\R^n\times[B,\infty)\big)$ a horoball tangent to $\partial\H^{n+1}$ at $\eta$. If we keep $\beta_1-\alpha_1=\ldots=\beta_n-\alpha_n=k y^{1/2}$ for a small enough constant $k>0$, it follows from \eqref{horoball} that the box-shaped horosphere subset
\begin{align*}
 \B=\Big\{u_1\vecom_1+\cdots+u_n\vecom_n+yi_n \mid u_i\in[\alpha_i,\beta_i]\:\textrm{for}\:i=1,\ldots,n\Big\}
\end{align*}
 can be placed in such a way that, as $y\to0$, $\B$ is completely contained inside $W^{-1}\big(\R^n\times[B,\infty)\big)$. Thus the projection of $\B$ to $\Gamma\setminus\H^{n+1}$ stays inside the cuspidal region $\F_1(B)$ (cf.\ \eqref{cuspreg}). Since $B>B_0$ it follows that $\B$ is far from being equidistributed on $\Gamma\setminus\H^{n+1}$.
\end{remark}

\subsubsection*{Acknowledgement} I am grateful to Andreas Str\"ombergsson for many helpful and inspiring discussions on this work.

\end{document}